\documentclass[aop,MSNbibl,seceqn,dvips]{arximspdf}
\usepackage{graphicx}

\doi{10.1214/13-AOP835} 
\volume{41}
\issue{6}
\pubyear{2013}
\firstpage{3910}
\lastpage{3972}

\makeatletter
\newcommand{\rrvert}{\vert}
\newcommand{\llvert}{\vert}

\newtheorem{theorem}{Theorem}[section]
\newtheorem{proposition}{Proposition}[section]
\newtheorem{lemma}{Lemma}[section]

\newproclaim{remark}{Remark}[section]

\renewcommand{\epsilon}{\varepsilon}

\newcommand{\R}{\mathbb{R}}
\newcommand{\N}{\mathbb{N}}
\newcommand{\Z}{\mathbb{Z}}
\newcommand{\PR}{\mathbb{P}}
\newcommand{\ES}{\mathbb{E}}

\makeatother

\begin{document}
\begin{frontmatter}

\title{Biased random walk in positive random conductances on $\mathbb{Z}^{\lowercase{d}}$}
\runtitle{\hspace*{-5pt}Biased random walk in positive random conductances on $\mathbb{Z}^{\lowercase{d}}$}

\begin{aug}
\author[A]{\fnms{Alexander} \snm{Fribergh}\corref{}\ead[label=e1]{alexander.fribergh@math.univ-toulouse.fr}}
\runauthor{A. Fribergh}
\affiliation{CNRS and Universit\'e de Toulouse}
\address[A]{Institut de Math\'ematiques\\
CNRS UMR 5219\\
31062 Toulouse cedex 09\\
France\\
\printead{e1}} 
\end{aug}

\received{\smonth{8} \syear{2011}}
\revised{\smonth{1} \syear{2013}}

%
\begin{abstract}
We study the biased random walk in positive random conductances
on~$\Z^d$. This walk is transient in the direction of the bias. Our
main result is that the random walk is ballistic if, and only if, the
conductances have finite mean. Moreover, in the sub-ballistic regime we
find the polynomial order of the distance moved by the particle. This
extends results obtained by Shen [\textit{Ann. Appl. Probab.}
\textbf{12} (2002) 477--510], who proved positivity of the speed in the
uniformly elliptic setting.
\end{abstract}

%

\end{frontmatter}

\section{Introduction}

One of the most fundamental questions in random walks in random media
is understanding the long-term behavior of the random walk. This topic
has been intensively studied, and we refer the reader to \cite
{Zeitouni} for a general survey of the field. An interesting feature of
random walks in random environments (RWRE) is that several models
exhibit anomalous behaviors. One of the main reasons for such behaviors
is trapping, a phenomenon observed by physicists long ago \cite{HB} and
which is a central topic in RWRE. The importance of trapping in several
physical models (including RWRE) motivated the introduction of the
Bouchaud trap model (BTM). This is an idealized model that received a
lot of mathematical attention. A~review of the main results can be
found in \cite{BC}, a survey which conjectures that the type of results
obtained in the BTM should extend to a wide variety of models,
including RWRE.

One very characteristic behavior associated to trapping is the
existence of a zero asymptotic speed for RWRE with directional
transience. In the last few years, several articles have analyzed such
models from a trapping perspective, such as \cite{ESZ} and \cite{ESZ1}
on $\Z$ and \cite{BFGH,BH} and \cite{H} on trees. The results
on the $d$-dimensional lattice (with $d\geq2$) are much more rare,
since RWRE on $\Z^d$ are harder to analyze. Among the most natural
examples of directionally transient RWRE in $\Z^d$ are biased random
walks in random conductances.
So far, mathematically, only two models of biased random walks in
$\Z^d$ have been studied from a trapping perspective: one is on a supercritical
percolation cluster and the other is in environments assumed to be uniformly
elliptic. Before further discussing trapping issues, we wish to mention that
biased random walks in random environment also raise many other
interesting questions, such as the Einstein relation which has led to many new
developments, see \cite{BHOZ,GMP} and \cite{LR}.


In the case of biased random walks on a percolation cluster in $\Z^d$,
it was shown in \cite{BGP} (for $d=2$) and in \cite{Sznitman} that the
walk is directionally transient and, more interestingly, there exists a
zero-speed regime. More recently in \cite{FH} a characterization of the
zero-speed regime has been achieved. Those results confirmed the
predictions of the physicists that trapping occurs in the model; see
\cite{Dhar} and \cite{DS}.

In the case of a biased random walk in random conductances which are
uniformly elliptic, it has been shown in \cite{Shen} that the walk is
directionally transient and has always positive speed and verifies an
annealed central limit theorem. These results are coherent with the
conjecture that, a directionally transient random walk in random
environment which is uniformly elliptic, should have positive speed;
see~\cite{Sznitman2}. Hence, trapping does not seem to appear under
uniform ellipticity conditions.

The results on those two models do not bring any understanding of the
behavior of the random walk in positive conductances that might be
arbitrarily close to zero. In such a model, we truly lose the uniform
elliptic assumption, as opposed to the biased random walk on the
percolation cluster, where the walk is still uniformly elliptic on the
graph where the walk is restricted.

Our purpose in this paper is to understand the ballistic-regime of a
biased random walk in positive i.i.d. conductances and how trapping
arises in such a model.

\section{Model}

We introduce $ {\mathbf P}[ \cdot]=P_*^{\otimes E(\Z^d)} $, where $P_*$ is
the law of a positive random variable $c_*\in(0,\infty)$. This measure
gives a random environment usually denoted $\omega$.

In order to define the random walk, we introduce a bias $\ell=\lambda
\vec\ell$ of strength $\lambda>0$ and a direction $\vec\ell$ which
is in the unit sphere with respect to the Euclidian metric of $\R^d$.
In an environment $\omega$, we consider the Markov chain of law
$P_x^{\omega}$ on $\Z^d$ with transition probabilities $p^{\omega
}(x,y)$ for $x,y\in\Z^d$ defined by:
\begin{longlist}[(2)]
\item[(1)] $X_0=x$, $P_x^{\omega}$-a.s.,
\item[(2)] $ {p^{\omega}(x,y) =\frac{c^{\omega}(x,y)}{\sum_{z \sim x}
c^{\omega}(x,z)}}$,
\end{longlist}
where $x\sim y$ means that $x$ and $y$ are adjacent in $\Z^d$, and also
we set
%
\begin{equation}
\label{defconduct} \mbox{for all $x\sim y \in\Z^d$}\qquad
c^{\omega}(x,y)=c_*^{\omega
}\bigl([x,y]\bigr)e^{(y+x)\cdot\ell}.
\end{equation}

This Markov chain is reversible with invariant measure given by
\[
\pi^{\omega}(x)=\sum_{y \sim x}
c^{\omega}(x,y).
\]

The random variable $c^{\omega}(x,y)$ is called the conductance between
$x$ and $y$ in the configuration $\omega$. This comes from the links
existing between reversible Markov chains and electrical networks. We
refer the reader to \cite{DoyleSnell} and \cite{LP} for a further
background on this relation, which we will use extensively. Moreover
for an edge $e=[x,y]\in E(\Z^d)$, we denote $c^{\omega}(e)= c^{\omega}(x,y)$.

Finally the annealed law of the biased random walk will be the
semi-direct product $\PR= {\mathbf P}[ \cdot] \times P_0^{\omega}[ \cdot]$.

In the case where $c_*\in( 1/K, K)$ for some $K<\infty$, the walk is
uniformly elliptic, and this model is the one previously studied
in \cite{Shen}.

\section{Results}
First, we prove that the walk is directionally transient.
%
\begin{proposition}\label{dirtrans}
We have
\[
\lim X_n\cdot\vec{\ell} = \infty,\qquad \PR\mbox{-a.s.}
\]
\end{proposition}

This proposition is a consequence of Proposition \ref{taufinite}.

Our main result is:
%
\begin{theorem}
\label{thetheorem}
For $d\geq2$, we have
\[
\lim\frac{X_n} n =v,\qquad \PR\mbox{-a.s.},
\]
where:
\begin{longlist}[(2)]
\item[(1)] if $E_*[c_*]<\infty$, then $v\cdot\vec{\ell} >0$;
\item[(2)] if $E_*[c_*]=\infty$, then $v=\vec0$.
\end{longlist}

Moreover, if $\lim\frac{\ln P_*[c_*>n]}{\ln n}=-\gamma$ with $\gamma
<1$, then
\[
\lim\frac{\ln X_n \cdot\vec{\ell}}{\ln n} =\gamma,\qquad \PR\mbox{-a.s.}
\]
\end{theorem}

This theorem follows from Propositions \ref{LLN1}, \ref{LB},
Lemma \ref{LB1} and Proposition~\ref{LB2}.

This result proves that trapping phenomena may occur in an elliptic
regime, that is, when all transition probabilities are positive.

\begin{figure}

\includegraphics{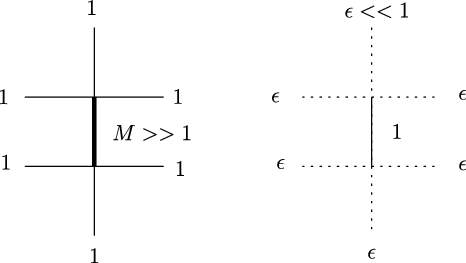}

\caption{The two main types of traps.}
\label{fig1}
\end{figure}

Let us rapidly discuss the different main ways the walk may be trapped
(see Figure \ref{fig1}):
\begin{longlist}[(2)]
\item[(1)] an edge with high conductance surrounded by normal conductances;
\item[(2)] a normal edge surrounded by very small conductances.
\end{longlist}

Let us discuss how the first type of traps function. Assume that we
have an edge $e$ of conductance $c_*(e)$ surrounded by edges of fixed
conductances, say 1. A simple computation shows that the walk will need
a time of the order of $c_*(e)$ to leave the endpoints of $e$. Hence,
if the expectation of $c_*$ is infinite, then the annealed exit time of
the $e$ is infinite. Heuristically, one edge is enough to trap the walk
strongly. This phenomenon is enough to explain the zero-speed regime.

At first glance it is surprising that, in Theorem \ref{thetheorem},
there is only a condition on the tail of $c_*$ at infinity. Indeed, if
the tail of $c_*$ at $0$ is sufficiently big, more precisely such that
$E[\min_{i=1,\ldots, 4d-2} 1/c_*^{(i)}]=\infty$ for $c_*^{(i)}$
i.i.d. chosen under $P_*$, then the second type of traps are such that
the annealed exit time of the central edge is infinite. This condition
does not appear in Theorem \ref{thetheorem} because the walk is
unlikely to reach such an edge. Indeed, it needs to cross an edge with
extremely low conductance to enter the trap. This type of trapping is
barely strong enough to create a zero-speed regime (see Remark~\ref
{tauinf}), nevertheless it forces us to be very careful in our
analysis of the model.

One may try to create traps similar to those encountered in the biased
random walk on the percolation cluster. In this model, if the bias is
high enough, a long dead-end in the direction of the drift can trap the
walk strongly enough to force zero-speed. In our context, we are not
allowed to use zero conductances, but we may use extremely low
conductances, forcing the walk to exit the dead-end at the same place
it entered. Nevertheless, this type of trap is very inefficient.
Indeed, most edges forming a dead-end have to verify $c_*(e)<\epsilon$
to be able to contain the walk for a long period, and this for any
fixed $\epsilon>0$. The probabilistic cost of creating such a trap is
way too high.

Hence, small conductances cannot create zero-speed, but high
conductances can. To conclude, we give an idealized version of the two
most important types of traps in this model,
\[
X_1=\operatorname{Geom}\bigl((1/c_*)\wedge1\bigr)
\]
or
\[
X_2=\cases{\operatorname{Geom}\bigl(c_*'\wedge1\bigr), &\quad with
probability $c_*'\wedge1$,
\cr
0, &\quad else,}
\]
where $c_*$ is chosen according to the law $P_*$, and $c_*'$ has the
law of
\[
\max_{i=1,\ldots, 4d-2} c_*^{(i)}
\]
where $c_*^{(i)}$ are i.i.d. chosen under the law $P_*$ and independent
of the geometric random variables. Intuitively, one should be able to
understand anything related to trapping with biased random walks in an
elliptic setting using those idealized traps.

Before moving on, let us say a word about the central limit theorem,
which is known to hold in the uniformly elliptic case (see \cite
{Shen}), at least in the annealed setting. Using the idealized model we
just described above, we are led to believe that if the two following
conditions hold:
\begin{longlist}[(2)]
\item[(1)] $P_*[1/c_*\geq x]\leq x^{-1/(4d-2)-\epsilon}$ for some
$\epsilon>0$,
\item[(2)] $P_*[c_*\geq x] \leq x^{-2-\epsilon}$ for some $\epsilon>0$,
\end{longlist}
then an annealed central theorem should hold. These conditions should
in some weak sense be necessary as indicated in Remark \ref{tcl}. It is
interesting to note that even though the tail of $1/c_*$ at $0$ does
not affect the law of large numbers, it is, however, important for the
central limit theorem. Although strongly related to the estimates made
in this article, the central limit theorem is not a directed
consequence of them, and due to the length of this paper, we choose not
to pursue this issue further.

Let us explain the organization of the paper. We begin by studying exit
probabilities of large boxes; the main point of Section \ref{sectBL}
is to prove Theorem \ref{BL} which is a property similar to Sznitman's
conditions $(T)$ and $(T)_{\gamma}$; see \cite{Sznitman2}. This
property is one of the key estimates for studying directionally
transient RWRE. It allows us to define regeneration times, similar to
the ones introduced in \cite{SZ}, and study them; this is done in
Section \ref{sectfirstregen}. The construction of regeneration times
in this model is complicated by the fact that we lack any type of
uniform ellipticity. This issue is explained in more details and dealt
with in Section \ref{sectopen}. The law of large numbers in the
positive speed regime is obtained in Section \ref{sectposspeed}. The
zero-speed regime is studied in Section \ref{sectzerospeed}. The next
section is devoted to notations which will be used throughout this paper.

\section{Notations}\label{sectnotation}

Let us denote by $(e_i)_{i=1,\ldots, d}$ an orthonormal basis of $\Z^d$
such that $e_1\cdot\vec\ell\geq e_2\cdot\vec\ell\geq\cdots\geq
e_d\cdot\vec\ell\geq0$; in particular we have $e_1\cdot\vec\ell
\geq1 /\sqrt d$. The set $\{\pm e_1,\ldots,\pm e_d\}$ will be denoted
by $\nu$. Moreover, we complete $f_1:=\vec{\ell}$ into an orthonormal
basis $(f_i)_{1\leq i \leq d}$ of $\R^d$.

We set
\[
\mathcal{H}^+(k)=\bigl\{x\in\Z^d; x\cdot\vec{\ell} > k\bigr\} \quad\mbox{and}\quad
\mathcal{H}^-(k)=\bigl\{x\in\Z^d; x\cdot\vec{\ell} \leq k\bigr
\}
\]
and
\[
\mathcal{H}^+_x=\mathcal{H}^+(x\cdot\vec{\ell}) \quad\mbox{and}\quad
\mathcal{H}^-_x=\mathcal{H}^-(x\cdot\vec{\ell}).
\]

For any graph $G$, let us introduce $d_{G}(x,y)$ the graph distance in
$G$ between $x$ and $y$. Define for $x\in G$ and $r>0$
\[
B_{G}(x,r)=\bigl\{y\in G; d_{G}(x,y)\leq r\bigr\}.
\]

Given a set $V$ of vertices of $\Z^d$, we denote by $\llvert
V\rrvert$ its
cardinality, by $E(V)=\{ [x,y] \in E(\Z^d); x,y\in V\}$ its edges and
\[
\partial V= \{x \notin V; y\in V\mbox{ and }x\sim y\}
\]
as well as
\[
\partial_E V= \bigl\{[x,y] \in E\bigl(\Z^d\bigr); x\in
V\mbox{ and }y \notin V \bigr\},
\]
its borders.

Given a set $E$ of edges of $\Z^d$, we denote $V(E)=\{x\in\Z^d;
x$ is an endpoint of $e\in E\}$ its vertices.

Denote for any $L,L' \geq1$
\[
B\bigl(L,L'\bigr)= \bigl\{z\in\Z^d; \llvert z \cdot
\vec{\ell}\rrvert\leq L\mbox{ and }\llvert z\cdot f_i\rrvert\leq
L' \mbox{ for $i \in[2,d]$} \bigr\}
\]
and
\[
\partial^+ B\bigl(L,L'\bigr)=\bigl\{z\in\partial B
\bigl(L,L'\bigr); z \cdot\vec{\ell} > L \bigr\}.
\]

We introduce the following notation. For any set of vertices $A$ of a
certain graph on which a random walk $X_n$ is defined, we denote
\[
T_A=\inf\{n\geq0; X_n\in A\},\qquad T_A^+=\inf
\{n\geq1; X_n\in A \}
\]
and
\[
T_A^{\mathrm{ex}}=\inf\{n\geq0; X_n\notin A\}.
\]
We will use a slight abuse of notation and write $x$ instead of $\{x\}$
when the set is a point $x$.

This allows us to define the hitting time of ``level'' $n$ by
\[
\Delta_n=T_{\mathcal{H}^+(n)}.
\]

Also, $\theta_n$ will denote the time shift by $n$ units of time.

Finally, we will use the notation $P[A,B]$ to designated $P[A\cap B]$,
when we are given a probability measure $P$ and two sets $A$ and $B$.
Similarly given a random variable $X$ and an event $A$, we may use the
notation $E[X,A]$ to designate $E[X{\mathbf1}{\{A\}}]$ when there is
no confusion possible.

In this paper constants are denoted by $c \in(0,\infty)$ or $C \in
(0,\infty)$ without emphasizing their dependence on $d$ and the law
$P_*$. Moreover the value of those constants may change from line to line.

\section{Exit probability of large boxes}
\label{sectBL}

Our first goal is to obtain estimates on the exit probabilities of
large boxes, which will allow us to prove directional transience and is
key to analyzing this model. In particular, we aim at showing:
%
\begin{theorem}
\label{BL}
For $\alpha>d+3$
\[
\PR[T_{\partial B(L,L^{\alpha})}\neq T_{ \partial^+ B(L,L^{\alpha})}]
\leq Ce^{-cL}.
\]
\end{theorem}

After this section $\alpha$ will be fixed, greater than $d+3$.

We will adapt a strategy of proof used in \cite{FH}. For the most part,
the technical details and notations are simpler in our context. We will
go over the parts of the proof which can be simplified, but we will
eventually refer the reader to \cite{FH} for the conclusion of the
proof which is exactly similar in both cases. The notations have been
chosen so that the reader can follow the needed proofs in \cite{FH} to
the word.

First, let us describe the strategy we will follow.

The fundamental idea is to partition the space into a good part where
the walk is well-behaved and a bad part consisting of small connected
components where we have very little control over the random walk.

The strategy is two-fold:
\begin{longlist}[(2)]
\item[(1)] We may study the behavior of the random walk at times where
it is in the good part of the space, in which it can easily be
controlled. We will refer to this object as the modified walk. We show
that the modified walk behaves nicely, that is, verifies Theorem \ref
{BL}. This is essentially achieved using a combination of spectral gap
estimates and the Carne--Varopoulos formula \cite{Carne}.
\item[(2)] We need to show that information on the exit probabilities
for this modified random walk allows us to derive interesting
statements on the actual random walk. This is a natural thing to
expect, since the bad parts of the environment are small.
\end{longlist}

A more detailed discussion of the strategy of proof can be found at the
beginning of Section 7 in \cite{FH}.

\subsection{Bad areas}\label{subsecbad}

We say that an edge $e$ is $K$-normal if $c_*(e)\in[1/K,K]$, where $K$
will be taken to be very large in the sequel. If an edge is not
$K$-normal, we will say it is $K$-abnormal which occurs with
arbitrarily small probability $\epsilon(K):=P_*[c_*\notin[1/K,K]]$,
since $c_*\in(0,\infty)$.

In relation to this, we will say that a vertex $x$ is $K$-open if for
all $y\sim x$ the edge $[x,y]$ is $K$-normal. If a vertex is not
$K$-open, we will say it is $K$-closed. By taking $K$ large enough, the
probability that a vertex is $K$-open goes to 1. Finally a vertex $x\in
\Z^d$ is $K$-good if there exists an infinite directed $K$-open path
starting at $x$; that is, we have $\{x=x_0,x_1,x_2,x_3,\ldots\}$ with
$x_0=x$ such that for all $i\geq0$:
\begin{longlist}[(2)]
\item[(1)] we have $x_{2i+1} -x_{2i} = e_1$ and $x_{2i+2}-x_{2i+1} \in
\{e_1,\ldots,e_d\}$;
\item[(2)] $x_i$ is $K$-open.
\end{longlist}
If a vertex is not $K$-good, it is said to be $K$-bad. The key property
of a good point will be that there exists a open path $(x_i)_{i\geq0}$
such that $x_0\cdot\vec{\ell}<x_1\cdot\vec{\ell}\leq x_2 \cdot
\vec{\ell
}<x_3\cdot\vec{\ell}\leq\cdots\,$, and also this path verifies
$(x_i-x_0)\cdot\vec{\ell} \geq c(d) i$.

To ease the notation, which will be used repeatedly throughout the
article, we will not always mention the $K$-dependences.

The first of these results is stated in terms of the width of a subset
$A \subseteq\Z^d$, which we define to be
\[
W(A)=\max_{1 \leq i \leq d} \Bigl(\max_{y\in A} y\cdot
e_i - \min_{y\in
A} y\cdot e_i \Bigr).
\]

Let us denote $\mathrm{BAD}_K(x)$ the connected component of $K$-bad
vertices containing~$x$, in case $x$ is good then $\mathrm
{BAD}_K(x)=\varnothing$.
%
\begin{lemma}
\label{BLsizeclosedbox}
There exists $K_0$ such that, for any $K\geq K_0$ and for any $x\in\Z
^d$, we have that the cluster $\mathrm{BAD}_K(x)$ is finite ${\mathbf
P}_p[ \cdot]$-a.s. and
\[
{\mathbf P}_p\bigl[ W\bigl(\mathrm{BAD}_K(x)\bigr) \geq n
\bigr] \leq C\exp\bigl(-\xi_1(K)n\bigr),
\]
where $\xi_1(K)\to\infty$ as $K$ tends to infinity.
\end{lemma}

\begin{pf}
We call two vertices 2-connected if $\| u-v\|
_1=2$, so that we may define $\mathrm{BAD}_K^e(x)$ as the
2-connected component of bad vertices containing $x$. Any element of
$\mathrm{BAD}(x)$ is a neighbor of $\mathrm{BAD}_K^e(x)$ so that
$W(\mathrm{BAD}_K(x))\leq W(\mathrm{BAD}_K^e(x))+2$.

Consider now the site percolation model on the even lattice $\Z
^d_{\mathrm{even}}=\{v\in\Z^d, \| v\|t_1$
is even$\}$ where $y$ is
even-open if and only if, in the original model, the vertices $y$,
$y+e_1$ and $y+e_1+e_i$ ($i\leq d$) are open. An edge $[y,z]$ is
even-open if, and only if, $y$ and $z$ are even-open. This last model
is a 4-dependent oriented bond percolation model, which has a measure
that we denote by~$P_{p,\mathrm{orient}}$.

Fix $p'$ close to 1. For $K$ large enough, the probability that a
vertex is $K$-open can be made arbitrarily close to $1$. This means
that we can make the probability of an edge being even-open arbitrarily
close to $1$, so, by Theorem 0.0 in \cite{LSS}, the law $P_{p,
\mathrm{orient}}$ dominates an i.i.d. bond percolation with parameter $p'$.

Let us introduce the outer edge-boundary $\partial_E \operatorname{BAD}_K^e(x)$
of $\mathrm{BAD}_K^e(x)$ in the graph $\Z_{\mathrm{even}}^d$ with the
following notion of adjacency: $x$ and $y$ are adjacent if $x-y\in\{
\pm(e_i \pm e_j)$ with $i\neq j$ and $i,j\leq d\}$.

We describe how to do the proof for $d=2$. We will assign an arrow to
any edge $[y,z]\in\partial_E \operatorname{BAD}_K^e(x)$, and assuming $y\in
\mathrm{BAD}_K^e(x)$ and $z\notin\mathrm{BAD}_K^e(x)$, we set:
\begin{longlist}[(4)]
\item[(1)] $\swarrow$, if $y-x=e_1-e_2$;
\item[(2)] $\nwarrow$, if $y-x=e_1+e_2$;
\item[(3)] $\nearrow$, if $y-x=-e_1+e_2$;
\item[(4)] $\searrow$, if $y-x=-e_1-e_2$.
\end{longlist}

This boundary is represented dually in Figure \ref{fig2}. By an argument similar
to that of Durrett \cite{Durrett}, page 1026, we see that $n_{\nearrow}
+n_{\searrow} =n_{\swarrow}+n_{\nwarrow}$, where $n_{\nearrow}$, for
example, is the number of edges labeled $\nearrow$ in $\partial_E
\operatorname{BAD}_K^e(x)$.

\begin{figure}

\includegraphics{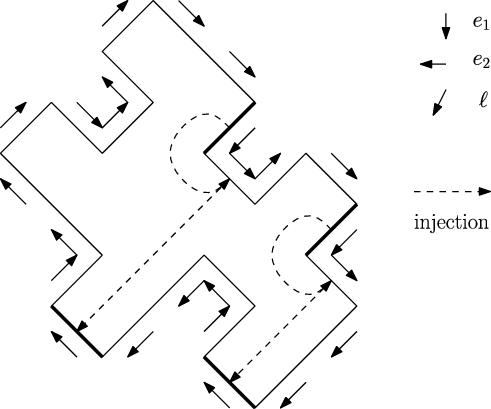}

\caption{The outer edge-boundary $\partial_E \operatorname{BAD}_K^e(x)$ of
$\mathrm{BAD}_K^e(x)$ represented dually on the even lattice when $d=2$.}
\label{fig2}
\end{figure}

\begin{longlist}[(2)]
\item[(1)] Any $\nwarrow$ edge of $\partial_E \operatorname{BAD}_K^e(x)$ has
one endpoint, say $y$, which is bad, and one, $y+e_1+e_2$, which is
good. This implies that $y$ is even-closed.
Now, we will argue that if $n_{\nwarrow} \geq n_{\swarrow}/2$, then at
least one sixth of the edges of $\partial_E \operatorname{BAD}_K^e(x)$ are
$\nwarrow$ edges, and hence even-closed. Indeed, since $n_{\nearrow}
+n_{\searrow} =n_{\swarrow}+n_{\nwarrow}$, we\vadjust{\goodbreak} know that $n_{\swarrow
}+n_{\nwarrow} =\llvert\partial_E \operatorname{BAD}_K^e(x)\rrvert
/2$ and also, using
our hypothesis, $n_{\swarrow}+n_{\nwarrow} \leq3 n_{\nwarrow}$. Those
two inequalities imply that $n_{\nwarrow}\geq\llvert\partial_E \operatorname{BAD}_K^e(x)\rrvert/6$. Hence, at least one sixth of the
edges of $\partial_E \operatorname{BAD}_K^e(x)$ are even-closed.
\item[(2)] Otherwise, let us assume that $n_{\swarrow}\geq
2n_{\nwarrow
}$. We may notice any $\swarrow$ edge followed (in the sense of the
arrows) by an $\searrow$ edge can be mapped in an injective manner to
an $\nwarrow$ edge. This injection is indicated in Figure~\ref{fig2} (by
considering the bold edges). This injection, with $n_{\swarrow}\geq
2n_{\nwarrow}$, means that at least half of the $\swarrow$ edges are
not followed by an $\searrow$ edge. So, using that $n_{\nearrow}
+n_{\searrow} =n_{\swarrow}+n_{\nwarrow}$, we see that at least one
sixth of the edges of $\partial_E \operatorname{BAD}_K^e(x)$ are $\swarrow$
edges that are not followed by an $\searrow$ edge. Consider such an
$\swarrow$ edge, and we can see that the endpoint $y$ of the $\swarrow$
edge which is inside $\mathrm{BAD}_K^e(x)$ verifies that $y+2e_1$ is not
in $\mathrm{BAD}_K^e(x)$. Hence for any such $\swarrow$, there is one
endpoint $y$ which is bad and such that $y+2e_1$ is good, and hence $y$
is even-closed. Once again at least one sixth of the edges of $\partial
_E \operatorname{BAD}_K^e(x)$ are even-closed.
\end{longlist}

This means that at least one sixth of the edges of $\partial_E \operatorname{BAD}_K^e(x)$ are even-closed. The outer boundary is a minimal cutset,
as described in \cite{Babson}. The number of such boundaries of size
$n$ is bounded (by Corollary 9 in \cite{Babson}) by $\exp(Cn)$. Hence,
if $p'$ is close enough to $1$, a counting argument allows us to obtain
the desired exponential tail for $W( \mathrm{BAD}_K^e(x))$ under
$P_{p'}$, and hence under $P_{p,\mathrm{orient}}$ (since the latter is
dominated by the former).

For general dimensions, we note that there exists $i_0\in[2,d]$ such
that a proportion at least $1/d$ of the edges of $\partial_E \operatorname{BAD}_K^e(x)$ are edges of the form $[y,y\pm e_{i_0}]$ and $[y,y\pm
e_{1}]$ for some $y\in\Z^d$. We may then apply the previous reasoning
in every plane $y+\Z e_1+\Z e_{i_0}$ containing edges of $\partial_E \operatorname{BAD}_K^e(x)$ to show that at least a proportion $1/6$ of those
edges are even-closed. Thus, at least a proportion $1/6d$ of the edges
of $\partial_E \operatorname{BAD}_K^e(x)$ verify the same property. By
repeating the same counting argument as in the previous paragraph we
can infer the lemma.
\end{pf}

Let us define $\mathrm{BAD}_K=\bigcup_{x\in\Z} \mathrm{BAD}_K(x)$ which
is a
union of finite sets. Also we set $\mathrm{GOOD}_K=\Z^d \setminus
\mathrm{BAD}_K$. We may notice that
%
\begin{equation}
\label{disjoint} \mbox{for any $x\in\mbox{BAD}_K$}\qquad\partial
\mathrm{BAD}_K(x) \subset\mathrm{GOOD}_K,
\end{equation}
since $\mathrm{BAD}_K(x)$ is a connected component of bad points.




In the sequel $K$ will always be large enough so that $\mathrm{BAD}_K(x)$
are finite for any $x\in\Z^d$.

\subsection{A graph transformation to seal off big traps}

Given a certain configuration $\omega$, we construct a graph $\omega
_{K}$ (with conductances) such that the random walk induced by
recording only the steps of the original random walk in $\omega$
outside of large traps has the same law as the random walk in $\omega_{K}$.

We denote $\omega_{K} $ the graph obtained from $\omega$ by the
following transformation. The vertices of $\omega_{K} $ are the
vertices of $\mathrm{GOOD}_K$, and the edges of $\omega_{K} $ are:
\begin{longlist}[(2)]
\item[(1)] $\{[x,y], x, y \in\mathrm{GOOD}_K$, with $x\notin
\partial\mathrm{BAD}_K$ or $y\notin\partial\mathrm{BAD}_K\}$ and
have conductance $c^{\omega_{K}}([x,y]):= c^{\omega}([x,y])$,
\item[(2)] $ \{[x,y], x,y \in\partial\mathrm{BAD}_K\}$ (including
loops) which have conductance
\begin{eqnarray*}
c^{\omega_{K}}\bigl([x,y]\bigr)&:=& \pi^{\omega}(x)P^{\omega}_x
\bigl[X_1\in\mathrm{BAD}_K\cup\partial
\mathrm{BAD}_K, T_y^+=T_{\partial\mathrm{BAD}_K}^+\bigr]
\\
&=& \pi^{\omega}(y)P^{\omega}_y\bigl[X_1
\in\mathrm{BAD}_K\cup\partial\mathrm{BAD}_K, T_x^+=T_{\partial\mathrm{BAD}_K}^+
\bigr],
\end{eqnarray*}
\end{longlist}
the last equality being a consequence of reversibility and ensures
symmetry for the conductances.

We call the walk induced by $X_n$ on $ \mathrm{GOOD}_K$, the walk $Y_n$
defined to be $Y_n=X_{\rho_n}$ where
\[
\rho_0=T_{ \mathrm{GOOD}_K} \quad\mbox{and}\quad \rho_{i+1}=T^+_{
\mathrm{GOOD}_K}
\circ\theta_{\rho_i},
\]
where we recall that $\circ\,\theta_i$ stands for the time shift by $i$
units of time. This means that $\rho_{i+1}$ is the first time
(strictly) after $\rho_i$ when the walk is in $\mathrm{GOOD}_K$. This
means that $\rho_i$ are the successive times when $X_n$ is in $\mathrm{GOOD}_K$.

From \cite{FH}, Proposition 7.2, we have the two following properties.
%
\begin{proposition}
\label{BLpropwalk}
The reversible walk defined by the conductances $\omega_{K}$ verifies
the two following properties:
\begin{longlist}[(2)]
\item[(1)] It is reversible with respect to $\pi^{\omega}(\cdot)$.
\item[(2)] If started at $x\in\omega_{K}$, it has the same law as the
walk induced by $X_n$ on $\mathrm{GOOD}_K$ started at $x$.
\end{longlist}
\end{proposition}

Furthermore, we have:

\begin{lemma}
\label{condinegalite}
For $x,y \in\mathrm{GOOD}_K$ which are nearest neighbors in $\Z^d$, we
have $c^{\omega_{K}}([x,y])\geq c^{\omega}([x,y])$.
\end{lemma}

Hence, we may notice the following:
%
\begin{remark}
\label{fakeUE}
We may notice that for any $x \in\mathrm{GOOD}_K$, we have
\[
\frac{c}Ke^{-2\lambda x\cdot\vec{\ell}}\leq\pi^{\omega_K}(x)\leq
CKe^{-2\lambda x\cdot\vec{\ell}}
\]
and for any $y\in\mathrm{GOOD}_K$ adjacent, in $\Z^d$, to $x$
\[
\frac{c}Ke^{-2\lambda x\cdot\vec{\ell}} \leq c^{\omega
_K}\bigl([x,y]\bigr)\leq
CKe^{-2\lambda x\cdot\vec{\ell}}.
\]
\end{remark}

\subsection{\texorpdfstring{Spectral gap estimate in $\omega_K$}{Spectral gap estimate in omega K}}
\label{spectralgap}

The following arguments are heavily inspired from \cite{Sznitman} and
use spectral gap estimates. After showing that the spectral gap in
$\omega_K$, we can deduce that the walk is likely to exit the box
quickly in $\omega_K$. Finally, we need to argue that exiting the box
quickly, we should exit it in the direction of the drift. This allows
us to obtain Theorem \ref{BL}, once we have argued that the exit
probabilities in $\omega$ and $\omega_K$ are strongly related. This
paper only contains the first step of this reasoning, the following
ones being treated in \cite{FH}.

For technical reasons, we introduce the notation
\[
\tilde{B}\bigl(L,L^{\alpha}\bigr)=\bigl\{x\in\Z^d, -L\leq x
\cdot\vec{\ell} \leq2 L \mbox{ and } \llvert x\cdot f_i\rrvert\leq
L^{\alpha} \mbox{ for $i\geq2$}\bigr\}.
\]

Let us introduce the principal Dirichlet eigenvalue of $I-P^{\omega
_K}$, in\break \mbox{$\tilde{B}(L,L^{\alpha})\cap\omega_K$}
%
\begin{equation}
\label{propdirichlet1} \Lambda_{\omega_K}\bigl(
\tilde{B}\bigl(L,L^{\alpha}\bigr)\bigr)= \cases{\inf\bigl\{
\mathcal{E}_{\mathrm{GOOD}_K}(f,f), f_{\mid(\tilde
{B}(L,L^{\alpha})\cap\omega_K)^c}=0,\vspace*{1pt}\cr
\qquad\hspace*{76.5pt} \| f
\|_{L^2(\pi
(\omega
_K))}=1\bigr\},
\vspace*{1pt}\cr
\hspace*{16.2pt}\qquad\mbox{when $\tilde{B}
\bigl(L,L^{\alpha}\bigr)\cap\omega_K \neq\varnothing$,}
\vspace*{1pt}\cr
\infty, \qquad\mbox{by convention when } \tilde{B}\bigl(L,L^{\alpha}\bigr)
\cap
\omega_K=\varnothing,}\hspace*{-35pt}
\end{equation}
where the Dirichlet form is defined for $f,g\in L^2(\pi^{\omega_K})$ by
\begin{eqnarray*}
\mathcal{E}_{\mathrm{GOOD}_K}(f,g)&=&\bigl(f,\bigl(I-P^{\omega_K}\bigr)g
\bigr)_{\pi
^{\omega
_K}}
\\
&=&\frac12 \sum_{x,y\ \mathrm{neighbors}\ \mathrm{in}\ \omega_K} \bigl(f(y)-f(x)\bigr)
\bigl(g(y)-g(x)\bigr) c^{\omega_K}\bigl([x,y]\bigr).
\end{eqnarray*}

We have:
%
\begin{lemma}
\label{BLDirichlet}
For $\omega$ such that $\tilde{B}(L,L^{\alpha})\cap\omega_K \neq
\varnothing$, we have
\[
\Lambda_{\omega_K}\bigl(\tilde{B}\bigl( L,L^{\alpha}\bigr)\bigr) \geq
c(K) L^{-(d+1)}.
\]
\end{lemma}

\begin{pf}
From any vertex $x\in\omega_K$, which is a good point, there exists a
directed open path $x=p_x(0), p_x(1) ,\ldots, p_{x}(l_x)$ in $\omega_K$
(which are neighbors in $\Z^d$) such that $p_x(i)\in\tilde
{B}(L,L^{\alpha})$ for $i<l_x$ and $p_{x}(l_x)\notin\tilde
{B}(L,L^{\alpha})$. This allows us to say that
%
\begin{equation}
\label{BLbacktrack} \max_{i\leq l_x}
\frac{ \pi_{\omega_K}(x) }{c_{\omega_K}([p_x(i+1),
p_x(i)])}\leq C \max_{i\leq l_x} \frac{ \pi_{\omega}(x) }{\pi
_{\omega
}(p_x(i))} \leq C,
\end{equation}
where we used Lemma \ref{condinegalite} and Remark \ref{fakeUE}.
Moreover $l_x\leq CL$.

We use a classical argument of Saloff-Coste \cite{Saloff-Coste}, and we
write for\break $\| f\|_{L^2(\pi
^{\omega_K})}=1$
\begin{eqnarray*}
1&=&\sum_x f^2(x)\pi_{\omega_K}(x)
= \sum_x \biggl[\sum_i
f\bigl(p_x(i+1)\bigr)-f\bigl(p_x(i)\bigr)
\biggr]^2\pi_{\omega_K}(x)
\\
&\leq&\sum_x l_x \biggl[\sum
_i \bigl(f\bigl(p_x(i+1)\bigr)-f
\bigl(p_x(i)\bigr)\bigr)^2 \biggr]\pi_{\omega_K}(x).
\end{eqnarray*}

Now by (\ref{BLbacktrack}), we obtain
\[
1\leq C \sum_{x,y\ \mathrm{neighbors}\ \mathrm{in}\ \omega_K} \bigl(f(z)-f(y)
\bigr)^2 c_{\omega
_K}\bigl([x,y]\bigr) \times\max
_{b\in E(\Z^d)} \sum_{x\in\omega_K\cap
\tilde
{B}(L,L^{\alpha}),b\in p_x} l_x,
\]
where $b\in p_x$ means that $b=[p_x(i),p_x(i+1)]$ for some $i$. Using that:
\begin{longlist}[(2)]
\item[(1)] $l_x\leq CL $ for any $x\in\omega_K$,
\item[(2)] $b=[x,y]\in\omega_K$ can only be crossed by paths
``$p_z$'' if $b\in E(\Z^d)$ and $z\in B_{\Z^d}(x,CL)$,
\end{longlist}
we have
\[
\max_b \sum_{x\in\tilde{B}(n,n^{\alpha}),b\in p_x}
l_x \leq CL^{d+1}
\]
and
\[
1\leq C L^{d+1} \sum_{x,y\ \mathrm{neighbors}\ \mathrm{in}\ \omega_K} \bigl(f(z)-f(y)
\bigr)^2 c_{\omega_K}\bigl([x,y]\bigr).
\]

Since this is true for every $f$ such that $f_{\mid(\tilde
{B}(L,L^{\alpha})\cap\omega_K)^c}=0$ and\break $\|
f\|_{L^2(\pi
(\omega
_K))}=1$, we can use (\ref{propdirichlet1}) to see
\[
\Lambda_{\omega_K}\bigl(\tilde{B}\bigl(n,n^{\alpha}\bigr)\bigr) \geq
c L^{-(d+1)}.
\]
\upqed\end{pf}

We explained how to obtain Theorem \ref{BL} at the beginning of
Section \ref{spectralgap}. The proof of Theorem \ref{BL} is almost
completely similar to the end of the proof of Theorem 1.4 in \cite{FH}.
The reader may read Sections 7.4, 7.5 and 7.6 of \cite{FH} for the
complete details.

To ease this task, the notations have been chosen so that only two
minor changes have to be made: in our case $K_{\infty}=\Z^d$ and
$\mathcal{I}=\Omega$.

The proof in \cite{FH} uses some reference to previous results, and for
the reader's convenience we specify the correspondence. The following
results in \cite{FH}: Lemma~7.5, Proposition 7.2, Lemma 7.9 and the
second part of Lemma 7.6 correspond, respectively, to Lemma \ref
{BLsizeclosedbox}, Proposition \ref{BLpropwalk}, Lemma \ref
{BLDirichlet} and~(\ref{disjoint}).

A final remark is that any inequality needed on $\pi^{\omega_K}$ can be
found in Remark~\ref{fakeUE}.

\section{Construction of $K$-open ladder points}
\label{sectopen}

A classical tool for analyzing directional transient RWRE is to use a
regeneration structure \cite{SZ}. We call ladder-point a new maximum of
the random walk in the direction $\vec{\ell}$. The standard way of
constructing regeneration times is to consider successive ladder points
and argue that there is a positive probability of never backtracking
again. Such a ladder point creates a separation between the past and
the future of the random walk leading to interesting independence
properties. We call this point a regeneration time.

A major issue in our case is that we do not have any type of uniform
ellipticity. Ladder points are conditioned parts of the environment
and, at least intuitively, the edge that led us to a ladder point
should have uncharacteristically high conductance. Those high
conductances (without uniform ellipticity) may strongly hinder the walk
from never backtracking and creating a regeneration time.
In order to adapt the classical construction we need, in some sense, to
show that the environment seen from the particle at a ladder-point is
relatively normal. To address this problem, we will prove that we
encounter open ladder-points and find tail estimates on the location of
the first open ladder-point.

We define the following random variable:
\begin{eqnarray*}
\mathcal{M}^{(K)}&=&\inf\{i\geq0, X_i \mbox{ is $K$-open
and for }j< i-2, X_j\cdot\vec{\ell} < X_{i-2} \cdot\vec{
\ell}
\\
&&\hspace*{109pt} \mbox{ and }X_i=X_{i-1}+e_1=X_{i-2}+2e_1
\}\\
&\leq&\infty.
\end{eqnarray*}
This means at that point we have reached a new maximum of the
trajectory (in the direction $\vec{\ell}$), made two steps in the
direction $e_1$ and reaching and open site. This definition is just
slightly different than the first open ladder point; it is only for
technical reasons that we consider this definition.

Our goal for this section is to obtain properties on this random
variable, namely that it is finite and has arbitrarily high polynomial moments.

The dependence on $K$ will be dropped outside of major statements and
definitions.

\subsection{Preparatory lemmas}

We need three preparatory lemmas before turning to the study of
$\mathcal{M}^{(K)}$. For this, we introduce the inner positive boundary
of $ B(n,n^{\alpha})$
\[
\partial^+_i B\bigl(n,n^{\alpha}\bigr)=\bigl\{x\in B
\bigl(n,n^{\alpha}\bigr)\mbox{, where }x\sim y \mbox{ with } y\in
\partial^+
B\bigl(n,n^{\alpha}\bigr)\bigr\}
\]
and the event
\[
A(n)=\{T_{\partial B(n,n^{\alpha})}\geq T_{ \partial^+_i
B(n,n^{\alpha
})}\}.
\]

It follows from Theorem \ref{BL} that:
%
\begin{lemma}
\label{an}
We have
\[
\PR\bigl[A(n)^c\bigr] \leq Ce^{-cn}.
\]
\end{lemma}

We say that a vertex $x\in B(n,n^{\alpha})$ is $K$-$n$-closed if there
exists a nearest neighbor $y\notin\mathcal{H}^+(n)$ of $x$ such that
$c^*([x,y])\notin[1/K,K]$.

Let us denote $\overline{K}_x(n)$ the $K$-$n$-closed connected
component of $x$. This allows us to introduce the event
%
\begin{equation}
\label{defcompo} B(n)=\bigl\{\mbox{for all $x\in
\partial^+_i B\bigl(n,n^{\alpha}\bigr)$, we have }\bigl\llvert
\overline{K}_x(n)\bigr\rrvert\leq\ln n\bigr\}.
\end{equation}

It is convenient to set $\overline{K}_x(n)=\{x\}$ when $\overline
{K}_x(n)$ is empty.

\begin{lemma}
\label{tailb}
For any $M<\infty$, we can find $K_0$ such that for any $K\geq K_0$
\[
{\mathbf P}\bigl[B(n)^c\bigr]\leq Cn^{-M}.
\]
\end{lemma}
\begin{pf}
Obviously, for any $x\in\partial^+ B(n,n^{\alpha})$
\[
\overline{K}_x(n)\subset\mathrm{CLOSED}_K(x),
\]
where $\mathrm{CLOSED}_K(x)$ is the connected component of $K$-closed
point containing $x$.

Using Lemma 5.1 in \cite{Kesten}, we may notice that there are at most
an exponential number of lattice animals. Hence, for any $x\in\partial
^+ B(n,n^{\alpha})$
\[
{\mathbf P}\bigl[\bigl\llvert\mathrm{CLOSED}_K(x)\bigr\rrvert\geq\ln n
\bigr] \leq\sum_{k\geq\ln n}C \bigl(C_1
\epsilon(K)\bigr)^{k}\leq Cn^{-\xi_2(K)},
\]
where $\xi_2(K)$ tends to infinity $K$ goes to infinity. The right-hand
side can be made lower than $n^{-M}$ for any $M$ by choosing $K$ large
enough. The result follows from a union bound.
\end{pf}

Next, we show a result which, in particular, implies that, if we are on
the ``positive boundary'' (i.e., in the direction of the bias
and where the largest conductances are) of a finite connected subset of
$\Z^d$ surrounded by normal edges, then we have some lower bound on the
probability to exit that set through this positive side.

\begin{lemma}
\label{exitnet}
Take $G\neq\varnothing$ to be a finite connected subset of $\Z^d$.
Assume that each edge $e$ of $\Z^d$ is assigned a positive conductance
$c(e)$ and that there exist $c_1>0$, $x\in\partial G$ and $y\in G$
such that $x\sim y$ and $c([x,y])\geq c_1 c(e)$ for any $e\in\partial
_E G$. We have
\[
P_y[T_x\leq T_{\partial G}] \geq\frac{c_1}{4d}
\llvert G\rrvert^{-1},
\]
where $P_y$ is the law of the random walk in $\Z^d$ started at $y$
arising from the conductances $(c(e))_{e\in E(\Z^d)}$.
\end{lemma}

\begin{pf} We will be using comparisons to electrical networks, and we
refer the reader to Chapter 2 of \cite{LP} for further background on
this topic.

Let us first notice that a walk started at $y\in G$ will reach
$\partial G$ before $\Z^d\setminus(G\cup\partial G)$, so this lemma
is actually a result on a finite graph $G\cup\partial G$.

To simplify the proof, we will consider the graph $\tilde{G}$ where all
edges emanating from $x$ that are not $[x,y]$ will be assigned
conductance $0$, which corresponds to reflecting the walk on those
edges. It is plain to see that
%
\begin{equation}
\label{hhh} P_y[T_x\leq T_{\partial G}] \geq
P_y^{\tilde G \cup\partial\tilde
G}[T_x\leq T_{\partial\tilde{G}}],
\end{equation}
where $P_y^{\tilde G \cup\partial\tilde G}$ is the law of the random
walk started at $y$ in the conductances of the graph $\tilde G \cup
\partial\tilde G$.

Hence, it is enough to prove our statement in the finite graph $\tilde
G \cup\partial\tilde G$. We may see that
%
\begin{equation}
\label{dammit} P_y^{\tilde G \cup\partial\tilde G}[T_x\leq
T_{\partial\tilde{G}}] =u(y),
\end{equation}
where $u(\cdot)$ is the voltage function verifying $u(x)=1$ and
$u(z)=0$ for $z\in\partial\tilde{G}\setminus\{x\}$. Let us denote
$i(\cdot)$ the associate intensity. Since $y$ is the only vertex
adjacent to $x$ in $\tilde G \cup\partial\tilde G$, we know that the
current flowing into the circuit at $x$ passes through the edges
$[x,y]$, so
\[
\frac1 {R^{\tilde G \cup\partial\tilde G}(x,\partial\tilde{G}
\setminus\{x\}) }=i
\bigl([x,y]\bigr),
\]
where $R^{\tilde G \cup\partial\tilde G}(x,\partial\tilde{G}
\setminus\{x\}) $ is the effective conductance between $x$ and
$\partial\tilde{G} \setminus\{x\}$ in $\tilde G \cup\partial\tilde
G$. By Ohm's law, we may deduce that
\[
u(x)-u(y)=r^{\tilde G \cup\partial\tilde G}\bigl([x,y]\bigr) i\bigl
([x,y]\bigr)=
\frac
{r^{\tilde G \cup\partial\tilde G}([x,y]) }{R^{\tilde G \cup\partial
\tilde G}(x,\partial\tilde{G} \setminus\{x\}) }.
\]

Now, since $x$ is the only vertex adjacent to $y$, we can see by an
electrical network reduction of resistances in series that $R^{\tilde G
\cup\partial\tilde G}(x,\partial\tilde{G} \setminus\{x\}) =
r^{\tilde G \cup\partial\tilde G}([x,y]) +R^{\tilde G \cup\partial
\tilde G \setminus\{x\}}(y,\partial\tilde{G} \setminus\{x\})$. This
means that
\begin{eqnarray*}
u(y) &=& \frac{R^{\tilde G \cup\partial\tilde G \setminus\{x\}
}(y,\partial\tilde{G} \setminus\{x\})}{R^{\tilde G \cup\partial
\tilde G \setminus\{x\}}(y,\partial\tilde{G} \setminus\{x\})
+r^{\tilde G \cup\partial\tilde G} ([x,y])}\\
&=& \biggl(1 + \frac
{r^{\tilde G \cup\partial\tilde G}([x,y])} {R^{\tilde G \cup\partial
\tilde G \setminus\{x\}}(y,\partial\tilde{G} \setminus\{x\})} \biggr)^{-1}.
\end{eqnarray*}

We recall that Rayleigh's monotonocity principle (see \cite{LP}) states
that increasing any value of the conductance of an edge (in particular
merging two vertices) increases any effective conductances. We consider
the graph $\tilde G \cup\partial\tilde G \setminus\{x\}$ and
collapse all vertices $\tilde{G} \cup\partial\tilde{G} \setminus\{
x,y\}$ into one vertex $\delta$. This increases all effective
conductances. In this new graph, $y$ is connected to $\delta$ by at
most $\llvert\tilde{G} \cup\partial\tilde{G} \rrvert$
edges of conductances
at least $(1/c_1) c^{\tilde G \cup\partial\tilde G}([x,y]) $, by our
assumptions on the graph. By network reduction of conductances in
parallel, this means
\[
R^{\tilde G \cup\partial\tilde G\setminus\{x\}}\bigl(y,\partial\tilde
{G} \setminus\{x\}\bigr) \geq
\frac{c_1} {\llvert\tilde{G} \cup
\partial\tilde{G} \rrvert} r^{\tilde G \cup\partial\tilde G}\bigl
([x,y]\bigr).
\]

The two last equations imply, with (\ref{dammit}), that
\[
P_y^{\tilde G \cup\partial\tilde G}[T_x\leq T_{\partial\tilde{G}}] \geq
\frac{c_1}{4d} \llvert G\rrvert^{-1},
\]
and with (\ref{hhh}) this concludes the proof.
\end{pf}

\subsection{Successive attempts to find an open ladder point}

Let us denote $\mathcal{B}_n:=B(n, n^{\alpha} )$. We will show that an
open ladder point can occur shortly after we exit a box $\mathcal{B}_{n}$.

Let us denote for $k\leq n$ the events
%
\begin{equation}
\label{defR} R^{(K)}(nk)=\bigl\{\mathcal{M}^{(K)}
> T_{\partial\mathcal{B}_{nk}}+2\bigr\}.
\end{equation}

We have:
%
\begin{lemma}
\label{recursion}
For any $\epsilon_1>0$ and $M<\infty$, we can find $K_0=K_0(\epsilon
_1,M)$ large enough such that the following holds: for any $K\geq K_0$
and any $k\in[2,n]$,
\[
\PR\bigl[R^{(K)}(kn)\bigr] \leq\bigl(1-cn^{-\epsilon_1}\bigr)\PR
\bigl[R^{(K)}\bigl((k-1)n\bigr)\bigr]+Cn^{-M},
\]
where the constants depend on $K$.
\end{lemma}

Essentially, our goal is to construct an open ladder-point. The idea of
the proof roughly goes as follows. If we exited $\mathcal{B}_{(k-1)n}$
without encountering an open ladder point, then:
\begin{longlist}[(3)]
\item[(1)] We are likely to exit $\mathcal{B}_{kn}$ through the
positive boundary.
\item[(2)] At this point, we look at the $K$-$n$-closed component [the
corresponding definition is above (\ref{defcompo})] of that exit
point. There is a positive probability that the ``positive''
boundary of that set is open.
\item[(3)] If that is the case, Lemma \ref{exitnet} implies that there
is a positive chance that we exit through the ``positive''
side of that set, which is included in~$\partial^+ \mathcal{B}_{kn}$.
\end{longlist}

The previous construction implies that we exit $\mathcal{B}_{kn}$ at
an open point which is on the positive boundary $\partial^+\mathcal
{B}_{kn}$. Hence that point is an open ladder point which can be used
to easily construct a point verifying the properties of~$\mathcal{M}^{(K)}$.

In the end this means that for each new larger box encountered,
$\mathcal{B}_{kn}$, $\mathcal{B}_{(k+1)n}, \ldots\,$, there is a
positive chance to encounter an open ladder point through a procedure
which is largely independent of what happened before in smaller boxes.
This will allow us to show that eventually we encountered a point with
all the properties of $\mathcal{M}^{(K)}$.
\begin{pf}
We introduce the $K$-$n$-closed component of the exit point of
$\mathcal{B}_n$
%
\begin{equation}
\label{ba} \mathcal{K}(n)=\overline{K}_{X_{T_{\partial_i^+ \mathcal
{B}_n}}}(n)\subseteq
\mathcal{B}_n,
\end{equation}
where we recall that the notation $\overline{K}_x(n)$ was defined
above (\ref{defcompo}). In case $T_{\partial_i^+ \mathcal
{B}_n}=\infty
$ (i.e., we never reach the inner positive boundary of $\mathcal{B}_n$)
we simply set $\mathcal{K}(n)=\varnothing$ and $\partial\mathcal
{K}(n)=\varnothing$. This case typically does not occur.

We introduce the event
\[
C(n)=\bigl\{\mbox{for all } x \in\partial\mathcal{K}(n) \cap\mathcal{H}^+(n)
\mbox{, the vertex }x \mbox{ is open} \bigr\}
\]
as well as the event
\[
D(n)=\{T_{\partial\mathcal{K}(n)} \circ\theta_{T_{\partial_i
\mathcal
{B}_n}} =T_{\partial\mathcal{K}(n) \cap\mathcal{H}^+(n)}\circ
\theta_{T_{\partial_i \mathcal{B}_n}}\}
\]
and the event
\begin{eqnarray*}
E(n)&=&\{X_{T_{\partial B(n,n^{\alpha})}+2}=X_{T_{\partial
B(n,n^{\alpha
})}+1}+e_1=X_{T_{\partial B(n,n^{\alpha})}}+2e_1
\\
&&\hspace*{55pt} \mbox{ and } X_{T_{\partial B(n,n^{\alpha})}+2}, X_{T_{\partial
B(n,n^{\alpha})}+1} \mbox{ are open}\}.
\end{eqnarray*}

\begin{figure}

\includegraphics{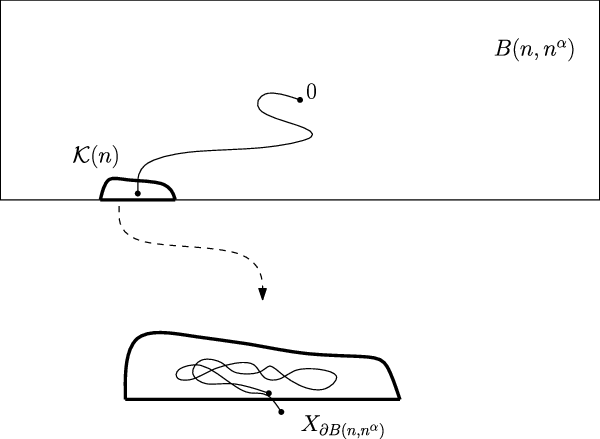}

\caption{A way to find an open ladder point.}\label{fig3}
\end{figure}

Let us consider an event in $A(n)\cap C(n)\cap D(n)\cap E(n)$. The
situation is illustrated in Figure \ref{fig3}. It verifies all the following conditions:
\begin{longlist}[(3)]
\item[(1)] on $A(n)$, we have $T_{\partial_i^+ \mathcal{B}_n} \leq
T_{\partial\mathcal{B}_n} $, so that
\[
T_{\partial B(n,n^{\alpha})}\geq T_{\partial\mathcal{K}(n)} \circ\theta
_{T_{\partial_i \mathcal{B}_n}};
\]
\item[(2)] on $D(n)$, by (\ref{ba}), we have
\[
T_{\partial\mathcal{K}(n)} \circ\theta_{T_{\partial_i \mathcal{B}_n}}
=T_{\partial\mathcal{K}(n) \cap\mathcal{H}^+(n)} \circ\theta
_{T_{\partial_i \mathcal{B}_n}}
\]
and hence
\[
T_{\partial B(n,n^{\alpha})}=T_{\partial\mathcal{K}(n) \cap\mathcal
{H}^+(n)};
\]
\item[(3)] on $C(n)$, we have $\partial\mathcal{K}(n) \cap\mathcal
{H}^+(n) $ is open.
\end{longlist}

Hence, on $A(n)\cap C(n)\cap D(n)\cap E(n)$, we see that
$X_{T_{\partial B(n,n^{\alpha})}} $ is a new maximum of the trajectory
in the direction $\vec{\ell}$ which is open, $X_{T_{\partial
B(n,n^{\alpha})}+1}=X_{T_{\partial B(n,n^{\alpha})}} +e_1$ and
$X_{T_{\partial B(n,n^{\alpha})}+2}=X_{T_{\partial B(n,n^{\alpha})}}+
2e_1 $ is a $K$-open point. This means that
\[
A(n)\cap C(n)\cap D(n)\cap E(n) \subset\{\mathcal{M}\leq T_{\partial
\mathcal{B}_n}+2\}.
\]

We have
%
\begin{eqnarray}
&& \PR\bigl[R(kn)\bigr]
\\
&&\qquad\leq \PR\bigl[R\bigl((k-1)n\bigr), \bigl(A(kn)\cap C(kn) \cap D(kn)
\cap E(kn)\bigr)^c\bigr]\nonumber
\\
&&\qquad\leq \PR\bigl[ A(kn)^c\bigr] + \PR\bigl[B(kn)^c
\bigr]+\cdots\nonumber
\\
\label{control1}
&&\qquad\quad{} + \PR\bigl[R\bigl((k-1)n\bigr),A(kn),B(kn),
C(kn)^c\bigr]+ \cdots
\\
\label{control2}
&&\qquad\quad{} + \PR\bigl[R\bigl((k-1)n\bigr),A(kn),B(kn), C(kn),
D(kn)^c\bigr]+\cdots
\\
\label{control3}
&&\qquad\quad{}+ \PR\bigl[R\bigl((k-1)n\bigr),A(kn),B(kn), C(kn),
D(kn), E(kn)^c\bigr],
\end{eqnarray}
which means that
%
\begin{eqnarray}
\label{long0} && \PR\bigl[R(kn)\bigr] - \PR\bigl[R\bigl((k-1)n\bigr),A(kn),B(kn)
\bigr]
\nonumber\\
&&\qquad\leq \PR\bigl[ A(kn)^c\bigr] + \PR\bigl[B(kn)^c
\bigr] \\
&&\qquad\quad{}- \PR\bigl[R\bigl((k-1)n\bigr),A(kn),B(kn), C(kn), D(kn),
E(kn)\bigr].\nonumber
\end{eqnarray}

The first term is controlled by Theorem \ref{BL},
%
\begin{equation}
\label{bres1} \PR\bigl[ A(kn)^c\bigr]\leq C\exp(-ckn)\leq C
\exp(-cn)
\end{equation}
and for any $M<\infty$, by Lemma \ref{tailb}, we can choose $K$ large
enough such that
%
\begin{equation}
\label{bres2} \PR\bigl[B(kn)^c\bigr] \leq n^{-M}.
\end{equation}

To finish the proof, we are going to control the terms in (\ref
{control1}), (\ref{control2}) and (\ref{control3}) which will allow us
to evaluate (\ref{long0}).\vspace*{9pt}

\textit{Step} 1: \textit{Control of the term in} (\ref{control1}).
We recall that $\mathcal{K}(kn)$ was defined at (\ref{ba}). For
$k\leq
n$, on $A(kn) \cap B(kn)$, we see that
%
\begin{equation}
\label{bb} \bigl\llvert\mathcal{K}(kn)\bigr\rrvert\leq\ln(kn)\leq2
\ln n
\end{equation}
and in particular,
\begin{eqnarray*}
&& \PR\bigl[R\bigl((k-1)n\bigr),A(kn),B(kn), C(kn)^c\bigr]
\\
&&\qquad= \sum_{F\subset\Z^d, \llvert F\rrvert \leq2 \ln n } \PR\bigl[R\bigl((k-1)n
\bigr),A(kn),B(kn),\mathcal{K}(kn)=F, C(kn)^c\bigr]
\\
&&\qquad= \sum_{F\subset\Z^d, \llvert F\rrvert \leq2 \ln n } {\mathbf E}\bigl
[P^{\omega
}\bigl[R
\bigl((k-1)n\bigr),A(kn),B(kn),\mathcal{K}(kn)=F\bigr]
\\
&&\hspace*{59pt}\hspace*{38.7pt}\qquad\quad{} \times{\mathbf1} {\bigl\{\mbox{some }x \in\partial F \cap\mathcal{H}^+(kn)
\mbox{ is closed}\bigr\}} \bigr].
\end{eqnarray*}

We may now see that:
\begin{longlist}[(2)]
\item[(1)] on the one hand, the random variable $P^{\omega
}[R((k-1)n),A(kn),B(kn)$, $\mathcal{K}(kn)=F]$ is measurable with respect
to $\sigma\{c_*([x,y])$, with $x,y \notin\mathcal{H}^+_{nk} \}$;
\item[(2)] on the other hand, the event $\{ \mbox{some }x \in
\partial
F \cap\mathcal{H}^+(kn) \mbox{ is closed}\}$ is measurable with
respect to $\sigma\{c_*([x,y])$, with $x \in\mathcal
{H}^+_{nk} \}$.
\end{longlist}

Hence, the random variables $P^{\omega}[R((k-1)n),A(kn),B(kn),\mathcal
{K}(kn)=F]$ and ${\mathbf1}{\{\mbox{some }x \in\partial F \cap
\mathcal{H}^+(kn) \mbox{ is closed}\}} $ are ${\mathbf P}$-independent.
This yields
\begin{eqnarray*}
&& \PR\bigl[R\bigl((k-1)n\bigr),A(kn),B(kn), C(kn)^c\bigr]
\\
&&\qquad= \sum_{F\subset\Z^d, \llvert F\rrvert \leq2 \ln n } \PR\bigl
[R\bigl((k-1)n
\bigr),A(kn),B(kn),\mathcal{K}(kn)=F\bigr]
\\
&&\hspace*{60pt}\qquad\quad{} \times{\mathbf P}\bigl[ \mbox{some }x \in\partial F \cap\mathcal{H}^+(kn)
\mbox{ is closed} \bigr]
\\
&&\qquad\leq \sum_{F\subset\Z^d, \llvert F\rrvert \leq2 \ln n } \PR\bigl
[R\bigl((k-1)n
\bigr),A(kn),B(kn),\mathcal{K}(kn)=F\bigr]
\\
&&\hspace*{60pt}\qquad\quad{} \times\bigl(1-{\mathbf P}\bigl[ \mbox{all }x \in\partial F \cap\mathcal{H}^+(kn)
\mbox{ are open} \bigr]\bigr).
\end{eqnarray*}

Now, we know by the Harris inequality \cite{Harris} that for $F\subset
\Z^d$, with $\llvert F\rrvert \leq2 \ln n$
\begin{eqnarray*}
{\mathbf P}\bigl[ \mbox{all }x \in\partial F \cap\mathcal{H}(kn) \mbox{ are open}
\bigr]&\geq&{\mathbf P}[x\mbox{ is open}]^{d\llvert F\rrvert}
\\
&\geq&\bigl(1-\epsilon(K)\bigr)^{2d \ln n}\\
&=&n^{2 d\ln(1-\epsilon(K))},
\end{eqnarray*}
where we recall that $\epsilon(K)=P_*[c_*\notin[1/K,K]]$.\eject

For any $\epsilon_1>0$, by choosing $K$ large enough, we can assume
that $2 d\ln(1-\epsilon(K))\geq-\epsilon_1$. This means that the two
previous equations imply that
\begin{eqnarray*}
&&
\PR\bigl[R\bigl((k-1)n\bigr),A(kn),B(kn), C(kn)^c\bigr] \\
&&\qquad\leq
\bigl(1-n^{-\epsilon_1}\bigr) \PR\bigl[R\bigl((k-1)n\bigr
),A(kn),B(kn)\bigr],
\end{eqnarray*}
which means that for any $\epsilon_1>0$, we have
%
\begin{eqnarray}
\label{bres3}
&&
n^{-\epsilon_1} \PR\bigl[R\bigl((k-1)n\bigr
),A(kn),B(kn)\bigr]\nonumber\\[-8pt]\\[-8pt]
&&\qquad\leq
\PR\bigl[R\bigl((k-1)n\bigr),A(kn),B(kn), C(kn)\bigr].\nonumber
\end{eqnarray}

\vspace*{9pt}

\textit{Step} 2: \textit{Control of the term in} (\ref{control2}).
We want to upper-bound $\PR[R((k-1)n),A(kn),B(kn), C(kn), D(kn)^c]$.
On $A(kn)\cap B(kn)\cap C(kn)$, we have reached the positive inner
border of $\mathcal{B}_{nk}$, and we know that $\mathcal{K}(kn)$ is not
too big and its ``positive border'' is open. By applying
Lemma \ref{exitnet}, we have an estimate for the probability of
exiting $\mathcal{K}(kn)$ through the positive border.

To start, we wish to decompose $\PR[R((k-1)n),A(kn),B(kn), C(kn),\break
D(kn)^c]$ according to all possible values of $X_{T_{\partial_i
\mathcal
{B}_{nk}}}$ and $\mathcal{K}(kn)$. For this, we notice that:
\begin{longlist}[(2)]
\item[(1)] on $A(kn)$, we have $X_{T_{\partial_i \mathcal
{B}_{nk}}}\in
\partial_i^+ \mathcal{B}_{nk}$, and by the definition of $\mathcal
{K}(kn)$ [see (\ref{ba})], $X_{T_{\partial_i \mathcal{B}_{nk}}}\in
\mathcal{K}(kn)$;
\item[(2)] moreover, on $A(kn) \cap B(kn)$, we have (\ref{bb}).
\end{longlist}

Hence,
\begin{eqnarray*}
&&\PR\bigl[R\bigl((k-1)n\bigr),A(kn),B(kn), C(kn), D(kn)^c\bigr]
\\
&&\qquad\leq \sum_{y,F} \PR\bigl[R\bigl((k-1)n
\bigr),A(kn),B(kn),\\
&&\qquad\quad\hspace*{27.5pt}X_{T_{\partial_i
\mathcal
{B}_{nk}}}=y,\mathcal{K}(kn)=F, C(kn),D(kn)^c
\bigr],
\end{eqnarray*}
where $\sum_{y,F}$ stands for $ \sum_{y\in\partial_i^+ \mathcal
{B}_{nk}} \sum_{F\subset\Z^d, \llvert F\rrvert \leq2 \ln
n, y\in F }$.

Let us notice that, for a fixed $\omega$, the events $R((k-1)n)$,
$A(kn)$, $B(kn)$, $\{X_{T_{\partial_i \mathcal{B}_{nk}}}=y\}$ and $\{
\mathcal{K}(kn)=F\}$ are $P^{\omega}$-measurable\vspace*{2pt} with respect to $\{
X_i,\break i \leq T_{\partial_i \mathcal{B}_{nk}}\}$. Thus, we may use the
Markov property at $T_{\partial_i \mathcal{B}_{nk}}$ to see that
%
\begin{eqnarray}
\label{bc} &&\PR\bigl[R\bigl((k-1)n\bigr),A(kn),B(kn), C(kn), D(kn)^c
\bigr]
\nonumber\hspace*{-35pt}\\
&&\quad\leq \sum_{y,F} {\mathbf E}
\bigl[P^{\omega
}\bigl[R\bigl((k-1)n\bigr),A(kn),B(kn),X_{T_{\partial_i \mathcal{B}_{nk}}}=y,
\mathcal{K}(kn)=F\bigr]\hspace*{-35pt}
\\
&&\hspace*{22pt}\qquad{} \times P^{\omega}_y[T_{\partial
F}<T_{\partial F \cap\mathcal{H}^+(kn)}]{
\mathbf1} {\bigl\{\mbox{$x$ is open, for }x\in\partial F \cap\mathcal{H}^+(kn)
\bigr\}} \bigr].\nonumber\hspace*{-35pt}
\end{eqnarray}

We wish to apply Lemma \ref{exitnet}, for this we will first prove
that on $\{\mathcal{K}(kn)=F\}$ and $\{\mbox{$x$ is open, for }x\in
\partial F \cap\mathcal{H}^+(kn)\}$ the set $\partial_E F$ is composed
of normal edges. Indeed:
\begin{longlist}[(2)]
\item[(1)] notice that the definition of $\mathcal{K}(kn)$ at (\ref
{ba}) [which is a $K$-$(kn)$-closed component] implies that $e$ is
normal for all $e\in\partial_E \mathcal{K}(kn)$ when $e$ has no
endpoint in $\mathcal{H}^+(kn)$;
\item[(2)] moreover, if for any $x\in\partial F \cap\mathcal{H}^+(kn)$
the vertex $x$ is open, then any edge $e\in\partial_E F$ with one
endpoint in $\mathcal{H}^+(kn)$ is normal.
\end{longlist}

Hence $\partial_E F$ is a set made of normal edges only.

Now, for any $y\in F \cap\partial_i^+ \mathcal{B}_{nk}$, there exists
$x\in\mathcal{H}^+(kn)$ adjacent to $y$. Since $F\subset\mathcal
{B}_{nk}$, we can see that for any $z\in\partial F\cup F$ we have
$(x-z)\cdot\vec{\ell} \geq-1$ and $(y-z)\cdot\vec{\ell} \geq-2$.
Using this, along with the fact that $\partial_E F$ is made of normal
edges, we see with (\ref{defconduct}) and Remark \ref{fakeUE} that
\[
\mbox{for all $e\in\partial_E F$}\qquad c^{\omega}(e)\leq
K^2 e^{3\lambda} c^{\omega}\bigl([x,y]\bigr).
\]

We can apply Lemma \ref{exitnet} to $F$, and we see that if $\{
\mathcal
{K}(kn)=F\}$ and $\{\mbox{$x$ is open, for }x\in\partial F \cap
\mathcal
{H}^+(kn)\}$, then we obtain
\begin{eqnarray*}
P^{\omega}_y[T_{\partial F}<T_{\partial F \cap\mathcal{H}(kn)}] &\leq&
P^{\omega}_y[T_{\partial F}<T_{x}] \leq
\bigl(1-c \llvert F\rrvert^{-1}\bigr) \\
&\leq&\bigl(1-c\ln^{-1} n
\bigr),
\end{eqnarray*}
since $\llvert F\rrvert\leq2 \ln n$.

This turns (\ref{bc}) into
\begin{eqnarray*}
&&\PR\bigl[R\bigl((k-1)n\bigr),A(kn),B(kn), C(kn), D(kn)^c\bigr]
\\
&&\qquad\leq \sum_{y,F} {\mathbf E}\bigl[P^{\omega
}
\bigl[R\bigl((k-1)n\bigr),A(kn),B(kn),X_{T_{\partial_i \mathcal
{B}_{nk}}}=y,\mathcal{K}(kn)=F
\bigr]
\\
&&\hspace*{113.2pt}\qquad\quad{} \times{\mathbf1} {\bigl\{\mbox{$x$ is open, for }x\in\partial F \cap
\mathcal
{H}^+(kn)\bigr\}} \bigr] \\
&&\qquad\quad\hspace*{12.5pt}{}\times\bigl(1-c\ln^{-1} n\bigr)
\\
&&\qquad\leq \bigl(1-c\ln n^{-1}\bigr) \PR\bigl[R\bigl((k-1)n
\bigr),A(kn),B(kn), C(kn)\bigr],
\end{eqnarray*}
which means that for some $c>0$, we have
%
\begin{eqnarray}
\label{bres4} &&c\ln n^{-1} \PR\bigl[R\bigl((k-1)n\bigr),A(kn),B(kn),
C(kn)\bigr]
\nonumber\\[-8pt]\\[-8pt]
&&\qquad\leq\PR\bigl[R\bigl((k-1)n\bigr),A(kn),B(kn), C(kn),
D(kn)\bigr].\nonumber
\end{eqnarray}

\vspace*{9pt}

\textit{Step} 3: \textit{Control of the term in} (\ref{control3}).
On $ A(kn) \cap B(kn) \cap C(kn) \cap D(kn)$, we know that
$X_{T_{\partial\mathcal{B}_{kn}}}\in\partial^+ \mathcal{B}_{kn}$ is
an open ladder point. Moreover, it is important to notice that the
event $E(kn)$ has a positive probability of happening and does not
depend on what happened inside $\mathcal{B}_{kn}$. We introduce
\[
R'\bigl((k-1)n\bigr)=R\bigl((k-1)n\bigr)\cap A(kn)\cap B(kn)\cap
C(kn)\cap D(kn).
\]

A vertex is said to be $x$-open if it is open in $\omega_x$ coinciding
with $\omega$ on all edges, but those that are adjacent to $x$ which
are normal in $\omega_x$. We see
\begin{eqnarray*}
&&\PR\bigl[R\bigl((k-1)n\bigr),A(kn),B(kn), C(kn), D(kn), E(kn)^c
\bigr]
\\
&&\qquad\leq \sum_{x\in\partial^+ \mathcal{B}_{nk}} \PR\bigl[R'
\bigl((k-1)n\bigr), X_{T_{\partial\mathcal{B}_{nk}}}=x, x\mbox{ is
open}, E(kn)^c
\bigr]
\\
&&\qquad\leq \sum_{x\in\partial^+ \mathcal{B}_{nk}} {\mathbf E}\bigl[P^{\omega
}
\bigl[R'\bigl((k-1)n\bigr), X_{T_{\partial\mathcal{B}_{nk}}}=x\bigr]
{\mathbf1} {\{
x\mbox{ is open}\}},
\\
&&\qquad\quad\hspace*{46pt} \bigl({\mathbf1} {\{x+e_1 \mbox{ or }x+2e_1 \mbox{
is not $x$-open}\}}
\\
&&\qquad\quad\hspace*{48.5pt}{} +P_x^{\omega}[X_1\neq x+e_1 \mbox{
or } X_2\neq x+2e_1] \\
&&\qquad\quad\hspace*{71pt}{}\times{\mathbf1} {\{x+e_1,x+2e_1
\mbox{ are $x$-open}\}}\bigr)\bigr].
\end{eqnarray*}

On $\{x+e_1,x+2e_1 \mbox{ are $x$-open}\}\cap\{x\mbox{ is open}\}$, we
see that $P_x^{\omega}[X_1=x+e_1, X_2=x+2e_1] \geq c>0$ by Remark \ref
{fakeUE}.
\begin{eqnarray*}
&&\PR\bigl[R\bigl((k-1)n\bigr),A(kn),B(kn), C(kn), D(kn), E(kn)^c
\bigr]
\\
&&\qquad\leq\sum_{x\in\partial^+ \mathcal{B}_{nk}} {\mathbf E}\bigl[P^{\omega
}
\bigl[R'\bigl((k-1)n\bigr), X_{T_{\partial\mathcal{B}_{nk}}}=x\bigr]
{\mathbf1} {\{
x\mbox{ is open}\}},
\\
&& \hspace*{79.1pt}\bigl({\mathbf1} {\{x+e_1 \mbox{ or }x+2e_1 \mbox{
is not $x$-open}\}} \\
&&\qquad\hspace*{68pt}{}+(1-c) {\mathbf1} {\{x+e_1,x+2e_1
\mbox{ are $x$-open}\}}\bigr)\bigr].
\end{eqnarray*}

Recalling the definition of $\nu$ at the beginning of Section \ref
{sectnotation}, we may also see that $\{R'((k-1)n), X_{T_{\partial
\mathcal{B}_{nk}}}=x, x\mbox{ is open}\}$ is measurable with respect to
$\sigma\{c_*(e), e\in E(\mathcal{B}_{nk}) \mbox{ or } e=[x,x+e']$
with $e'\in\nu\}$, whereas $\{x+e_1,x+2e_1\break \mbox{are $x$-open}\}$
is measurable with respect to $\sigma\{c_*(e), e\notin E(\mathcal
{B}_{nk})$ and $e\neq[x,x+e']$ with $e'\in\nu\}$. So
these random variables are independent, which yields
\begin{eqnarray*}
&& \PR\bigl[R\bigl((k-1)n\bigr),A(kn),B(kn), C(kn), D(kn), E(kn)^c
\bigr]
\\
&&\qquad\leq \PR\bigl[R'\bigl((k-1)n\bigr)\bigr]\bigl({\mathbf
P}[x+e_1 \mbox{ or }x+e_2 \mbox{ is not $x$-open}]
\\
&&\hspace*{73.5pt}\qquad\quad{} +(1-c){\mathbf P}[x+e_1,x+e_2 \mbox{ are $x$-open}]
\bigr)
\\
&&\qquad\leq(1-c)\PR\bigl[R\bigl((k-1)n\bigr),A(kn),B(kn), C(kn), D(kn)\bigr],
\end{eqnarray*}
since ${\mathbf P}[x+e_1,x+e_2 \mbox{ are $x$-open}]>0$. This means that
there exists $c>0$ such that
%
\begin{eqnarray}
\label{bres5} && c\PR\bigl[R\bigl((k-1)n\bigr),A(kn),B(kn), C(kn),
D(kn)\bigr]
\nonumber\\[-8pt]\\[-8pt]
&&\qquad\leq  \PR\bigl[R\bigl((k-1)n\bigr),A(kn),B(kn), C(kn), D(kn), E(kn)
\bigr].\nonumber
\end{eqnarray}

\vspace*{9pt}

\textit{Step} 4: \textit{Conclusion.}
For any $\epsilon_1>0$, we see using (\ref{bres1}), (\ref
{bres2}), (\ref
{bres3}), (\ref{bres4}), (\ref{bres5}) (which are valid $K$ chosen
larger than some $K_0$ depending only on $M<\infty$) and (\ref{long0}),
that we have for any $k\in[2,n]$
\[
\PR\bigl[R(kn)\bigr]\leq\PR\bigl[R\bigl((k-1)n\bigr)\bigr] \bigl(1-c\ln
n^{-1}n^{-\epsilon_1} \bigr)+Cn^{-M},
\]
which implies the result.
\end{pf}

We now prove the following:
%
\begin{lemma}
\label{tailM}
For any $M$, there exists $K_0$ such that, for any $K\geq K_0$,
\[
\PR[X_{\mathcal M^{(K)}} \cdot\vec{\ell} \geq n ] \leq C(K)n^{-M}.
\]
\end{lemma}

\begin{pf}
For any $M<\infty$, by Lemma \ref{recursion}, there exists $K_0$ such
that, for any $K\geq K_0$ such that
\[
\PR\bigl[R(nk)\bigr] \leq\bigl(1-cn^{-1/2}\bigr)\PR\bigl[R\bigl(n(k-1)
\bigr)\bigr]+n^{-M}.
\]

By a simple induction, this means that for
\[
\PR\bigl[R\bigl(n^2\bigr)\bigr]\leq\bigl(1-cn^{-1/2}
\bigr)^n + n^{-M+1} \leq2n^{-M+1}.
\]

Recalling the definition of $R(n)$ at (\ref{defR}), we see by the
Borel--Cantelli lemma that $\mathcal{M}^{(K)}<\infty$.

Also this implies that
\[
\PR\bigl[X_{\mathcal{M}^{(K)}}\cdot\vec{\ell} > n^2+2\bigr]
\leq2n^{-M+1}
\]
and
\[
\PR[X_{\mathcal{M}^{(K)}}\cdot\vec{\ell} > n] \leq Cn^{-(M+1)/2},
\]
which proves the lemma, since $M$ is arbitrary.
\end{pf}

\subsection{Consequence of our estimates on $\mathcal{M}$}

A natural consequence of the previous estimate is that the successive
open ladder points cannot be too distant, and this is what we aim at
showing next in a form that will be useful for us in the sequel. Let us
introduce the ladder times
%
\begin{equation}
\label{defW} W_0=0 \quad\mbox{and}\quad W_{k+1}=
\inf\{n \geq0, X_n\cdot\vec{\ell} >X_{W_k} \cdot\vec{\ell}
\}.
\end{equation}

We introduce the event
%
\begin{eqnarray}
\label{defM} M^{(K)}(n)&=&\bigl\{\mbox{for $k$ with
$W_k \leq\Delta_n$,}\nonumber\\[-8pt]\\[-8pt]
&&\hspace*{5.6pt}\mbox{we have } X_{\mathcal{M}^{(K)}\circ\theta
_{W_k}+W_k}\cdot\vec{
\ell} -X_{W_k}\cdot\vec{\ell}\leq n^{1/2}\bigr\}.\nonumber
\end{eqnarray}

\begin{lemma}
\label{Mn}
For any $M<\infty$, there exists $K_0$ such that, for any $K\geq K_0$
we have
\[
\PR\bigl[M^{(K)}(n)^c\bigr] \leq Cn^{-M}.
\]
\end{lemma}

\begin{pf}
We introduce the event
\[
M_j(n)=\bigl\{\mbox{for $k\leq j-1$, we have } X_{\mathcal{M}^{(K)}\circ
\theta_{W_k}+W_k}
\cdot\vec{\ell} -X_{W_k}\cdot\vec{\ell}\leq n^{1/2}\bigr\}
\]
and the event
\[
N_j(n)=\bigl\{X_{\mathcal{M}^{(K)}\circ\theta_{W_j}+W_j}\cdot\vec{\ell} -X_{W_j}
\cdot\vec{\ell} > n^{1/2} \mbox{ and } X_{W_j}\in B
\bigl(n,n^{\alpha}\bigr)\bigr\}.
\]

On $M^{(K)}(n)^c\cap\{T_{\partial B(n,n^{\alpha})}=T_{\partial^+
B(n,n^{\alpha})}\}$, since all $X_{W_i} $ are different necessarily, we
have $k\leq Cn^{2d\alpha}$ for any $k$ such that $W_k\leq\Delta_n$.
Moreover there exists a $k\leq\Delta_n$ such that $N_k(n)$ holds. By
decomposing along the smallest such $k$, we see that
%
\begin{eqnarray}
\label{sstep} \PR\bigl[M(n)^c\bigr] &\leq&\sum
_{k=1}^{Cn^{2d\alpha}} \PR\bigl[M_k(n),N_k(n),
T_{\partial B(n,n^{\alpha})}=T_{\partial^+ B(n,n^{\alpha})}\bigr]\nonumber\\[-8pt]\\[-8pt]
&&{}+\exp(-cn)\nonumber
\end{eqnarray}
by Theorem \ref{BL}.

Now, we see that on $\{ M_j(n),N_j(n),T_{\partial B(n,n^{\alpha
})}=T_{\partial^+ B(n,n^{\alpha})}\}$ we have $X_{W_{j}}\in
B(n,n^{\alpha})$, so by a simple union bound argument,
\begin{eqnarray*}
&& \PR\bigl[M_j(n),N_j(n),T_{\partial B(n,n^{\alpha})}=T_{\partial^+
B(n,n^{\alpha})}
\bigr]
\\
&&\qquad\leq  \sum_{x\in B(n,n^{\alpha})}\PR\bigl[X_{W_{j}}=x,X_{\mathcal
{M}^{(K)}\circ\theta_{W_j}+W_j}
\cdot\vec{\ell} -x\cdot\vec{\ell}>n^{1/2}\bigr]
\\
&&\qquad\leq  \sum_{x\in B(n,n^{\alpha})} {\mathbf E}\bigl[P^{\omega}_x
\bigl[X_{\mathcal
{M}^{(K)}}\cdot\vec{\ell}-x\cdot\vec{\ell} \geq n^{1/2}
\bigr]\bigr]
\\
&&\qquad\leq  \bigl\llvert B\bigl(n,n^{\alpha}\bigr)\bigr\rrvert\PR\bigl
[\mathcal
{M}^{(K)} \geq n^{1/2}\bigr]
\end{eqnarray*}
by Markov's property at $W_{j}$ and translation invariance of the
environment. Hence, by the two last equations, Lemma \ref{tailM} and
using (\ref{sstep}), we may see that
\[
\PR\bigl[M(n)^c\bigr]\leq n^{-M}.
\]
\upqed\end{pf}

\section{Regeneration times}\label{sectfirstregen}

The aim of this section is to define regeneration times and prove some
standard properties on them. These properties are summed up in
Section \ref{fundprop}.

The idea is to find a maximum of the trajectory, in the direction $\vec
{\ell}$ from which the random walk will never backtrack, that is, go to
a point with lower scalar product with $\vec{\ell}$. Essentially, we
would like to call regeneration time the first time that such a
situation occurs. For technical reasons it is convenient for us to
consider only the maxima which are also $K$-open points (or more
precisely points verifying the properties of $\mathcal{M}$). This is
the only difference from the standard definitions of regeneration times.

We define the time it takes for the walk to go back beyond (with
respect to the scalar product with $\vec{\ell}$) its starting point
\[
D=\inf\{n> 0\mbox{ such that } X_n\cdot\vec{\ell}\leq
X_0\cdot\vec{\ell}\}.
\]

Also we introduce the maximum (in the direction $\vec{\ell}$) of the
trajectory before $D$
\[
M=\sup_{n\leq D} X_n\cdot\vec{\ell}.
\]

We define the configuration dependent stopping times $S_k$, $k\geq0$
and the levels $M_k$, $k\geq0$,
%
\begin{eqnarray}
\label{defS} &\displaystyle S_0=0,\qquad M_0=X_0
\cdot\vec{\ell}\quad\mbox{and}&\nonumber\\[-8pt]\\[-8pt]
&\displaystyle \mbox{for $k\geq0$}\qquad
S_{k+1}=\mathcal{M}^{(K)}\circ\theta_{T_{\mathcal
{H}^+(M_k)}}+T_{\mathcal{H}^+(M_k)},&\nonumber
\end{eqnarray}
where
%
\begin{equation}
\label{defrealM} M_k=\sup\{X_m \cdot\vec{
\ell} \mbox{ with }0\leq m \leq R_k\}
\end{equation}
with
\[
R_{k} = D\circ\theta_{S_k}+S_k.
\]

In words, $R_k$ is the time it takes to go back beyond $X_{S_k}$, $M_k$
is the maximum of the trajectory before $R_k$, and $S_{k+1}$ is the
first time we see an open ladder point (more precisely a point
verifying the properties of $\mathcal{M}^{(K)}$) after getting past $M_k$.

These definitions imply that if $S_{i+1}<\infty$, then
%
\begin{equation}
\label{rightdir} X_{S_{i+1}}\cdot\vec{\ell} -
X_{S_i} \cdot\vec\ell\geq2 e_1\cdot\vec{\ell} \geq
\frac{2}{\sqrt d}.
\end{equation}

Finally we define the basic regeneration time
%
\begin{equation}
\label{deftau1} \tau_1=S_N\qquad \mbox{with } N=\inf\{k \geq1\mbox{ with } S_k<\infty\mbox{ and }
M_k=\infty\}.
\end{equation}

Let us give some intuition about those definitions. Assume $S_k$ is
constructed; it is, by definition, an open-point ladder point. We will
show that at such a point there is a lower bounded chance of never
backtracking again. If the walks never backtrack again (then
$R_k=\infty
$ and thus $M_k=\infty$), we have created a point separating the past
and the future of the random walk: a regeneration point called $\tau
_1$. This finishes the procedure.

In case a regeneration time is not created, the future of the random
walk and the environment ahead of us may be conditioned by the fact
that the walk will eventually backtrack: a conditioning limited to the
conductances of the edges adjacent to the trajectory of the walk before
it backtracks and the trajectory of the walk itself before backtracking.

In our definitions we introduced a random variable $M_k$ chosen large
enough so that all the edges we just described have one endpoint in
$\mathcal{H}^-(M_k)$. Hence the environment in $\mathcal{H}^+(M_k)$ and
the walk after it reaches this set are largely unconditioned. After
reaching that set, we have the opportunity to construct another open
ladder point (for a walk free of any constraints from its past) by
considering the first open ladder point after entering $\mathcal
{H}^+(M_k)$. This is how we define $S_{k+1}$ and from there we start
over the procedure until we find a regeneration time.

\subsection{Control the variables $M$}

We want to show that the random variables $M_k$ in (\ref{defrealM})
cannot be too big. For this we prove the following lemma:

\begin{lemma}
\label{M1}
We have
\[
\PR[M\geq n \mid D<\infty]\leq C\exp(-cn).
\]
\end{lemma}
\begin{pf}
We have
\begin{eqnarray*}
&&\PR\bigl[2^k\leq M< 2^{k+1}\bigr]
\\
&&\qquad\leq  \PR[T_{\partial B(2^k,2^{\alpha k})}\neq T_{\partial
^+B(2^k,2^{\alpha k})}]
\\
&&\qquad\quad{} + \PR\bigl[X_{T_{\partial B(2^k,2^{\alpha k})}} \in\partial^+B\bigl
(2^k,2^{\alpha k}
\bigr),\\
&&\qquad\quad\hspace*{24.5pt} T_{\mathcal{H}^-(0)}^+ \circ\theta_{
T_{\partial
^+B(2^k,2^{\alpha k})}}<T_{\mathcal{H}^+(2^{k+1})}\circ
\theta_{
T_{\partial^+B(2^k,2^{\alpha k})}}\bigr]
\end{eqnarray*}
and by a union bound on the $2^{\alpha k}$ possible positions of $X_{
T_{\partial^+B(2^k,2^{\alpha k})}}$ and using translation invariance arguments,
\[
\PR\bigl[2^k\leq M< 2^{k+1}\bigr] \leq2^{\alpha k}
\PR[T_{\partial B(2^{k+1},2^{\alpha(k+1)})}\neq T_{\partial
^+B(2^{k+1},2^{\alpha(k+1)})}]+e^{-(2^k)}
\]
as a consequence of Theorem \ref{BL}.

Hence, using Theorem \ref{BL}, again
\[
\PR\bigl[2^k\leq M< 2^{k+1}\bigr]\leq c
2^{\alpha k}e^{-c(2^k)}
\]
and since $M<\infty$ on $D<\infty$, we see that
\[
\PR[M\geq n\mid D<\infty] \leq\frac1 {\PR[D<\infty]} \sum
_{k, 2^k\geq
n} \PR\bigl[2^k\leq M< 2^{k+1}\bigr]
\leq C e^{-cn}.
\]
\upqed\end{pf}


Recalling the definition of $M_k$ at (\ref{defrealM}), we introduce
the event
%
\begin{equation}
\label{defSn} S(n)=\bigl\{\mbox{for $i$ with $S_i
\leq\Delta_n$ and $M_i<\infty$, } M_i-X_{S_i}
\cdot\vec{\ell}\leq n^{1/2}\bigr\}.
\end{equation}

Let us prove the following lemma:
%
\begin{lemma}
\label{Sn}
We have
\[
\PR\bigl[S(n)^c\bigr] \leq\exp\bigl(-n^{1/2}\bigr).
\]
\end{lemma}
\begin{pf}
By (\ref{rightdir}), we know that $\operatorname{card}\{ i; S_i\leq\Delta
_n\}
\leq n$. So, we see that
%
\begin{eqnarray}
\label{cdeb}
\qquad\PR\bigl[S(n)^c\bigr] & \leq&\PR[T_{B(n,n^{\alpha})}\neq
T_{\partial^+
B(n,n^{\alpha})}]
\nonumber\\[-8pt]\\[-8pt]
&&{} +\sum_{i\leq n}\sum
_{x\in B(n,n^\alpha)}\PR\bigl[M_i-X_{S_i}\cdot\vec{
\ell}> n^{1/2}, M_i<\infty, X_{S_i}=x\bigr],\nonumber
\end{eqnarray}
where the first term can be controlled by Theorem \ref{BL}.

Now, we may see that
%
\begin{eqnarray}
\label{ca} &&\PR\bigl[M_i-X_{S_i}\cdot\vec{\ell}>
n^{1/2}, M_i<\infty, X_{S_i}=x\bigr]
\nonumber\\[-8pt]\\[-8pt]
&&\qquad\leq  \PR\Bigl[\sup_{n\leq D\circ\theta_{S_i}+S_i} (X_n-x)\cdot
\vec{\ell} > n^{1/2}, M_i<\infty, X_{S_i}=x
\Bigr].\nonumber
\end{eqnarray}


If $M_i<\infty$, we have $D\circ\theta_{S_i}+S_i <\infty$, hence
%
\begin{eqnarray}
\label{cc}\quad && \PR\Bigl[\sup_{n\leq D\circ\theta_{S_i}+S_i} (X_n-x)\cdot
\vec
{\ell} \geq n^{1/2}, M_i<\infty, X_{S_i}=x \Bigr]
\nonumber\\[-8pt]\\[-8pt]
&&\qquad\leq \PR\Bigl[ \sup_{n\leq D\circ\theta_{S_i}+S_i} (X_n-x)\cdot
\vec{\ell} \geq n^{1/2}, D\circ\theta_{S_i}+S_i
<\infty, X_{S_i}=x \Bigr].\nonumber
\end{eqnarray}

Recalling that $X_{S_i}$ is a maximum in the direction $\vec{\ell}$ of
the past trajectory, we can use Markov's property at the time $S_i$ to
see that
%
\begin{eqnarray}
\label{cd} && \PR\Bigl[ \sup_{n\leq D\circ\theta_{S_i}+S_i} (X_n-x)\cdot
\vec{\ell} \geq n^{1/2}, D\circ\theta_{S_i}+S_i
<\infty, X_{S_i}=x \Bigr]
\nonumber\\
&&\qquad\leq  {\mathbf E} \Bigl[P^{\omega}_x \Bigl[\sup
_{n\leq D} X_n\cdot\vec{\ell} \geq n^{1/2},
D<\infty\Bigr] \Bigr]
\\
&&\qquad\leq  \PR\Bigl[\sup_{n\leq D} X_n\cdot
\vec{\ell} \geq n^{1/2}\mid D<\infty\Bigr] \leq C\exp
\bigl(-cn^{1/2}\bigr),\nonumber
\end{eqnarray}
where we used translation invariance and Lemma \ref{M1}. The result
follows from putting together (\ref{cdeb}), (\ref{ca}), (\ref{cc})
and (\ref{cd}).
\end{pf}

\subsection{Exponential tails for backtracking}

As a result of Theorem \ref{BL}, we know that the walk is exponentially
unlikely to backtrack a lot. This can be seen as follows, and given a
large $n$, it is extremely likely to exit $B(2^n,2^{\alpha n})$ through
the positive side. Centering at that exit point a box of size
$B(2^{n+1},\break2^{\alpha(n+1)})$, we are again very likely to exit through
the positive side. Applying this reasoning recursively, we see that we
are unlikely to reach $\mathcal{H}^-(-2^n)$.
%
\begin{lemma}
\label{thisisboring}
We have for any $n$,
\[
\PR[T_{\mathcal{H}^-(-n)}<\infty]\leq C\exp(-cn).
\]
\end{lemma}
\begin{pf}
Fix $n>0$. For this proof, we will use the event
\[
A(n)= \{T_{\partial B(2^n,2^{n\alpha})}= T_{ \partial^+
B(2^n,2^{n\alpha
})}\}.
\]

For any $k\geq n$, let us denote $B_X(2^{k+1},2^{(k+1)\alpha}) = \{
z\in\Z^d, z=X_{T_{\partial B(2^{k},2^{k\alpha})}} +y$ with
$y\in B(2^{k+1},2^{(k+1)\alpha})\}$, and we introduce the event
\[
C(k)=\{T_{\partial^+B_X(2^{k+1},2^{(k+1)\alpha})}\circ\theta
_{T_{\partial B(2^k,2^{k\alpha})}} =T_{\partial
B_X(2^{k+1},2^{(k+1)\alpha})}\circ
\theta_{T_{\partial
B(2^k,2^{k\alpha
})}}\}.
\]

A simple induction shows that on $ \bigcap_{k\in[n,m]} C(k)\cap A(n)$, we
have\break $\{T_{\mathcal{H}^-(-2^n-1)} \geq T_{B(2^m,m2^{\alpha m})}\}$, and
hence we see that
%
\begin{equation}
\label{aa} \bigcap_{k\geq n} C(k)\cap A(n) \subseteq\{
T_{\mathcal
{H}^-(-2^n-1)}=\infty\}.
\end{equation}

Denote for $m> n$,
\[
D(n,m)= C(k)\cap C(m)^c\cap\biggl(\bigcap
_{n\leq k<m} A(n) \biggr),
\]
so that using (\ref{aa}), we see
\[
\{T_{\mathcal{H}^-(-2^n)}<\infty\}\subset\biggl(\bigcap_{k\geq n}
C(k)\cap A(n) \biggr)^c \subset\bigcup_{m\geq n}
D(n,m) \cup A(n)^c,
\]
which implies
%
\begin{eqnarray}
\label{ab} \PR[T_{\mathcal{H}^-(-2^n)}<\infty] &\leq&\PR\bigl[A(n)^c\bigr]+
\sum_{m\geq n} \PR\bigl[D(n,m)\bigr]\nonumber\\[-8pt]\\[-8pt]
&\leq&\exp
\bigl(-c2^n\bigr)+ \sum_{m\geq n} \PR
\bigl[D(n,m)\bigr]\nonumber
\end{eqnarray}
by Theorem \ref{BL}.

We may notice that on $D(n,m)$, we have $\{T_{\partial
B(2^m,m2^{m\alpha
})}=T_{\partial^+ B(2^m,m2^{m\alpha})}\}$ [note that is different from
$A(m)$]. Hence, when using Markov's property at $T_{\partial
B(2^m,m2^{m\alpha})}$ the random walk is located in $\partial^+
B(2^m,m2^{m\alpha})$, so
\begin{eqnarray*}
&& \PR\bigl[D(n,m)\bigr]
\\
&&\qquad\leq  \sum_{x\in\partial^+ B(2^m,m2^{m\alpha})} {\mathbf E}\bigl[P^{\omega
}[X_{ T_{\partial B(2^{m},m2^{m\alpha})}}=x]\\
&&\qquad\quad\hspace*{75.4pt}{}\times P^{\omega}_x[
T_{x+\partial
B(2^{m+1},2^{(m+1)\alpha})}\neq T_{x+\partial
^+B(2^{m+1},2^{(m+1)\alpha
})}]\bigr]
\\
&&\qquad\leq  C m^d2^{dm\alpha} \\
&&\qquad\quad\hspace*{0pt}{}\times\max_{x\in\partial^+ B(2^{m},m2^{m\alpha})}
{\mathbf E}
\bigl[P^{\omega}_x[ T_{x+\partial B(2^{m+1},2^{(m+1)\alpha})}\neq
T_{x+\partial^+B(2^{m+1},2^{(m+1)\alpha})}]
\bigr]
\\
&&\qquad\leq C m^d2^{dm\alpha} \PR\bigl[ A(m+1)^c\bigr] \leq
C\exp\bigl(-c2^m\bigr)
\end{eqnarray*}
by translation invariance and Theorem \ref{BL}.

The lemma follows from the previous and (\ref{ab}).
\end{pf}

\subsection{Uniformly bounded chance of never backtracking at open points}

We recall that $\nu$ was defined at the beginning of Section \ref
{sectnotation}. We denote $\mathcal{C}=\{x> 0\}^{\nu}$, which is seen
as the configuration, that is, values of conductances, adjacent to a
point. For any $a\in\mathcal{C}$, we define the environment $\omega
_x^a$ on the edges of $\Z^d$ to be the environment which has the same
conductances as in $\omega$ except on edges adjacent to $x$, and on
these edges the conductances are given by the configuration $a$, that
is, $c^{\omega_x^a}_*([x,x+e])=a(e)$ for any $e\in\nu$.

We say that $a\in\mathcal{C}$ is $K$-open if $a(e)\in[1/K,K]$ for any
$e\in\nu$. In the sequel, we will use the notation $\max_{a\in
\mathcal
{C}\ \mathrm{open}}$ to designate the maximum taken over all
configurations $a\in\mathcal{C}$ that are open.

\begin{lemma}
\label{posescape}
We have
\[
{\mathbf E} \Bigl[ \max_{a\in\mathcal{C}\ \mathrm{open}}P^{\omega
^a_0}[D<\infty] \mid0
\mbox{ is good} \Bigr] <1.
\]
\end{lemma}

This result is natural. Indeed, by following the directed open path we
can bring the random walk far in the direction of the bias with a
positive probability independent of the environment, and after this
point it will be unlikely by Lemma \ref{thisisboring} to backtrack
past your starting point. This means that there is always a positive
escape probability from a good point.

\begin{pf}
Fix $n>0$. On the event that $\{0\mbox{ is good}\}$, we denote
$\mathcal
{P}(i)$ a directed path starting at $0$ where all points are open. We
denote $L_{\partial^+B(n,n^{\alpha})}=\inf\{ i, \mathcal{P}(i)\in
\partial^+B(n,n^{\alpha})\}$. Now, we see that if the two following
conditions are verified:
\begin{longlist}[(2)]
\item[(1)] $X_i=\mathcal{P}(i)$ for $i\leq L_{\partial
^+B(n,n^{\alpha})}$,
\item[(2)] $T_{\mathcal{H}^-(2)}\circ\theta_{T_{\mathcal
{P}(L_{\partial
^+B(n,n^{\alpha})})}} =\infty$,
\end{longlist}
then $D=\infty$.

We can see that if $\{0\mbox{ is good}\}$, then $L_{\partial
^+B(n,n^{\alpha})}\leq Cn$, so that
\[
\min_{a\in\mathcal{C}\ \mathrm{open}} P^{\omega_0^a}\bigl[X_i=
\mathcal{P}(i) \mbox{ for }i\leq L_{\partial^+B(n,n^{\alpha})}\bigr]
\geq\kappa_0^{Cn}
\]
by Remark \ref{fakeUE}.

In particular, we have
\begin{eqnarray*}
&& {\mathbf E} \Bigl[ \min_{a\in\mathcal{C}\ \mathrm{open}} P^{\omega
_0^a}[D=\infty]\mid0
\mbox{ is good} \Bigr]
\\
&&\qquad\geq  {\mathbf E} \Bigl[\min_{a\in\mathcal{C}\ \mathrm{open}} P^{\omega
_0^a}
\bigl[X_i=\mathcal{P}(i) \mbox{ for }i\leq L_{\partial
^+B(n,n^{\alpha
})}\bigr]
\\
&&\hspace*{78.3pt}{} \times P^{\omega_0^a}_{\mathcal{P}(L_{\partial^+B(n,n^{\alpha})})}[
T_{\mathcal{H}^-(2)} =\infty]\mid0\mbox{
is good} \Bigr]
\\
&&\qquad\geq  \kappa_0^{Cn} {\mathbf E} \Bigl[ \min
_{a\in\mathcal{C}\ \mathrm{open}} P^{\omega_0^a}_{\mathcal{P}(L_{\partial
^+B(n,n^{\alpha})})}[ T_{\mathcal
{H}^-(2)} =
\infty]\mid0\mbox{ is good} \Bigr].
\end{eqnarray*}

Moreover, since we will not use the edges adjacent to $0$ to know if we
hit $\mathcal{H}^-(2)$, we see that
\[
P^{\omega_0^a}_{\mathcal{P}(L_{\partial^+B(n,n^{\alpha})})}[ T_{\mathcal
{H}^-(2)} =\infty]=P^{\omega}_{\mathcal{P}(L_{\partial
^+B(n,n^{\alpha
})})}[
T_{\mathcal{H}^-(2)} =\infty],
\]
so that for any $n$
\begin{eqnarray*}
&& {\mathbf E} \Bigl[ \min_{a\in\mathcal{C}\ \mathrm{open}} P^{\omega
_0^a}[D=\infty]\mid0
\mbox{ is good} \Bigr]
\\
&&\qquad\geq  \kappa_0^{Cn} {\mathbf E} \bigl[P^{\omega}_{\mathcal
{P}(L_{\partial
^+B(n,n^{\alpha})})}[
T_{\mathcal{H}^-(2)} =\infty]\mid0\mbox{ is good} \bigr].
\end{eqnarray*}

Now,
\begin{eqnarray*}
&& {\mathbf E} \bigl[ P^{\omega}_{\mathcal{P}(L_{\partial^+B(n,n^{\alpha
})})}[ T_{\mathcal{H}^-(2)} <\infty]
\mid\mbox{0 is good} \bigr]
\\
&&\qquad\leq {\mathbf P}[\mbox{0 is good}]^{-1} {\mathbf E} \bigl[
P^{\omega
}_{\mathcal
{P}(L_{\partial^+B(n,n^{\alpha})})}[ T_{\mathcal{H}^-(2)} <\infty]
\bigr]
\\
&&\qquad\leq  C \PR[T_{\mathcal{H}^-(-n+2)} <\infty],
\end{eqnarray*}
where we use translation invariance.

Now, by Lemma \ref{thisisboring}, we see that the previous quantity
is less than $1/2$ for $n\geq n_0$. Hence combining the last two equations,
\[
{\mathbf E} \Bigl[ \min_{a\in\mathcal{C}\ \mathrm{open}} P^{\omega
_0^a}[D=\infty]\mid0
\mbox{ is good} \Bigr] \geq(1/2) \kappa_0^{Cn_0} >0,
\]
which implies the result.
\end{pf}

\subsection{Number of trials before finding an open ladder point which
is a regeneration time}

Let us introduce the collection of edges with maximum scalar product
with $\vec{\ell}$
\[
\mathcal{E}=\{ \mbox{$e\in\nu$ such that $e\cdot\vec{\ell} =e_1
\cdot\vec{\ell}$}\}
\]
and
%
\begin{equation}
\label{defbx}\qquad \mathcal{B}_x=\bigl\{e\in\mathbb{E}
\bigl(\Z^d\bigr), e=[-e_1,f-e_1] \mbox{ with
$f$ any unit vector of $\mathcal{E}$}\bigr\}.
\end{equation}

Imagining the bias is oriented to the right, the set of edges to the
``left'' of $x$ is defined to be
%
\begin{equation}
\label{defleft} \mathcal{L}^x:=\bigl\{[y,z]\in E
\bigl(\Z^d\bigr), y\cdot\ell\leq x \cdot\ell\mbox{ and } z\cdot\ell
\leq x \cdot\ell\bigr\}\cup\mathcal{B}_x,
\end{equation}
and the edges to the ``right'' are
%
\begin{equation}
\label{defright} \mathcal{R}^x:=\bigl\{[y,z]\in E
\bigl(\Z^d\bigr), y\cdot\ell> x \cdot\ell\mbox{ or } z\cdot\ell> x
\cdot\ell\bigr\}\cup\mathcal{B}_x.
\end{equation}

We recall that $N$ was defined at (\ref{deftau1}). Since each time we
arrive at a new $S_k$, we are at an open-ladder point with unspoiled
environment ahead of us, there will be a positive chance of being a
good point and never backtracking, which would create a regeneration
time. This means that $N$ should be similar to a geometric random
variables. Here, we will only prove that $N$ has exponential tails.
%
\begin{lemma}
\label{Kn}
We have
\[
\PR[N\geq n]\leq\exp(-cn).
\]
\end{lemma}

\begin{pf}
We introduce the event
\[
C(n):=\{\mbox{for all }k \leq n\mbox{ such that }S_k<\infty\mbox{,
we have } D\circ S_k+S_k<\infty\},
\]
which verifies
%
\begin{equation}
\label{fa} \{N\geq n\}\subseteq C(n).
\end{equation}

Because of the way our regeneration times are constructed, we can see
that $C(n)$ is $P^{\omega}$-measurable with respect to $\sigma\{X_k
\mbox{ with } k \leq S_{n+1}\}$; see the discussion below~(\ref
{deftau1}). Using Markov's property at $S_{n+1}$,
\begin{eqnarray*}
\PR\bigl[C(n+1)\bigr] &\leq& \sum_{x\in\Z^d} {\mathbf E}
\bigl[P^{\omega}\bigl[X_{S_{n+1}}=x,C(n)\bigr] P^{\omega}_x[D<
\infty]\bigr]
\\
&\leq& \sum_{x\in\Z^d} {\mathbf E}\Bigl[P^{\omega}
\bigl[X_{S_{n+1}}=x, C(n)\bigr] \max_{a\ \mathrm{open}}
P^{\omega^{x,a}}_x[D<\infty]\Bigr],
\end{eqnarray*}
where we used the fact that $X_{S_{n+1}}$ is open. Furthermore:
\begin{longlist}[(2)]
\item[(1)] $\!\!P^{\omega}[X_{S_{n+1}}\,{=}\,x, C(n)] $ is measurable with
respect to \mbox{$\sigma\{c_*(e) \mbox{ with }e\,{\in}\,\mathcal{L}_{x}\}$},
\item[(2)] $\!\!\max_{a\ \mathrm{open}} P^{\omega_x^a}_x[D\,{<}\,\infty]$ is
measurable with respect to \mbox{$\sigma\{c_*(e) \mbox{ with }
e\,{\notin}\,\mathcal{L}_{x}\}$}.
\end{longlist}

So we have ${\mathbf P}$-independence between the random variables in $(1)$
and in $(2)$. Hence
\begin{eqnarray*}
&& \PR\bigl[C(n+1)\bigr]
\\
&&\qquad\leq  \PR\bigl[C(n)\bigr] {\mathbf P} \Bigl[\max_{a\ \mathrm{open}}
P^{\omega
^{x,a}}_x[D<\infty] \Bigr]
\\
&&\qquad\leq  \PR\bigl[C(n)\bigr]\Bigl( {\mathbf P}[x\mbox{ is not good}]+{\mathbf E}
\Bigl[{
\mathbf1} {\{x\mbox{ is good}\}} \max_{a\in\mathcal{C}\
\mathrm{open}}
P^{\omega
^a_x}[D<\infty] \Bigr]\Bigr)
\\
&&\qquad\leq  \PR\bigl[C(n)\bigr] \Bigl(1-{\mathbf P}[0\mbox{ is good}] \Bigl
(1-{\mathbf E}
\Bigl[\max_{a\in\mathcal{C}\ \mathrm{open}} P^{\omega^a_0}[D<\infty]\mid
0\mbox{ is
good} \Bigr] \Bigr) \Bigr),
\end{eqnarray*}
where we used translation invariance. Furthermore, we can use
Lemma \ref
{posescape} to see that
\[
\PR\bigl[C(n+1)\bigr]\leq(1-c) \PR\bigl[C(n)\bigr]\leq\cdots
\leq(1-c)^n,
\]
hence, the result by (\ref{fa}).
\end{pf}

\subsection{Tails of regeneration times}

We have all the tools necessary to show that the first regeneration
time does not occur too far away from the origin.

\begin{theorem}
\label{tailtau}
For any $M<\infty$, there exists $K_0$ such that, for any $K\geq K_0$
we have $\tau_1^{(K)}< \infty$, $\PR$-a.s. and
\[
\PR[X_{\tau_1^{(K)}}\cdot\vec{\ell} \geq n] \leq C(M)n^{-M}.
\]
\end{theorem}

\begin{pf}
Recalling definitions (\ref{defS}) and (\ref{defW}), we may see that
$\{T_{\mathcal{H}^+(M_k)},\break k\geq0\} \subset\{W_k, k\geq0\}$. This
means that on $M(n)$, defined at (\ref{defM}),
\[
\mbox{for $k$ such that $S_{k} \leq\Delta_n$}\qquad \mbox{we have }
X_{S_{k+1}}\cdot\vec{\ell} -X_{T_{\mathcal{H}^+(M_k)}}\cdot\vec{\ell
}\leq
n^{1/2}.
\]

Moreover, on $S(n)$ [defined at (\ref{defSn})], we have
\[
\mbox{for $k$ such that $S_k\leq\Delta_n$ and
$M_k<\infty$}\qquad \mbox{we have } M_k-X_{S_k}\cdot\vec{
\ell}\leq n^{1/2}.
\]

Noticing that $X_{T_{\mathcal{H}^+(M_k)}}\cdot\vec{\ell} \leq M_k+1$,
we may see that, on $S(n)\cap M(n)$
\[
X_{S_{k+1}}\cdot\vec{\ell}-X_{S_k}\cdot\vec{\ell}
\leq2n^{1/2}+1
\]
for any $k$ with $S_{k}<\Delta_n$ and $M_k<\infty$. By induction, this
means that if $k\leq n^{1/2}/3$, $S_{k}<\Delta_n$ and $M_k<\infty$, then
\[
X_{S_{k+1}}\cdot\vec{\ell} \leq k\bigl(2n^{1/2}+1\bigr)<n
\quad\mbox{and}\quad S_{k+1}\leq\Delta_n,
\]
and the second part following from the fact that $X_{S_{k+1}}$ is a new
maximum for the random walk in the direction $\vec{\ell}$. In
particular, if $N\leq n^{1/2}/3$, then we can apply\vspace*{1pt} the previous
equation to $k=N-1$. Recalling (\ref{deftau1}) we see that, if $\{N
\leq n^{1/2}/3\}$ and $M(n)\cap S(n)$, then for $n$ large enough,
\[
X_{\tau_1}\cdot\vec{\ell} \leq\bigl(n^{1/2}/3\bigr)
\bigl(2n^{1/2}+1\bigr) < n.
\]

Thus
\begin{eqnarray*}
\PR[X_{\tau_1}\cdot\vec{\ell} \geq n] &\leq& \PR\bigl[N \geq
n^{1/2}/3\bigr] +\PR\bigl[M(n)^c\bigr] +\PR
\bigl[S(n)^c\bigr]
\\
&\leq& \exp\bigl(-cn^{1/2}\bigr) + 2n^{-M}
\leq3n^{-M}
\end{eqnarray*}
by Lemmas \ref{Sn}, \ref{Mn} and \ref{Kn}. This completes
the proof.
\end{pf}

\subsection{Fundamental property of regeneration times}\label{fundprop}

We are going to define the sequence $\tau_0:=0<\tau_1<\tau_2< \cdots
<\tau_k< \cdots$ of successive regeneration times. Using a slight abuse
of notation by viewing $\tau_k(\cdot,\cdot)$ as a function of a walk
and an environment, we can define previous sequence via the following procedure:
%
\begin{equation}
\label{regenstruct} \tau_{k+1}=\tau_1+
\tau_k\bigl(X_{\tau_1+ \cdot}-X_{\tau_1}, \omega( \cdot
+X_{\tau_1})\bigr),\qquad k\geq0,
\end{equation}
meaning that the $k+1$th regeneration time is the $k$th regeneration
time after the first one.

We set
\[
\mathcal{G}_{k}:=\sigma\bigl\{ \tau_1,\ldots,
\tau_k; (X_{\tau_k\wedge
m})_{m\geq0}; c_*(e) \mbox{ with } e
\in\mathcal{L}^{X_{\tau
_{k}}}\bigr\}.
\]

Let us introduce for any $x\in\Z^d$,
\[
a_x=\bigl(c_*\bigl([x-e_1,x-e_1+e]\bigr)
\bigr)_{e\in\mathcal{E}} =\bigl(c_*(e)\bigr)_{e\in\mathcal
{B}_x}\in[1/K,K]^{\mathcal{E}},
\]
recalling notation from (\ref{defbx}). For $a\in[1/K,K]^{\mathcal
{E}}$, we set
\[
{\mathbf P}_x^a =\delta_a \bigl(\bigl(c_*
\bigl([x-e_1,x-e_1+e]\bigr)\bigr)_{e\in\mathcal{E}} \bigr)
\otimes\int_{e\in E(\Z^d)\setminus\mathcal{B}_x} \otimes\, d {\mathbf P}\bigl
(c_*(e)\bigr),
\]
where $\otimes$ denotes the product of measures. We introduce the
associated annealed measure
\[
\PR_x^a ={\mathbf P}_x^a \times
P_x^{\omega}.
\]

In words, $\PR_x^a$ denotes the annealed measure for the walk started
at $x$ but where the conductances of the edges in $\mathcal{B}_x$ are
fixed and given by $a$. We will use the notation $\PR^a$ (resp., ${\mathbf
P}^a$) for $\PR_0^a$ (resp., ${\mathbf P}_0^a$).

We may notice that Theorem \ref{tailtau}, can easily be generalized
to become:
%
\begin{theorem}
\label{tailtau2}
For any $M<\infty$, there exists $K_0$ such that, for any $K\geq K_0$,
we have $\tau_1^{(K)}< \infty$, $\PR_0^a$-a.s. for $a\in
[1/K,K]^{\mathcal{E}}$ and
\[
\max_{a\in[1/K,K]^{\mathcal{E}}} \PR_0^a[X_{\tau_1^{(K)}}
\cdot\vec{\ell} \geq n] \leq Cn^{-M}.
\]
\end{theorem}

Similarly, we can turn Theorem \ref{BL} into:
%
\begin{theorem}
\label{BL2}
For $\alpha>d+3$,
\[
\max_{a\in[1/K,K]^{\mathcal{E}}} \PR_0^a[T_{\partial B(L,L^{\alpha
})}
\neq T_{ \partial^+ B(L,L^{\alpha})}] \leq C e^{-cL}.
\]
\end{theorem}

The fundamental properties of regeneration times are that:
\begin{longlist}[(2)]
\item[(1)] the past and the future of the random walk that has arrived
$X_{\tau_k}$ are only linked by the conductances of the edges in
$a_{X_{\tau_k}}$;
\item[(2)] the law of the future of the random walk has the same law as
a random walk under $\PR_0^{a_{X_{\tau_k}}}[ \cdot\mid D=\infty]$.
\end{longlist}

We recall that $\mathcal{R}^0$ was defined at (\ref{defright}). Let us
state a theorem corresponding to the previous heuristic.
%
\begin{theorem}
\label{thindep}
Let us fix $K$ large enough. Let $f$, $g$, $h_k$ be bounded functions
which are measurable with respect to $\sigma\{X_n\dvtx  n\geq0\}$, $\sigma
\{
c_*(e),e\in\mathcal{R}^0\}$ and $\mathcal{G}_k$, respectively. Then
for $a\in[1/K,K]^{\mathcal{E}}$,
\[
\ES^a\bigl[f(X_{\tau_k+\cdot}-X_{\tau_k})g\circ
t_{X_{\tau_k}}h_k\bigr]=\ES^a\bigl[h_k
\ES_0^{a_{X_{\tau_k}}}[fg\mid D=\infty]\bigr].
\]
\end{theorem}

A similar theorem was proved in \cite{Shen} (as Theorems 3.3 and 3.5).
In our context, the random variables studied ($\tau_1$, $D$ etc.) are
defined differently from the corresponding ones in \cite{Shen}.
Nevertheless, our notation was chosen so that we may prove Theorem \ref
{thindep} simply by following word for word the proofs of Theorems 3.3
and 3.5 in \cite{Shen}. The reader should start reading after Remark
3.2 in \cite{Shen} to have all necessary notation. To avoid any
possible confusion we point out that in~\cite{Shen}, the measures $\PR
$, $P_{x,\omega}$ and $P$ correspond, respectively, to the environment,
the quenched random walk and the annealed measure and that $\omega(b)$
denotes the conductances of an edge $b$.

We bring the reader's attention to the fact that we have not proved yet
that $\tau_k<\infty$, $\PR$-a.s. (or $\PR^a$-a.s. for any $a\in
[1/K,K]^{\mathcal{E}}$). We only know this for $k=1$. This is enough to
prove Theorem \ref{thindep} for $k=1$. Using Theorem \ref{thindep}
for $k=1$ and Theorem \ref{tailtau2}, we may show that $\tau
_2<\infty
$, $\PR$-a.s. (or $\PR^a$-a.s. for any $a\in[1/K,K]^{\mathcal{E}}$) and
thereafter obtain Theorem \ref{thindep} for $k=2$. Hence, we may
proceed by induction to prove Theorem \ref{thindep} alongside the
following result.

\begin{proposition}\label{taufinite}
Let us fix $K$ large enough. For any $k\geq1$, we have $\tau
_k^{(K)}<\infty$ $\PR$-a.s. (or $\PR^a$-a.s. for any $a\in
[1/K,K]^{\mathcal{E}}$).
\end{proposition}

We see that this implies Proposition \ref{dirtrans}, which states
directional transience in the direction $\vec{\ell}$ for the random walk.

As in \cite{Shen}, we may notice that a consequence of Theorem \ref
{thindep} is:
%
\begin{proposition}\label{prop}
Let
\[
\Gamma:= \N\times\Z^d \times[1/K,K]^{\mathcal{E}}
\]
with its canonical product $\sigma$-algebra, and let
$y_i=(j^i,z^i,a^i)\in\Gamma$, $i\geq0$. For $a\in[1/K,K]^{\mathcal
{E}}$ and $G\subset\Gamma$ measurable let also
\[
\tilde{R}_K(a; G):=\PR_0^a\bigl[\bigl(
\tau_1^{(K)},X_{\tau
_1^{(K)}},a_{X_{\tau
_1^{(K)}}}\bigr)\in G
\mid D=\infty\bigr].
\]
Then under $\PR$ the $\Gamma$-valued random variables (with $\tau_0=0$),
%
\begin{equation}
\label{defyk} Y_i^K:=(J_i,Z_i,A_i):=
\bigl(\tau_{i+1}^{(K)}-\tau_i^{(K)},X_{\tau
_{i+1}^{(K)}}-X_{\tau_i^{(K)}},a_{\tau_{i+1}^{(K)}}
\bigr),\qquad i\geq0,\hspace*{-28pt}
\end{equation}
define a Markov chain on the state space $\Gamma$, which has
transition kernel
\[
\PR[Y_{i+1}\in G \mid Y_0=y_0,\ldots,Y_i=y_i]=\tilde{R}_K
\bigl(a^i;G\bigr)
\]
and initial distribution
\[
\tilde{\Lambda}_K(G):=\PR\bigl[\bigl(\tau_1^{(K)},X_{\tau
_1^{(K)}},a_{X_{\tau
_1^{(K)}}}
\bigr)\in G\bigr].
\]

Similarly, on the state space $[1/K,K]^{\mathcal{E}}$, the random variables
%
\begin{equation}
\label{defA} A_i=a_{X_{\tau_{i+1}^{(K)}}},\qquad k\geq0,
\end{equation}
also define a Markov chain under $\PR$. With $a\in[1/K,K]^{\mathcal
{E}}$ and $B\subset[1/K,K]^{\mathcal{E}}$ measurable, its transition
kernel is
\[
R_K(a;B):=\PR_0^a[a_{X_{\tau_1^{(K)}}}\in B
\mid D=\infty]= \sum_{j\in
\N,z\in\Z^d} \tilde{R}_K
\bigl(a;(j,z,B)\bigr)
\]
and the initial distribution is
%
\begin{equation}
\label{deflambda} \Lambda_K(B):=
\PR[a_{X_{\tau_1^{(K)}}}\in B]= \tilde{\Lambda}_K\bigl((j,z,B)\bigr).
\end{equation}
\end{proposition}

Now let us quote Lemma 3.7 and Theorem 3.8 from \cite{Shen}, which
essentially states that the environment seen at regeneration times
converges exponentially fast to some measure $\nu_K$.
%
\begin{theorem}\label{thnu}
There exists a unique invariant distribution $\nu_K$ for the transition
kernel $R_K$. It verifies
\[
\sup_{a\in[1/K,K]^{\mathcal{E}}} \bigl\| R_K^m(a;
\cdot)-\nu_K(\cdot)\bigr\|_{\mathrm{var}} \leq
Ce^{-cm},\qquad m\geq0,
\]
where $\|\cdot\|_{
\mathrm{var}}$ denotes the total variation distance.

Further, this probability measure $\nu_K$ is invariant with respect to
the transition kernel $R$; that is, $\nu_K R_K= \nu_K$, and the Markov
chain $(A_k)_{k\geq0}$, defined in (\ref{defA}) with transition
kernel $R_K$ and initial distribution $\nu_K$ on the state space
$[1/K,K]^{\mathcal{E}}$ is ergodic.
Moreover, the initial distribution $\Lambda_K(\cdot)$ given in (\ref
{deflambda}) is absolutely continuous with respect to $\nu_K(\cdot)$.
\end{theorem}

\begin{theorem}\label{thnu2}
The distribution $\tilde{\nu}_K:=\nu_K \tilde{R}_K$ is the unique
invariant distribution for the transition kernel $\tilde{R}_K$. It verifies
\[
\sup_{a\in[1/M,M]^{\mathcal{E}}} \bigl\|\tilde
{R}_K^m(a; \cdot)-\tilde{\nu}_K(\cdot)
\bigr\|_{\mathrm{var}} \leq Ce^{-cm},\qquad m\geq0.
\]

With initial distribution equal $\tilde{\nu}_K$, the Markov chain
$(Y_k)_{k\geq0}$ defined in (\ref{defyk}) is ergodic. Moreover, the
law of the Markov chain $(Y_{k+1})_{k\geq0}$ under $\PR$ is absolutely
continuous with respect to the law of the chain with initial
distribution $\tilde{\nu}_K$.
\end{theorem}

The proofs in \cite{Shen} carry over to our context simply, once we
have shown the following Doeblin condition: there exists $c>0$ such
that for any $a\in[1/K,K]^{\mathcal{E}}$ and $B\subset
[1/K,K]^{\mathcal
{E}}$, we have
\[
R_K(a,B)\geq c \otimes_{\mathcal{E}}{\mathbf P}
\bigl[c_*\in B\mid c_* \in[1/K,K]\bigr].
\]

Before proceeding to the proof of this condition we need to introduce
one notation. A vertex $x\in\Z^d$ is called open-good if $x$ is good
in a configuration $\omega_x^a$ where $a$ is open (the corresponding
notations were introduced above Lemma \ref{posescape}). Note that we
do not need to specify the value of the conductances of the edges since
the event $\{x \mbox{ is good}\}$ is measurable with respect to the
event $\{y \mbox{ is open}\}$ for $y\in\Z^d$.

To prove this condition, we only need to describe an event where
$a_{X_{\tau_1^{(K)}}}\in B$ and $D=\infty$ which has a $\PR
^a$-probability comparable to $\otimes_{\mathcal{E}}{\mathbf P}[c_*\in
B\mid c_* \in[1/K,K]]$. The event we will consider is $X_1=e_1$,
$X_2=2e_1$ (with $e_1$ and $2e_1$ being open) and regenerating at time
$2$. In terms of equations this translates into
\begin{eqnarray*}
R_K(a;B) & = & \PR_0^a[a_{X_{\tau_1}}\in
B \mid D=\infty]
\\
& = &\frac1{\PR_0^a[D=\infty] } {\mathbf E}^a
\bigl[P_0^{\omega}[a_{X_{\tau
_1}}\in B, D=\infty]\bigr].
\end{eqnarray*}

The previous equation implies, using Remark \ref{fakeUE},
\begin{eqnarray*}
R_K(a;B) &\geq& c{\mathbf E}^a\bigl[P_0^{\omega}[X_1=e_1, X_2=2e_1, D
\circ\theta_2=\infty]\ldots
\\
&&\hspace*{22.8pt}  a_{X_2}\in B, 0\mbox{ and }e_1 \mbox{ are open
and } 2e_1 \mbox{ is good}\bigr]
\\
&\geq&  c \kappa_0^2{\mathbf E}^a\bigl[
P_{2e_1}^{\omega}[D=\infty],a_{2e_1}\in B, 0\mbox{ and }
e_1 \mbox{ is open and } 2e_1 \mbox{ is good}\bigr]
\\
&\geq&  c \kappa_0^2{\mathbf E}^a \Bigl[ \min
_{a\ \mathrm{open}}P_{2e_1}^{\omega
_0^a}[D=\infty],a_{2e_1}
\in B, 0,e_1, 2e_1 \mbox{ are open}
\\
&&\hspace*{152.5pt} \mbox{ and } 2e_1 \mbox{ is open-good} \Bigr]
\end{eqnarray*}
and seeing that:
\begin{longlist}[(2)]
\item[(1)] $ \min_{a\ \mathrm{open}}P_{2e_1}^{\omega_0^a}[D=\infty]$ and
$\{2e_1 \mbox{ is open-good}\}$ are measurable with respect to $\sigma
\{
c_*([y,z])$  with $(y-2e_1)\cdot\vec{\ell} >0$ and
$z\neq 2e_1\}$;
\item[(2)] $\{a_{2e_1}\in B, 0, e_1, 2e_1 \mbox{ are open}\}$ is
measurable with respect to $\sigma\{c_*([y,z])$ with
$(y-e_1)\cdot\vec{\ell} \leq0$ or $2e_1=z\}$, and
\end{longlist}
hence they are ${\mathbf P}$-independent so that
\begin{eqnarray*}
R_K(a;B) & \geq & c{\mathbf E}^a \Bigl[ \min
_{a\ \mathrm{open}}P_{2e_1}^{\omega
_0^a}[D=\infty],
2e_1 \mbox{ is open-good} \Bigr]
\\
&&{} \times{\mathbf P}_0^a[ a_{2e_1}\in B, 0,
e_1, 2e_1 \mbox{ are open}]
\\
&\geq& c{\mathbf P}_0^a[ a_{2e_1}\in B; 0,
e_1, 2e_1 \mbox{ are open}]
\end{eqnarray*}
by Lemma \ref{posescape} (since $\{2e_1 \mbox{ is good}\}\subset\{
2e_1 \mbox{ is open-good}\}$). Now by simple combinatorics we see that
\[
R_K(a;B) \geq c \otimes_{\mathcal{E}}{\mathbf P}
\bigl[c_*\in B\mid c_* \in[1/K,K]\bigr],
\]
which is the Doeblin condition we were looking for.

\section{Positive speed regime}
\label{sectposspeed}

Our aim for this section is to show:
%
\begin{theorem}
\label{tauu1}
If $E_*[c_*]<\infty$, we have
\[
\max_{a\in[1/K,K]^{\mathcal{E}}}\ES^a\bigl[\Delta_n{
\mathbf1} {\{0\in\mathrm{GOOD}_K\}}\bigr]\leq C(K)n
\]
for any $K\geq K_0$ for some $K_0$.
\end{theorem}

Used in combination with the existence of a law of large numbers
provided by the existence of a regeneration structure, this will allow
us to prove the positivity of the speed if $E_*[c_*]<\infty$. Let us
enumerate the key points for proving the previous result:
\begin{longlist}[(3)]
\item[(1)] The number of visits to a good point is bounded; see
Lemma \ref{backbone}. This limits the expected number of entries in a
trap to, roughly, the size of its border.
\item[(2)] The time spent during one visit to a trap is linked to the
size of the trap and the conductances in that trap; see Lemma \ref
{timetrap}. It is already known by Lemma \ref{BLsizeclosedbox} that
the size of traps is extremely small, so we will be able to neglect the
effect from the size of traps.
\item[(3)] The conductances in traps are, relatively, similar to usual
conductances. In particular, they do not have infinite expectation and
cannot force zero-speed; see Lemma \ref{condtrap}.
\end{longlist}

This reasoning allows us to say that, essentially, $\Delta_n$ should be
of the same order as the number of sites visited before $\Delta_n$
(i.e., of order $n$), since there is no local trapping. More precisely,
we get an upper-bound of $\ES[\Delta_n]$ in terms of the number of the
probability of reaching a point during the first regeneration time; see
Lemma~\ref{lkjh}. The last step of the proof is to estimate the
probability that we reach $x$ during the first regeneration time; see
Lemma~\ref{hittrap}.

We proceed to give the details associated with the previous outline.
First, we notice:
%
\begin{lemma}
\label{backbone}
For any $x\in\mathrm{GOOD}_K(\omega)$, we have
\[
E^{\omega}_x \Biggl[\sum_{i=0}^{\infty}
{\mathbf1} {\{X_i=x\}} \Biggr]\leq C(K)<\infty.
\]
\end{lemma}

\begin{pf}
We see that
\[
E^{\omega}_x \Biggl[\sum_{i=0}^{\infty}
{\mathbf1} {\{X_i=x\}} \Biggr]=\frac1 {P^{\omega
}_x
\bigl[T_x^+=\infty\bigr]}=\frac{\pi^{\omega}(x)}{C^{\omega
}(x\leftrightarrow
\infty)},
\]
where $C^{\omega}(x\leftrightarrow\infty)$ is the effective
conductance between $x$ and infinity in $\omega$. Since $x\in
\mathrm{GOOD}_K$, we can upper-bound $\pi^{\omega}(x)$ using Remark \ref
{fakeUE}, and we may use Rayleigh's monotonicity principle (see \cite
{LP}) to see that
\[
C^{\omega}(x\leftrightarrow\infty) \geq\frac cK \sum
_{i\geq0} c^{\omega}(p_i) \geq c \exp(2\lambda
x \cdot\vec{\ell}),
\]
where $(p_i)_{i\geq0}$ is a directed path of open points starting at
$x$. This yields the result.
\end{pf}

\subsection{Time spent in traps}

For $x \in\partial\mathrm{BAD}(\omega) $, we define $
\mathrm{BAD}^{\mathrm{s}}_x(K)=\{x\}\cup\bigcup_{y\sim x} \mathrm{BAD}_K(y)$ the union of all
bad areas adjacent to $x$. Following Lemma \ref{BLsizeclosedbox} is:
%
\begin{lemma}\label{tgb}
For $x \in\partial\mathrm{BAD}(\omega) $, we have that $
\mathrm{BAD}^{\mathrm{s}}_x(K)$ is finite ${\mathbf P}$-a.s. for $K$ large enough and
\[
{\mathbf P}_p\bigl[ W\bigl(\mathrm{BAD}^{\mathrm{s}}_x(K)
\bigr) \geq n\bigr] \leq C\exp\bigl(-\xi_1(K)n\bigr),
\]
where $\xi_1(K)\to\infty$ as $K$ tends to infinity.
\end{lemma}

By a slightly subtle use of the mean return time formula we are able to
obtain the following:
%
\begin{lemma}
\label{timetrap}
For any $x \in\partial\mathrm{BAD}(\omega) $ we have that
\[
E_x^{\omega}\bigl[T^+_{\mathrm{GOOD}(\omega)}\bigr]\leq C(K) \exp
\bigl(3\lambda\bigl\llvert\partial\mathrm{BAD}^{\mathrm{s}}_x(
\omega)\bigr\rrvert\bigr) \biggl(1+\sum_{e\in
E(\mathrm{BAD}^{\mathrm{s}}_x)}
c_*^{\omega}(e) \biggr).
\]
\end{lemma}

\begin{pf}
The first remark to be made is that since $x \in\partial
\mathrm{BAD}(\omega) \subset\mathrm{GOOD}(\omega) $, all $y\sim x $, then
$c_*([x,y])\in[1/K,K]$.

We introduce the notation $\mathrm{BAD}^{\mathrm{ss}}_x(K)=
\mathrm{BAD}^{\mathrm{s}}_x(K)\setminus\{x\}$.

Let us consider the finite network obtained by taking $
\mathrm{BAD}^{\mathrm{ss}}_x(\omega)\cup\partial\mathrm{BAD}^{\mathrm{ss}}_x(\omega)$ and merging all points of $ \partial
\mathrm{BAD}^{\mathrm{ss}}_x(\omega)$ (which contains $x$) to one point $\delta$. We denote
$\omega_{\delta}$ the resulting graph which is obviously finite by
Lemma \ref{tgb} and connected, since the different connected components
$\mathrm{BAD}_K(y)$ ($x\sim y$) are connected through $x$. We may apply
the mean return formula at $\delta$ (see \cite{LP} exercise 2.33) to
obtain that
\[
E_{\delta}^{\omega_{\delta}}\bigl[T_{\delta}^+\bigr]=2
\frac{\sum_{e\in
E(\omega
_{\delta})} c(e)}{\pi^{\omega_{\delta}}(\delta)}=2\frac{\sum_{e\in
E(\mathrm{BAD}^{\mathrm{ss}}_x)} c(e) +\pi^{\omega_{\delta}}(\delta
)}{\pi
^{\omega_{\delta}}(\delta)},
\]
where $\pi^{\omega_{\delta}}(\delta)=\sum_{e\in\partial_E \operatorname{BAD}^{\mathrm{ss}}_x(K)} c^{\omega}(e)$.

We know that $\delta$ was formed by merging only good (hence open)
points. Hence, for $y$ a neighbor of $\delta$ in $\omega_{\delta}$, we
have, by Remark \ref{fakeUE},
%
\begin{equation}
\label{condborder}\qquad c \exp\Bigl(2 \lambda\min
_{y\in\partial\mathrm{BAD}^{\mathrm
{ss}}_x(\delta
)}y \cdot\vec{\ell} \Bigr) \leq c^{\omega_{\delta}}\bigl([
\delta,y]\bigr)\leq C \exp\Bigl(2 \lambda\max_{y\in\partial\mathrm
{BAD}^{\mathrm
{ss}}_x(\delta
)}y \cdot
\vec{\ell} \Bigr),
\end{equation}
so that
we know that
%
\begin{eqnarray}
\label{measborder}
&&
c \exp\Bigl(2 \lambda\max
_{y\in\partial\mathrm{BAD}^{\mathrm
{ss}}_x(\delta
)} y\cdot\vec{\ell} \Bigr)\nonumber\\[-8pt]\\[-8pt]
&&\qquad\leq\pi^{\omega_{\delta}}(\delta)
\leq C\bigl\llvert\partial\mathrm{BAD}^{\mathrm{ss}}_x(\omega)\bigr
\rrvert\exp\Bigl(2 \lambda\max_{y\in\partial\mathrm{BAD}^{\mathrm
{ss}}_x(\delta)} y\cdot\vec{\ell}
\Bigr).\nonumber
\end{eqnarray}

Using (\ref{measborder}), we get
\[
\mbox{for $e\in E\bigl(\mathrm{BAD}^{\mathrm{ss}}_x\bigr)$}\qquad
\frac{c(e)}{\pi
^{\omega
_{\delta}}(\delta)} \leq C c_*(e),
\]
which means that
%
\begin{equation}
\label{meanreturn} E_{\delta}^{\omega_{\delta}}
\bigl[T_{\delta}^+\bigr]\leq C \sum_{e\in
E(\mathrm{BAD}^{\mathrm{ss}}_x)}
c_*(e)+C.
\end{equation}

The transition probabilities of the random walk in $\omega_{\delta}$ at
any point different from $\delta$ are the same as that of the walk in
$\omega$. This implies that
%
\begin{eqnarray}
\label{meantime1}\quad E_{\delta}^{\omega_{\delta}}
\bigl[T_{\delta}^+\bigr]&=&\sum_{y\sim\delta}
P^{\omega_{\delta}}_{\delta}[X_1=y]E_y^{\omega}[T_{\partial
\mathrm{BAD}^{\mathrm{ss}}_x}]
\nonumber\\
&=&\sum_{y\in\mathrm{BAD}^{\mathrm{ss}}_x, y\sim\partial
\mathrm{BAD}^{\mathrm
{ss}}_x} P^{\omega_{\delta}}_{\delta}[X_1=y]E_y^{\omega
}[T_{\partial
\mathrm{BAD}^{\mathrm{ss}}_x}]\nonumber\\[-8pt]\\[-8pt]
&\geq&\max_{y\in\mathrm{BAD}^{\mathrm{ss}}_x, y\sim\partial\mathrm
{BAD}^{\mathrm{ss}}_x} P^{\omega_{\delta}}_{\delta}[X_1=y]E_y^{\omega
}[T_{\partial\mathrm{BAD}^{\mathrm{ss}}_x}]\nonumber
\\
&\geq&\min_{y\in\mathrm{BAD}^{\mathrm{ss}}_x, y\sim\partial\mathrm
{BAD}^{\mathrm{ss}}_x} P^{\omega_{\delta}}_{\delta}[X_1=y]
\max_{y\in
\mathrm{BAD}^{\mathrm{ss}}_x, y\sim\partial\mathrm{BAD}^{\mathrm
{ss}}_x}E_y^{\omega}[T_{\partial\mathrm{BAD}^{\mathrm{ss}}_x}].\nonumber
\end{eqnarray}

Moreover by (\ref{condborder}) and (\ref{measborder}), we have
\[
P_{\delta}^{\omega_{\delta}}[X_1=y] =\frac{c^{\omega_{\delta
}}([\delta,y])}{\pi^{\omega_{\delta}}(\delta)}\geq c
\frac{ \exp(2 \lambda
\min_{y\in\partial\mathrm{BAD}^{\mathrm{ss}}_x}y\cdot\vec{\ell})}{ \llvert
\partial\mathrm{BAD}^{\mathrm{ss}}_x(\omega)\rrvert \exp(2
\lambda\max_{y\in
\partial\mathrm{BAD}^{\mathrm{ss}}_x} y\cdot\vec{\ell})},
\]
and, since $\mathrm{BAD}^{\mathrm{ss}}_x\cup\partial\mathrm{BAD}^{\mathrm
{ss}}_x$ is connected, we have
\[
\max_{y\in\partial\mathrm{BAD}^{\mathrm{ss}}_x(\delta)} y \cdot\vec{\ell
}-\min_{y\in\partial\mathrm{BAD}^{\mathrm{ss}}_x}
y \cdot\vec{\ell}\leq\bigl\llvert\partial\mathrm{BAD}^{\mathrm{ss}}_x(
\omega)\bigr\rrvert,
\]
so that
\[
\min_{y\in\mathrm{BAD}^{\mathrm{ss}}_x, y\sim\partial\mathrm
{BAD}^{\mathrm
{ss}}_x} P_{\delta}^{\omega_{\delta}}[X_1=y]
\geq c \bigl\llvert\partial\mathrm{BAD}^{\mathrm{ss}}_x(\omega)
\bigr\rrvert^{-1} \exp\bigl(-2\lambda\bigl\llvert\partial\mathrm
{BAD}^{\mathrm{ss}}_x(\omega)\bigr\rrvert\bigr).
\]

This, with (\ref{meanreturn}) and (\ref{meantime1}), and considering
the fact that $\partial\mathrm{BAD}^{\mathrm{ss}}_x\subset\mathrm{GOOD}$ yields
\[
\max_{y\in\mathrm{BAD}^{\mathrm{ss}}_x, y\sim\partial\mathrm
{BAD}^{\mathrm
{ss}}_x} E_y^{\omega}[T_{\mathrm{GOOD}(\omega)}]
\leq C \exp\bigl(3\lambda\bigl\llvert\partial\mathrm{BAD}^{\mathrm
{ss}}_x(
\omega)\bigr\rrvert\bigr) \biggl(1+\sum_{e\in E(\mathrm
{BAD}^{\mathrm{ss}}_x)}
c_*(e) \biggr).
\]

This implies that
\[
E_x^{\omega}\bigl[T_{\mathrm{GOOD}(\omega)}^+\bigr]\leq1+C \exp\bigl(3
\lambda\bigl\llvert\partial\mathrm{BAD}^{\mathrm{s}}_x(\omega)\bigr
\rrvert\bigr) \biggl(1+\sum_{e\in
E(\mathrm
{BAD}^{\mathrm{s}}_x)} c_*(e) \biggr).
\]
\upqed\end{pf}

\subsection{Conductances in traps}

Let us understand, partially, how the conductances in traps are conditioned.
%
\begin{lemma}
\label{condtrap}
Fix $a\in[1/K,K]^{\mathcal{E}}$. Take $n\geq0$ and $K\geq1$,
$F\subset E(\Z^d)$ such that $0\notin V(F)$, and $e \in F$. If
$E_*[c_*]<\infty$, then
\begin{eqnarray*}
&& {\mathbf E}^a\bigl[{\mathbf1} {\bigl\{E\bigl(\mathrm{BAD}^{\mathrm{s}}_x(K)
\bigr)=F\bigr\}} P^{\omega}[T_{V(F)} \leq\Delta_n]
c_*(e) \bigr]
\\
&&\qquad\leq  C(K) {\mathbf E}^a\bigl[{\mathbf1} {\bigl\{E\bigl(
\mathrm{BAD}^{\mathrm{s}}_x(K)\bigr)=F\bigr\} } P^{\omega
}[T_{V(F)}
\leq\Delta_n] \bigr].
\end{eqnarray*}

If $\lim\frac{\ln P_*[c_*>n]}{\ln n}=-\gamma$ with $\gamma<1$, then
for any $\epsilon>0$ we have
\[
{\mathbf E}^a\bigl[{\mathbf1} {\bigl\{E\bigl(\mathrm{BAD}^{\mathrm{s}}_x(K)
\bigr)=F\bigr\}} c_*(e)^{\gamma
-\epsilon} \bigr] \leq C(K){\mathbf E}^a
\bigl[{\mathbf1} {\bigl\{E\bigl(\mathrm{BAD}^{\mathrm{s}}_x(K)\bigr)=F
\bigr\}} \bigr].
\]
\end{lemma}


Let us recall that abnormal edges were defined in the beginning of
Section \ref{subsecbad}. We introduce the notation $\omega^{e \mathrm{
an}}$ to signify that for $e'\in E(\Z^d)\setminus\{e\}$,
\[
{\mathbf1} {\bigl\{e' \mbox{ is abnormal}\bigr\}}\bigl(
\omega^{e \mathrm{
an}}\bigr)={\mathbf1} {\bigl\{e' \mbox{ is
abnormal}\bigr\}}(\omega)
\]
and
\[
{\mathbf1} {\{e \mbox{ is abnormal}\}}\bigl(\omega^{e \mathrm{ an}}\bigr)=1.
\]

\begin{pf}
First, let us notice that if there exists $M$ such that ${\mathbf
P}^a[c_*<M]=1$, then we may obtain the first part of the lemma with
$C=M$. We will now assume that ${\mathbf P}^a[c_*>M]>0$ for any $M$.

Take $F\subset E(\Z^d)$ such that $x\in F$ and $0\notin F$. For any $e
\in F$, let us\break notice that on the event $\{c_*(e)>K\}$, $e$ is abnormal
which means that\break $E(\mathrm{BAD}^{\mathrm{s}}_x(\omega^{e\ \mathrm{
an}}))=E(\mathrm{BAD}^{\mathrm{s}}_x(\omega))$. This means that
%
\begin{eqnarray}
\label{ss1} && {\mathbf E}^a\bigl[{\mathbf1} {\bigl\{E\bigl(
\mathrm{BAD}^{\mathrm{s}}_x\bigr)=F\bigr\}} P^{\omega}[T_{V(F)}
\leq\Delta_n] c_*(e) \bigr]
\nonumber\\
&&\qquad\leq  K{\mathbf E}^a\bigl[{\mathbf1} {\bigl\{E\bigl(
\mathrm{BAD}^{\mathrm{s}}_x\bigr)=F\bigr\}} P^{\omega}[T_{V(F)}
\leq\Delta_n] \bigr]\nonumber
\\
&&\qquad\quad{} + {\mathbf E}^a\bigl[c_*(e) {\mathbf1} {\bigl\{ c_*(e)> K
\bigr\}} {\mathbf1} {\bigl\{E\bigl(\mathrm{BAD}^{\mathrm{s}}_x\bigr)=F
\bigr\}} P^{\omega
}[T_{V(F)} \leq\Delta_n]\bigr]\nonumber
\\
&&\qquad= K{\mathbf E}^a\bigl[{\mathbf1} {\bigl\{E\bigl(
\mathrm{BAD}^{\mathrm{s}}_x\bigr)=F\bigr\}} P^{\omega}[T_{V(F)}
\leq\Delta_n] \bigr]
\\
&&\qquad\quad{} + {\mathbf E}^a\bigl[c_*(e){\mathbf1} {\bigl\{c_*(e)> K\bigr
\}} {\mathbf1} {\bigl\{E\bigl(\mathrm{BAD}^{\mathrm{s}}_x\bigl(
\omega^{e\ \mathrm{
an}}\bigr)\bigr)=F\bigr\}}P^{\omega
}[T_{V(F)}
\leq\Delta_n]\bigr]\nonumber
\\
&&\qquad\leq  K{\mathbf E}^a\bigl[{\mathbf1} {\bigl\{E\bigl(
\mathrm{BAD}^{\mathrm{s}}_x\bigr)=F\bigr\}} P^{\omega}[T_{V(F)}
\leq\Delta_n] \bigr]\nonumber
\\
&&\qquad\quad{} + {\mathbf E}^a\bigl[c_*(e){\mathbf1} {\bigl\{ E\bigl(\mathrm
{BAD}^{\mathrm{s}}_x\bigl(\omega^{e\ \mathrm{an}}\bigr)\bigr)=F
\bigr\}}P^{\omega
}[T_{V(F)} \leq\Delta_n]\bigr].\nonumber
\end{eqnarray}

Using the fact that $c_*(e)$ is independent of $\{E(\mathrm{BAD}^{\mathrm
{s}}_x(\omega^{e\ \mathrm{an}}))=F\}$ and $P^{\omega}[T_{V(F)} \leq
\Delta
_n]$ since $0\notin V(F)$ and $e\in F$. Hence
%
\begin{eqnarray}
\label{ss2} && {\mathbf E}^a\bigl[c_*(e) {\mathbf1} {\bigl\{E\bigl(
\mathrm{BAD}^{\mathrm{s}}_x\bigl(\omega^{e\
\mathrm{an}}\bigr)\bigr)=F
\bigr\}}P^{\omega}[T_{V(F)} \leq\Delta_n]\bigr]
\nonumber\\[-8pt]\\[-8pt]
&&\qquad\leq  {\mathbf E}^a\bigl[c_*(e)\bigr] {\mathbf E}^a
\bigl[ {\mathbf1} {\bigl\{E\bigl(\mathrm{BAD}^{\mathrm
{s}}_x\bigl(
\omega^{e\ \mathrm{an}}\bigr)\bigr)=F\bigr\}}P^{\omega}[T_{V(F)}
\leq\Delta_n]\bigr].\nonumber
\end{eqnarray}

We can use this same independence property again to write
%
\begin{eqnarray}
\label{ss3} &&{\mathbf E}^a\bigl[ {\mathbf1} {\bigl\{E\bigl(
\mathrm{BAD}^{\mathrm{s}}_x\bigl(\omega^{e\ \mathrm
{ an}}\bigr)\bigr)=F
\bigr\}}P^{\omega}[T_{V(F)} \leq\Delta_n]\bigr]
\nonumber\\
&&\qquad\leq  \frac1 {{\mathbf P}^a\bigl[c_*(e)> K\bigr]}\nonumber\\
&&\qquad\quad{}\times {\mathbf
E}^a\bigl[ {\mathbf1} {\bigl\{c_*(e)> K\bigr\}} {\mathbf1} {\bigl\{E
\bigl(\mathrm{BAD}^{\mathrm{s}}_x\bigl(\omega^{e\
\mathrm{an}}\bigr)
\bigr)=F\bigr\}}P^{\omega}[T_{V(F)} \leq\Delta_n]
\bigr]
\\
&&\qquad\leq  \frac1 {{\mathbf P}^a\bigl[c_*(e)> K\bigr]} {\mathbf
E}^a\bigl[ {\mathbf1} {\bigl\{c_*(e)> K\bigr\}} {\mathbf1} {\bigl\{E
\bigl(\mathrm{BAD}^{\mathrm{s}}_x\bigr)=F\bigr\}}P^{\omega
}[T_{V(F)}
\leq\Delta_n]\bigr]\nonumber
\\
&&\qquad\leq  \frac1 {{\mathbf P}^a\bigl[c_*(e)> K\bigr]} {\mathbf
E}^a\bigl[{\mathbf1} {\bigl\{E\bigl(\mathrm{BAD}^{\mathrm{s}}_x
\bigr)=F\bigr\}}P^{\omega}[T_{V(F)} \leq\Delta_n]
\bigr].\nonumber
\end{eqnarray}

Putting together (\ref{ss1}), (\ref{ss2}) and (\ref{ss3}) proves the
first part of the result. The second part can be handled using exactly
the same techniques.
\end{pf}

\subsection{Correlations of hitting probabilities of traps and their shape}

We have:
%
\begin{lemma}\label{endthis}
Fix $a\in[1/K,K]^{\mathcal{E}}$ and $x\in\Z^d$. Take $n\geq0$ and
$F\subset E(\Z^d)$, for $K$ large enough, we have
\begin{eqnarray*}
&& {\mathbf E}^a\bigl[{\mathbf1} {\bigl\{E\bigl(\mathrm{BAD}^{\mathrm{s}}_x(K)
\bigr)=F\bigr\}} P^{\omega}[T_{V(F)} \leq\Delta_n]
\bigr]
\\
&&\qquad\leq  C\exp\bigl(\lambda\bigl\llvert\partial V(F)\bigr\rrvert\bigr)
{\mathbf
P}^{a}\bigl[E\bigl(\mathrm{BAD}^{\mathrm{s}}_x\bigr)=F
\bigr]\PR^a[T_{B_{\Z^d}^{\infty}(x,3\llvert
\partial V(F)\rrvert)}<\Delta_n],
\end{eqnarray*}
where $B_{\Z^d}^{\infty}(x,r)$ is the ball of center $x$ and radius $r$
in $\Z^d$ for the infinity norm $\|\cdot\|
_{\infty}$.
\end{lemma}
\begin{pf}
First, we introduce the set of vertices
\[
\partial^{\mathrm{good}}F =\{z\notin F \mbox{ where $z=y+e_1\pm
e_i$ with $y\in F$ and $i\leq d$}\},
\]
which implies that $\partial^{\mathrm{good}} \mathrm{BAD}^{\mathrm{s}}_x(K)$
is composed of good points.

Then we set $\partial^{\mathrm{good}}_{\mathrm{edge}} F$ to be the set of
edges $\{[z,z+e_1], [z+e_1,z+e_1+e_i]$ with $z+e_1+e_i\in
\partial^{\mathrm{good}} F$ and $i\leq d\}$. Those definitions
ensure that $\{E(\mathrm{BAD}^{\mathrm{s}}_x(K))=F\}$ is measurable with
respect to the conductances of the edges adjacent to $F$, the edges in
$\partial^{\mathrm{good}}_{\mathrm{edge}} F$ and the events $\{y$ is
good$\}$ for $y\in\partial^{\mathrm{good}}F$. These\vspace*{1pt} notations are
illustrated in Figure \ref{fig4}.

\begin{figure}

\includegraphics{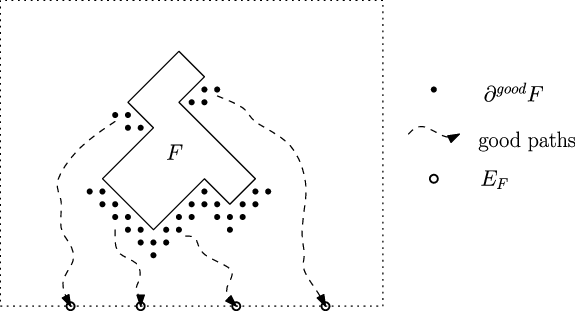}

\caption{The event $C_F$ depends only on what happens inside the dotted
box $E(B_{\Z^d}^{\infty}(x,3\llvert\partial V(F)\rrvert
))$. In this picture the
direction of $e_1$ is down.} \label{fig4}
\end{figure}

We extend the notion of $x$-open in the following manner: for any set
$A$ of edges, we say that a point $y$ is $A$-open in $\omega$, if $y$
is open in the environment coinciding with $\omega$ on all edges of $A$
and where all other edges are open.

Let us notice that if $E(\mathrm{BAD}^{\mathrm{s}}_x(K))=F$, and then we
have the two following events:
\begin{longlist}[(2)]
\item[(1)] $E(\mathrm{BAD}^{\mathrm{s}}_x(K))$ contains at least $F$, and
this can be determined by the conductances of the edges adjacent to
$y\in F$ and the edges in $\partial^{\mathrm{good}}_{\mathrm{edge}} F$;
\item[(2)] all vertices of $\partial^{\mathrm{good}} F$ are connected
with an $E(B_{\Z^d}^{\infty}(x,3\llvert\partial V(F)\rrvert
))$-open directed
path with even length to $\partial^+ B_{\Z^d}^{\infty}(x,3\llvert
\partial V(F)\rrvert):=\{z\in B_{\Z^d}^{\infty}(x,\break3\llvert\partial
V(F)\rrvert)$, with
$(z-x)\cdot e_1=3\llvert\partial V(F)\rrvert\}$. We denote
by $E_F$ the
endpoints of these paths in $\partial^+ B_{\Z^d}^{\infty}(x,3\llvert
\partial V(F)\rrvert)$.
\end{longlist}

We denote by $C_F$ the event which is the intersection of the two
previous events. The key properties of this definition are that:
\begin{longlist}[(2)]
\item[(1)] $C_F$ is measurable with respect to the conductances of the
edges in $E(B_{\Z^d}^{\infty}(x,3\llvert\partial V(F)\rrvert))$,
which meant
that it is ${\mathbf P}$-independent of
\[
P^{\omega}[T_{B_{\Z^d}^{\infty
}(x,3\llvert\partial V(F)\rrvert)}<\Delta_n];
\]
\item[(2)] on $C_F$, if the points in $E_F$ are $E(B_{\Z^d}^{\infty
}(x,3\llvert\partial V(F)\rrvert))^c$-good, then all points
in $\partial
^{\mathrm
{good}} F$ are good, and since $E(\mathrm{BAD}^{\mathrm{s}}_x(K))$ contains
at least $F$, we have $E(\mathrm{BAD}^{\mathrm{s}}_x(K))=F$.
\end{longlist}

Hence
\begin{eqnarray*}
&&
{\mathbf E}^a\bigl[{\mathbf1} {\bigl\{E\bigl(\mathrm{BAD}^{\mathrm{s}}_x(K)
\bigr)=F\bigr\}} P^{\omega
}[T_{V(F)} \leq\Delta_n]
\bigr]\\
&&\qquad\leq {\mathbf E}^a\bigl[{\mathbf1} {\{C_F\}}
P^{\omega}[T_{B_{\Z^d}^{\infty}(x,3\llvert\partial V(F)\rrvert)}<\Delta_n\bigr]
\\
&&\qquad = {\mathbf P}^a[C_F] \PR^a[T_{B_{\Z^d}^{\infty}(x,3\llvert
\partial V(F)\rrvert)}<
\Delta_n].
\end{eqnarray*}

Moreover, since $E_F$ can be injected into $\partial^{\mathrm{good}} F$,
we see that $\llvert E_F\rrvert\leq\llvert\partial
^{\mathrm{good}} F\rrvert\leq
C(d)\llvert\partial V(F)\rrvert$. Now, using the FKG
inequality (see \cite{Harris}) in
the first line and the fact that $\{E_F=E\}$ and $\{\mbox{the points in
}E\mbox{ are }E(B_{\Z^d}^{\infty}(x,\break3\llvert\partial
V(F)\rrvert))^c\mbox{-good}\}$
are ${\mathbf P}^a$-independent in the second, we obtain
\begin{eqnarray*}
&& {\mathbf P}[0\mbox{ is good}]^{C(d)\llvert\partial V(F)\rrvert
}{\mathbf P}^a[C_F]
\\
&&\qquad\leq  \sum_{E\subset\Z^d, \llvert E\rrvert\leq C(d)\llvert\partial
V(F)\rrvert} {\mathbf P}^a[E_F=E, C_F]\\
&&\hspace*{120.1pt}{}\times
{\mathbf P}^a\bigl[\mbox{points in }E\mbox{ are $E\bigl(B_{\Z
^d}^{\infty}
\bigl(x,3\bigl\llvert\partial V(F)\bigr\rrvert\bigr)\bigr)^c$-good}
\bigr]
\\
&&\qquad\leq  \sum_{E\subset\Z^d, \llvert E\rrvert\leq C(d)\llvert\partial
V(F)\rrvert} {\mathbf P}^a
\bigl[E_F=E, C_F\mbox{, points in }E\\
&&\qquad\quad\hspace*{108.2pt}\mbox{are $E
\bigl(B_{\Z^d}^{\infty
}\bigl(x,3\bigl\llvert\partial V(F)\bigr
\rrvert\bigr)\bigr)^c$-good}\bigr]
\\
&&\qquad\leq  \sum_{E\subset\Z^d, \llvert E\rrvert\leq C(d)\llvert\partial
V(F)\rrvert} {\mathbf P}^a
\bigl[E_F=E, E\bigl(\mathrm{BAD}^{\mathrm{s}}_x(K)
\bigr)=F\bigr]
\\
&&\qquad= {\mathbf P}^a\bigl[E\bigl(\mathrm{BAD}^{\mathrm{s}}_x(K)
\bigr)=F\bigr].
\end{eqnarray*}

For $K$ large enough, we have ${\mathbf P}[0\mbox{ is
$K$-good}]^{C(d)\llvert\partial V(F)\rrvert}\geq\exp
(-\lambda\llvert\partial V(F)\rrvert)$, and hence
using two previous equation, we see that
\begin{eqnarray*}
&& {\mathbf E}^a\bigl[{\mathbf1} {\bigl\{E\bigl(\mathrm{BAD}^{\mathrm{s}}_x(K)
\bigr)=F\bigr\}} P^{\omega}[T_{V(F)} \leq\Delta_n]
\bigr]
\\
&&\qquad\leq  C\exp\bigl(\lambda\bigl\llvert\partial V(F)\bigr\rrvert\bigr)
{\mathbf
P}^{a}\bigl[E\bigl(\mathrm{BAD}^{\mathrm{s}}_x\bigr)=F
\bigr]\PR^a[T_{B_{\Z^d}^{\infty}(x,3\llvert
\partial V(F)\rrvert)}<\Delta_n].
\end{eqnarray*}
\upqed\end{pf}

\subsection{\texorpdfstring{Proof of Theorem \protect\ref{tauu1}}{Proof of Theorem 8.1}}

We are now getting to the pivotal point in the proof where we control
$\Delta_n$ in terms of the expected number of sites encountered
before~$\Delta_n$.
%
\begin{lemma}\label{lkjh}
There exists $K_0$ such that, for any $K\geq K_0$, we have
\[
\max_{a\in[1/K,K]^{\mathcal{E}}}\ES^a\bigl[\Delta_n {
\mathbf1} {\{0\in\mathrm{GOOD}_K\}}\bigr]\leq C(K) \max
_{a\in[1/K,K]^{\mathcal{E}}} \ES^a[N_n],
\]
where $N_n= \llvert\{x\in\Z^d$ such that $T_x\leq\Delta
_n\} \rrvert$ is
the number of sites reached before~$\Delta_n$.
\end{lemma}

\begin{pf}
\textit{Step} 1: \textit{Decompose $\Delta_n$ into the time spent in traps and
outside of traps.}
If $0\in\mathrm{GOOD}(\omega) $, then a walk started at $0$ can only be
in a vertex of $\mathrm{BAD}(\omega) $ between visits to $\partial
\mathrm
{BAD}(\omega)$. Hence, on $\{0\in\mathrm{GOOD}(\omega)\}$,
\[
\sum_{x\in\mathrm{BAD}(\omega) } \sum_{i=0}^{\Delta_n}{
\mathbf1} {\{ X_i=x\}} \leq\sum_{x\in\partial\mathrm{BAD}(\omega) }
\sum_{i=0}^{\Delta
_n}{\mathbf1} {
\{X_i=x\}} T^+_{\mathrm{GOOD}(\omega)}\circ\theta_i.
\]

Hence, on $\{0\in\mathrm{GOOD}(\omega)\}$, since $\Z^d$ is partitioned
into two parts $\mathrm{GOOD}(\omega)$ and $\mathrm{BAD}(\omega)$ we have
%
\begin{eqnarray}
\label{tau}\qquad \Delta_n &\leq& \sum_{x\in\Z^d}
\sum_{i=0}^{\Delta_n}{\mathbf1} {\{
X_i=x\}}
\nonumber\\
&\leq & \sum_{x\in\mathrm{GOOD}(\omega)} \sum
_{i=0}^{\Delta
_n}{\mathbf1} {\{X_i=x\}} +
\sum_{x\in\mathrm{BAD}(\omega) } \sum_{i=0}^{\Delta_n}{
\mathbf1} {\{X_i=x\}}
\\
&\leq & \sum_{x\in\mathrm{GOOD}(\omega)} \sum
_{i=0}^{\Delta
_n}{\mathbf1} {\{X_i=x\}} +
\sum_{x\in\partial\mathrm{BAD}(\omega) } \sum_{i=0}^{\Delta
_n}{
\mathbf1} {\{X_i=x\}} T^+_{\mathrm{GOOD}(\omega)}\circ\theta_i.\nonumber
\end{eqnarray}

Hence, on $\{0\in\mathrm{GOOD}(\omega)\}$,
%
\begin{eqnarray}
\label{tau1} \Delta_n &\leq& \sum_{x\in\mathrm{GOOD}(\omega)}
{\mathbf1} {\{ T_x\leq\Delta_n\}} \sum
_{i=0}^{\infty} {\mathbf1} {\{X_i=x\}}
\nonumber\\[-8pt]\\[-8pt]
&&{} + \sum_{x\in\partial\mathrm{BAD}(\omega) } {\mathbf1} {
\{T_x\leq\Delta_n\}} \sum_{i=0}^{\infty}{
\mathbf1} {\{X_i=x\}} T_{\mathrm{GOOD}(\omega)}^+\circ\theta_i.\nonumber
\end{eqnarray}

\vspace*{9pt}

\textit{Step} 2: \textit{Reduce the problem to hitting probabilities and time of
excursions in traps.}
We can use Markov's property to say that for any $x\in\Z^d$, on $\{
0\in
\mathrm{GOOD}(\omega)\}$
\[
E^{\omega} \Biggl[{\mathbf1} {\{T_x\leq\Delta_n\}}
\sum_{i=0}^{\infty} {\mathbf1} {
\{X_i=x\}} \Biggr]=P^{\omega}[T_x \leq
\Delta_n]E^{\omega}_x \Biggl[\sum
_{i=0}^{\infty} {\mathbf1} {\{X_i=x\}}
\Biggr]
\]
and
\begin{eqnarray*}
&& E^{\omega} \Biggl[{\mathbf1} {\{T_x\leq\Delta_n
\}} \sum_{i=0}^{\infty} {\mathbf1} {
\{X_i=x\}}T_{\mathrm{GOOD}(\omega)}^+\circ\theta_i \Biggr]
\\
&&\qquad= P^{\omega}[T_x \leq\Delta_n]E^{\omega}_x
\Biggl[\sum_{i=0}^{\infty} {\mathbf1} {
\{X_i=x\}}T_{\mathrm{GOOD}(\omega)}^+\circ\theta_i \Biggr].
\end{eqnarray*}

This implies, using (\ref{tau1}), that on $\{0\in\mathrm{GOOD}(\omega
)\}$
\begin{eqnarray*}
E^{\omega}[\Delta_n] &\leq& \sum_{x\in\mathrm{GOOD}(\omega)}
P^{\omega
}[T_x \leq\Delta_n]E^{\omega}_x
\Biggl[\sum_{i=0}^{\infty} {\mathbf1} {
\{X_i=x\}} \Biggr]
\\
&&{} +\sum_{x\in\partial\mathrm{BAD}(\omega) } P^{\omega}[T_x \leq
\Delta_n] E^{\omega}_x \Biggl[\sum
_{i=0}^{\infty} {\mathbf1} {\{X_i=x\}
}T_{\mathrm
{GOOD}(\omega
)}^+\circ\theta_i \Biggr].
\end{eqnarray*}

Now we have
%
\begin{eqnarray}
\label{start} &&\max_{a\in[1/K,K]^{\mathcal{E}}}\ES^a\bigl[
\Delta_n{\mathbf1} {\{0\in\mathrm{GOOD}_K\}}\bigr]
\nonumber\\
&&\qquad\leq C\max_{a\in[1/K,K]^{\mathcal{E}}} \Biggl[ {\mathbf E}^a
\Biggl[\sum_{x\in\mathrm{GOOD}(\omega)} P^{\omega}[T_x
\leq\Delta_n]E^{\omega
}_x \Biggl[\sum
_{i=0}^{\infty} {\mathbf1} {\{X_i=x\}}
\Biggr] \Biggr]\nonumber
\\
&&\qquad\quad{}+\max_{a\in[1/K,K]^{\mathcal{E}}}{\mathbf E}^a \Biggl[{\mathbf1}
{\{0 \notin\mathrm{BAD}\}}\\
&&\hspace*{77.5pt}\qquad\quad{}\times\sum_{x\in\partial\mathrm{BAD}(\omega) }
P^{\omega}[T_x \leq\Delta_n]\nonumber\\
&&\hspace*{133.2pt}\qquad\quad{}\times E^{\omega}_x
\Biggl[\sum_{i=0}^{\infty} {\mathbf1} {\{
X_i=x\}}T_{\mathrm
{GOOD}(\omega)}^+\circ\theta_i \Biggr]
\Biggr] \Biggr].\nonumber
\end{eqnarray}

So that, using Lemma \ref{backbone},
\begin{eqnarray*}
&&\max_{a\in[1/K,K]^{\mathcal{E}}} \ES^a\bigl[\Delta_n{
\mathbf1} {\{0\in\mathrm{GOOD}_K\}}\bigr]
\\
&&\qquad\leq  C \Biggl[\max_{a\in[1/K,K]^{\mathcal{E}}}\ES^a \biggl[\sum
_{x\in\Z
^d} {\mathbf1} {\{T_x \leq
\Delta_n\}} \biggr]
\\
&&\hspace*{11.5pt}\qquad\quad{} +\max_{a\in[1/K,K]^{\mathcal{E}}} {\mathbf E}^a \Biggl[\sum
_{x\in
\partial\mathrm{BAD}(\omega) } P^{\omega}[T_x \leq
\Delta_n] \\
&&\hspace*{11.5pt}\qquad\hspace*{133.2pt}{}\times E^{\omega
}_x \Biggl[\sum
_{i=0}^{\infty} {\mathbf1} {\{X_i=x
\}}T_{\mathrm
{GOOD}(\omega
)}^+\circ\theta_i \Biggr] \Biggr] \Biggr].
\end{eqnarray*}

Let us focus, for now, on the second term which is the more difficult
to upper-bound. It corresponds to the time spent in the traps we
encounter in the first regeneration time. By Markov's property at $x\in
\partial\mathrm{BAD}(\omega)$,
%
\begin{eqnarray}
\label{cantstandthis}
&&
E^{\omega}_x \Biggl[\sum
_{i=0}^{\infty} {\mathbf1} {\{X_i=x\}
}T_{\mathrm
{GOOD}(\omega
)}^+\circ\theta_i \Biggr] \nonumber\\
&&\qquad = \sum
_{i=0}^{\infty} E^{\omega}_x\bigl[ {
\mathbf1} {\{X_i=x\}}\bigr] E^{\omega}_x
\bigl[T^+_{\mathrm{GOOD}(\omega)}\bigr]
\nonumber\\[-8pt]\\[-8pt]
&&\qquad= \Biggl(\sum_{i=0}^{\infty}
E^{\omega}_x\bigl[ {\mathbf1} {\{X_i=x\}}\bigr]
\Biggr) E^{\omega
}_x\bigl[T^+_{\mathrm{GOOD}(\omega)}\bigr]
\nonumber\\
&&\qquad\leq C E^{\omega}_x\bigl[T^+_{\mathrm{GOOD}(\omega)}
\bigr],\nonumber
\end{eqnarray}
where we used Lemma \ref{backbone}. Hence
%
\begin{eqnarray}
\label{step1} && \max_{a\in[1/K,K]^{\mathcal{E}}} {\mathbf E }^a \Biggl[{
\mathbf1} {\{0 \notin\mathrm{BAD}\}}\sum_{x\in\partial\mathrm{BAD}(\omega
) }
P^{\omega}[T_x \leq\Delta_n]\nonumber\\[-0.7pt]
&&\qquad\quad\hspace*{131.1pt}{}\times E^{\omega}_x
\Biggl[\sum_{i=0}^{\infty} {\mathbf1} {\{
X_i=x\}}T_{\mathrm
{GOOD}(\omega)}^+\circ\theta_i \Biggr]
\Biggr]
\nonumber\\[-0.7pt]
&&\qquad\leq  C \max_{a\in[1/K,K]^{\mathcal{E}}} {\mathbf E }^a \biggl[{
\mathbf1} {\{0 \notin\mathrm{BAD}\}}\sum_{x\in\partial\mathrm{BAD}(\omega)
}
P^{\omega}[T_x \leq\Delta_n] E^{\omega}_x
\bigl[T^+_{\mathrm{GOOD}(\omega)}\bigr] \biggr]\nonumber\hspace*{-25pt}
\\[-0.7pt]
&&\qquad\leq  C \max_{a\in[1/K,K]^{\mathcal{E}}} \sum
_{x\in\Z^d} {\mathbf E }^a \bigl[{\mathbf1} {\bigl\{x\in
\partial\mathrm{BAD}(\omega) \bigr\}} {\mathbf1} {\{0 \notin\mathrm{BAD}
\}}\\[-0.7pt]
&&\qquad\quad\hspace*{93.1pt}{}\times P^{\omega}[T_{x} \leq\Delta_n]
E^{\omega}_x\bigl[T^+_{\mathrm
{GOOD}(\omega)}\bigr] \bigr]\nonumber
\\[-0.7pt]
&&\qquad\leq  C\max_{a\in[1/K,K]^{\mathcal{E}}}\sum_{x\in\Z^d}
{\mathbf E }^a \bigl[{\mathbf1} {\bigl\{x\in\partial\mathrm{BAD}(\omega)
\bigr\}} {\mathbf1} {\{0 \notin\mathrm{BAD}\}}\nonumber\\
&&\qquad\quad\hspace*{93.1pt}{}\times P^{\omega
}[T_{\mathrm{BAD}^{\mathrm{s}}_x}
\leq\Delta_n] E^{\omega
}_x\bigl[T^+_{\mathrm
{GOOD}(\omega)}
\bigr] \bigr],\nonumber
\end{eqnarray}
where we used that for $x\in\partial\mathrm{BAD}(\omega)$, we have
$x\in\mathrm{BAD}^{\mathrm{s}}_x$ by definition.\vspace*{9pt}

\textit{Step} 3: \textit{Estimate the time spent during an excursion in a trap
by its size.}
Using Lemma \ref{timetrap},
%
\begin{eqnarray}
\label{step3} && {\mathbf E }^a \bigl[{\mathbf1} {\{0 \notin\mathrm{BAD}\}}
{\mathbf1} {\bigl\{ x\in\partial\mathrm{BAD}(\omega) \bigr\}} P^{\omega
}[T_{\mathrm
{BAD}^{\mathrm{s}}_x}
\leq\Delta_n] E^{\omega}_x\bigl[T^+_{\mathrm{GOOD}(\omega)}
\bigr] \bigr]
\nonumber\\[-0.7pt]
&&\qquad\leq  C{\mathbf E }^a \biggl[{\mathbf1} {\{0 \notin\mathrm{BAD}
\}} {\mathbf1} {\bigl\{x\in\partial\mathrm{BAD}(\omega) \bigr\}
}\nonumber\\[-0.7pt]
&&\qquad\quad\hspace*{24pt}{}\times P^{\omega}[T_{\mathrm
{BAD}^{\mathrm{s}}_x}
\leq\Delta_n] \exp\bigl(3\lambda\bigl\llvert\partial
\mathrm{BAD}^{\mathrm{s}}_x(\omega)\bigr\rrvert\bigr)\nonumber
\\[-0.7pt]
&&\qquad\quad\hspace*{95.5pt}{}\times \biggl(1+\sum_{e\in E(\mathrm{BAD}^{\mathrm{s}}_x)} c_*(e) \biggr)
\biggr]
\\[-0.7pt]
&&\qquad\leq \mathop{\sum_{ F\subset E(\Z^d) }}_{ x \in V(F), 0\notin V(F) }
C \exp\bigl(3\lambda\bigl\llvert\partial V(F)\bigr\rrvert\bigr) \nonumber\\[-0.7pt]
&&\qquad\quad\hspace*{59pt}{}\times \sum
_{e\in F} {\mathbf E }^a \bigl[{\mathbf1} {\bigl\{E\bigl(
\mathrm{BAD}^{\mathrm{s}}_x\bigr)=F\bigr\}} {\mathbf1} {\bigl\{x\in
\partial\mathrm{BAD}(\omega) \bigr\}}\nonumber
\\[-0.7pt]
&&\qquad\quad\hspace*{127pt}{}\times  P^{\omega}[T_{V(F)} \leq\Delta_n]
\bigl(1+c_*(e) \bigr) \bigr].\nonumber
\end{eqnarray}
Now by Lemma \ref{condtrap} (applied in the case $E_*[c_*]<\infty$)
and Lemma \ref{endthis},
%
\begin{eqnarray}
\label{step25}\quad && {\mathbf E }^a \bigl[{\mathbf1} {\bigl\{E\bigl(
\mathrm{BAD}^{\mathrm{s}}_x\bigr)=F\bigr\}} P^{\omega}[T_{V(F)}
\leq\Delta_n] c_*(e) \bigr]
\nonumber\\[-0.7pt]
&&\qquad\leq  C {\mathbf E }^a \bigl[{\mathbf1} {\bigl\{E\bigl(
\mathrm{BAD}^{\mathrm{s}}_x\bigr)=F\bigr\} } P^{\omega
}[T_{V(F)}
\leq\Delta_n] \bigr]
\\[-0.7pt]
&&\qquad\leq C \exp\bigl(\lambda\bigl\llvert\partial V(F)\bigr\rrvert\bigr)
{\mathbf P}^{a}\bigl[E\bigl(\mathrm{BAD}^{\mathrm{s}}_x
\bigr)=F\bigr]\PR^a[T_{B_{\Z^d}^{\infty}(x,3\llvert
\partial V(F)\rrvert)}<\Delta_n],\nonumber
\end{eqnarray}
so using (\ref{step3}), we have
\begin{eqnarray}
\label{step33} && {\mathbf E }^a \bigl[{\mathbf1} {\{0 \notin\mathrm{BAD}
\}} {\mathbf1} {\bigl\{ x\in\partial\mathrm{BAD}(\omega) \bigr\}}
P^{\omega}[T_{\mathrm
{BAD}^{\mathrm{s}}_x} \leq\Delta_n] E^{\omega}_x
\bigl[T^+_{\mathrm{GOOD}(\omega)}\bigr] \bigr]
\nonumber\\
&&\qquad\leq  C \mathop{\sum_{ F\subset E(\Z^d) }}_{ x \in V(F), 0\notin
V(F) }
\exp\bigl(4\lambda\bigl\llvert\partial V(F)\bigr\rrvert\bigr) {\mathbf
P}^{a}\bigl[E\bigl(\mathrm{BAD}^{\mathrm{s}}_x\bigr)=F
\bigr]\nonumber\\[-8pt]\\[-8pt]
&&\qquad\quad\hspace*{68.6pt}{}\times\PR^a[T_{B_{\Z^d}^{\infty}(x,3\llvert
\partial V(F)\rrvert)}<\Delta_n]\nonumber
\nonumber\\
&&\qquad= C \sum_{k\geq0} \exp(4\lambda k) {\mathbf
P}^{a}\bigl[\bigl\llvert\partial\mathrm{BAD}^{\mathrm{s}}_x
\bigr\rrvert=k\bigr]\PR^a[T_{B_{\Z
^d}^{\infty}(x,3k)}<\Delta_n],\nonumber
\end{eqnarray}
where we sum over the sets $F$ of same size in the last line.\vspace*{9pt}

\textit{Step} 4: \textit{Conclusion.}
When $x\notin\partial\mathrm{BAD}$ we use the notation $\mathrm
{BAD}^{\mathrm
{s}}_x=\partial\mathrm{BAD}^{\mathrm{s}}_x=\{x\}$. Using (\ref
{step1}), (\ref{step33}) and (\ref{start}) we obtain
%
\begin{eqnarray}
\label{step501} && \max_{a\in[1/K,K]^{\mathcal{E}}} \ES^a\bigl[
\Delta_n{\mathbf1} {\{ 0\in\mathrm{GOOD}_K\}}\bigr]
\nonumber\\
&&\qquad\leq  C\max_{a\in[1/K,K]^{\mathcal{E}}} \biggl[\sum
_{x\in\Z^d} \PR^a[T_{\mathrm{BAD}^{\mathrm{s}}_x} \leq
\Delta_n]\nonumber\\
&&\qquad\quad\hspace*{60pt}{}+ \sum_{x\in\Z
^d}{\mathbf E
}^a \bigl[ \exp\bigl(4\lambda\bigl\llvert\partial
\mathrm{BAD}^{\mathrm
{s}}_x\bigr\rrvert\bigr)P^{\omega}[T_{\mathrm{BAD}^{\mathrm{s}}_x}
\leq\Delta_n] \bigr] \biggr]\nonumber
\\
&&\qquad\leq  C\max_{a\in[1/K,K]^{\mathcal{E}}} \sum
_{x\in\Z^d}{\mathbf E }^a \bigl[ \exp\bigl(4\lambda\bigl
\llvert\partial\mathrm{BAD}^{\mathrm{s}}_x\bigr\rrvert
\bigr)P^{\omega
}[T_{\mathrm{BAD}^{\mathrm{s}}_x} \leq\Delta_n] \bigr]
\\
&&\qquad\leq  C\max_{a\in[1/K,K]^{\mathcal{E}}} \sum
_{x\in\Z^d} \sum_{k\geq
0} \exp(4\lambda
k) {\mathbf P}^{a}\bigl[\bigl\llvert\partial\mathrm{BAD}^{\mathrm{s}}_x
\bigr\rrvert=k\bigr]\PR^a[T_{B_{\Z^d}^{\infty
}(x,3k)}<\Delta_n]\nonumber
\\
&&\qquad\leq  C\max_{a\in[1/K,K]^{\mathcal{E}}} \max_{x\in\Z^d} \biggl[
\sum_{k\geq0} \biggl[ \exp(4\lambda k) {\mathbf
P}^{a}\bigl[\bigl\llvert\partial\mathrm{BAD}^{\mathrm{s}}_x
\bigr\rrvert=k\bigr] \nonumber\\
&&\qquad\quad\hspace*{99.3pt}{}\times\biggl(\sum_{x\in\Z^d}
\PR^a[T_{B_{\Z
^d}^{\infty
}(x,3k)}<\Delta_n] \biggr) \biggr]\biggr].\nonumber
\end{eqnarray}

Now, we can use the fact that $\PR^a[T_{B_{\Z^d}^{\infty
}(x,3k)}<\Delta
_n] \leq\sum_{y\in B_{\Z^d}^{\infty}(x,3k)} \PR^a[T_{y}<\Delta
_n]$ so
that $\sum_{x\in\Z^d} \PR^a[T_{B_{\Z^d}^{\infty}(x,3k)}<\Delta
_n] \leq
Ck^d \sum_{x\in\Z^d} \PR^a[T_{x}<\Delta_n]\leq\break Ck^d \ES^a[N_n] $. This
means that
%
\begin{eqnarray}
\label{step51} && \max_{a\in[1/K,K]^{\mathcal{E}}} \ES^a\bigl[
\Delta_n{\mathbf1} {\{ 0\in\mathrm{GOOD}_K\}}\bigr]
\nonumber\\
&&\qquad\leq  C \max_{a\in[1/K,K]^{\mathcal{E}}} \max_{x\in\Z^d}
\biggl[\sum_{k\geq0} k^d \exp(4\lambda k) {
\mathbf P}^{a}\bigl[\bigl\llvert\partial\mathrm{BAD}^{\mathrm{s}}_x
\bigr\rrvert=k\bigr] \biggr] \\
&&\qquad\quad\hspace*{0pt}{}\times\Bigl(\max_{a\in
[1/K,K]^{\mathcal
{E}}}\ES
^a[N_n] \Bigr)\nonumber
\\
&&\qquad\leq  C\max_{a\in[1/K,K]^{\mathcal{E}}} \max_{x\in\Z^d} {\mathbf
E}^a \bigl[ \exp\bigl(5\lambda\bigl\llvert\partial
\mathrm{BAD}^{\mathrm{s}}_x\bigr\rrvert\bigr) \bigr] \max
_{a\in
[1/K,K]^{\mathcal{E}}}\ES^a[N_n].\nonumber
\end{eqnarray}

The random variables $\mathrm{BAD}^{\mathrm{s}}_x$ are measurable with
respect to the events $\{y \mathrm{ is open}\}$ for $y\in\Z^d$. The
expectation is that ${\mathbf E}^a$ conditions only one such random
variable, so
\[
\max_{x\in\Z^d}\max_{a\in[1/K,K]^{\mathcal{E}}} {\mathbf E}^a
\bigl[ \exp\bigl(5\lambda\bigl\llvert\partial\mathrm{BAD}^{\mathrm{s}}_x
\bigr\rrvert\bigr) \bigr] \leq\frac1{{\mathbf P}[0\ \mathrm{is}\ \mathrm{open}]}\max
_{x\in\Z^d}{\mathbf E} \bigl[ \exp\bigl(5\lambda\bigl\llvert\partial
\mathrm{BAD}^{\mathrm{s}}_x\bigr\rrvert\bigr) \bigr]
\]
by translation invariance, and using Lemma \ref{tgb} the right-hand
side is finite for $K$ large enough. This and the two last equations
complete the proof.
\end{pf}

Let us estimate $\PR^a[T_x \leq\tau_1^{(K)}]$.
%
\begin{lemma}\label{hittrap}
Take $x\in\Z^d$, and then for any $M<\infty$, there exists $K_0$ such
that for any $K\geq K_0$,
\[
\max_{a\in[1/K,K]^{\mathcal{E}}}\PR^a\bigl[T_x \leq
\tau_1^{(K)} \bigr] \leq C \| x
\|_{\infty}^{-M}.
\]
\end{lemma}

\begin{pf}
Denote by $\chi$ the smallest integer so that $\{X_i, i\in[0,\tau
_1]\}
\subset B(\chi,\chi^{\alpha})$. First let us notice that
%
\begin{eqnarray}
\label{expostep12}\qquad
&&
\max_{a\in[1/K,K]^{\mathcal{E}}}
\PR^a[\chi\geq k ] \nonumber\\
&&\qquad\leq\max_{a\in
[1/K,K]^{\mathcal{E}}}
\PR^a[X_{\tau_1}\cdot\vec{\ell}\geq k]
\\
&&\qquad\quad{} +\max_{a\in[1/K,K]^{\mathcal{E}}}\PR^a\Bigl[X_{\tau_1}
\cdot\vec{\ell} < k, \max_{0\leq i,j \leq\tau_1} \max_{l\in[2,d]}
\bigl\llvert(X_j-X_i)\cdot f_l\bigr
\rrvert\geq k^{\alpha}\Bigr].\nonumber
\end{eqnarray}

We can upper-bound the first term as follows:
\[
\max_{a\in[1/K,K]^{\mathcal{E}}}\PR^a[X_{\tau_1^{(K)}}\cdot\vec{\ell
}\geq k]\leq C k^{-M}
\]
for any $M$ by choosing $K$ large enough by Theorem \ref{tailtau2}.

The second term can be upper-bounded with the following reasoning: on
the event that $X_{\tau_1}\cdot\vec{\ell}\leq k$ and $\max_{j\neq
1}\max_{0\leq j_1,j_2 \leq\tau_{1}} \llvert
(X_{j_1}-X_{j_2})\cdot f_j \rrvert
\geq k^{\alpha}$,\vspace*{1pt} $X_n$ does not exit the box $B(k,k^{\alpha})$
through $\partial^+ B(k,k^{\alpha})$, this means
\begin{eqnarray*}
&& \max_{a\in[1/K,K]^{\mathcal{E}}}\PR^a\Bigl[X_{\tau_1}\cdot
\vec{\ell} < k, \max_{0\leq i,j \leq\tau_1} \max_{l\in[2,d]} \bigl
\llvert(X_j-X_i)\cdot f_l\geq
k^{\alpha}\bigr\rrvert\Bigr]
\\
&&\qquad\leq  \max_{a\in[1/K,K]^{\mathcal{E}}}\PR^a[T_{\partial
B(k,k^{\alpha
})}\neq
T_{\partial^+ B(k,k^{\alpha})}]\leq ce^{-ck}
\end{eqnarray*}
by Theorem \ref{BL2}.

This turns (\ref{expostep12}) into
%
\begin{equation}
\label{expostep15} \max_{a\in[1/K,K]^{\mathcal{E}}}
\PR^a[\chi\geq k]\leq Ck^{-M}
\end{equation}
for any $M$ for $K$ large enough.

Now assume that $\{T_{x} \leq\tau_1\}$ then $\chi\geq
\| x\|_{\infty}^{1/\alpha}$. Hence,
by the previous equation
\[
\max_{a\in[1/K,K]^{\mathcal{E}}}\PR^a[T_{x} \leq
\tau_1] \leq C \| x\|
_{\infty
}^{-M/\alpha},
\]
which proves the lemma, since $\alpha$ is fixed and $M$ is arbitrary.
\end{pf}

We can now prove Theorem \ref{tauu1}.
\begin{pf}
Let us introduce $\tilde{N}_{\tau_i}$ the number of different sites
seen between $\tau_{i-1}$ and $\tau_i$ for any $i\geq1$.
By Theorem \ref{thindep} and choosing $K\geq K_0$ so that we may apply
Lemma \ref{hittrap} with $M>2d$, we can see that for $i\geq1$
\begin{eqnarray*}
\max_{a\in[1/K,K]^{\mathcal{E}}} \ES^a[\tilde{N}_{\tau_i}] &\leq&
\sum_{x\in\Z^d} \max_{a\in[1/K,K]^{\mathcal{E}}}
\PR^a[T_x<\tau_1\mid D=\infty]\\
&\leq&\sum
_{x\in\Z^d} C \| x\|
_{\infty}^{-M/2}<C,
\end{eqnarray*}
where $C$ does not depend on $i\geq1$. A similar inequality holds for $i=0$.

Since $\Delta_n\leq\tau_n= \sum_{i=1}^{n} (\tau_{i}-\tau_{i-1})$, we
see that $N_n\leq\sum_{i=1}^n {\tilde{N}}_{\tau_i}$ which means that
\[
\max_{a\in[1/K,K]^{\mathcal{E}}} \ES^a[N_n]\leq\sum
_{i=1}^n \max_{a\in
[1/K,K]^{\mathcal{E}}}
\ES^a[{\tilde{N}}_{\tau_i}]\leq Cn.
\]

Then, using Lemma \ref{lkjh}, we can prove Theorem \ref{tauu1}.
\end{pf}

\subsection{Law of large numbers}

We can use exactly the same type of proof as in \cite{Shen} to obtain:
%
\begin{proposition}\label{LLN1}
If $E_*[c_*]<\infty$, then there exists $K_0$ such that for any $K\geq
K_0$ we have
\[
\frac{X_n}n \to v =\frac{\ES^{\Pi_K}[X_{\tau_1^{(K)}}]}{\ES^{\Pi
_K}[\tau
_1^{(K)}]} ,\qquad\mbox{$\PR$-a.s. with } v\cdot
\vec{\ell}>0,
\]
where
\[
\Pi_K[\cdot]=\int\nu_K(da)\PR_0^a[
\cdot\mid D=\infty] \quad\mbox{and}\quad \ES^{\Pi_K}[\cdot]:= \int
\nu_K(da) \ES^a_0[ \cdot\mid D=\infty],
\]
where $\nu_K$ is the unique invariant distribution on
$[1/K,K]^{\mathcal
{E}}$ given in Theorem~\ref{thnu}.
\end{proposition}

Let us explain quickly where this result comes from. The regeneration
structure will allow us to obtain a law of large numbers for $X_n/n\to
v$ and $\Delta_n/n\to C_{\infty}$. By a standard inversion argument we
will see that $v \cdot\vec{\ell}=1/C_{\infty}$. Then Theorem~\ref
{tauu1} will imply that $C_{\infty}$ cannot be infinite which means
that $X_n/n$ cannot go to $0$.

\begin{pf}
We will now follow the strategy of proof of Theorem 5.1 in \cite{Shen}
to obtain the result.

First, we notice that by (\ref{expostep15}), we have $\ES^{\Pi
}[\llvert X_{\tau_1}\rrvert]<\infty$. 

Recalling the notations of Proposition \ref{prop}, we know that $Y_i^K$
with initial distribution $\tilde{\nu}$ is stationary and ergodic and
that the law of $(Y_i^K)_{i\geq1}$ under $\PR$ is absolutely continuous
with respect to the law with initial distribution $\tilde{\nu}$.
Therefore using Birkhoff's ergodic theorem (page 341 of \cite{durrett1})
we obtain that for any $f\in L^1(\Gamma,\tilde{\nu})$,
\[
\frac1n \sum_{k=1}^nf(Y_k)
\to_{n\to\infty} \int f d\tilde{\nu},\qquad \mbox{$\PR$-a.s.}
\]

Applying the previous formula to the functions $f(y)=j$ and $f(y)=x$
where $y=(j,z,a)\in\Gamma$, we obtain $\PR$-a.s.
\[
\frac1{n-1}\bigl(\tau_n^{(K)}-\tau_1^{(K)}
\bigr) \to_{n\to\infty} \int\ES^a\bigl[\tau_1^{(K)}
\mid D=\infty\bigr] \nu(da)=\ES^{\Pi_K}\bigl[\tau_1^{(K)}
\bigr]
\]
and
\[
\frac1{n-1}(X_{\tau_n^{(K)}}-X_{\tau_1^{(K)}}) \to_{n\to\infty} \int\ES
^a[X_{\tau_1^{(K)}}\mid D=\infty] \nu(da)=\ES^{\Pi_K}[X_{\tau_1^{(K)}}].
\]

Moreover recalling that, by Proposition \ref{taufinite}, we have
$\tau
_1^{(K)}<\infty$, $\PR$-a.s. we actually see that $\PR$-a.s.
\[
\frac1n \tau_n^{(K)} \to_{n\to\infty}
\ES^{\Pi_K}\bigl[\tau_1^{(K)}\bigr] \quad\mbox{and}\quad
\frac1n X_{\tau_n^{(K)}} \to_{n\to\infty} \ES^{\Pi_K}[X_{\tau_1^{(K)}}].
\]

Now, we may introduce $k_n$, a nondecreasing sequence going to infinity
such that $\tau_{k_n}\leq n<\tau_{k_{n+1}}$. By dividing this equation
by $k_n$ we see that the previous estimate implies that $\PR$-a.s.
\[
\frac{k_n}n \to\frac1 { \ES^{\Pi_K}[\tau_1^{(K)}]}.
\]

Furthermore, we observe that $\PR$-a.s.
\[
\frac{X_n}n =\frac{X_{\tau_{k_n}}}n+\frac{X_n-X_{\tau_{k_n}}}n,
\]
which in view of the two previous equations implies that $\PR$-a.s.
\[
\frac{X_{\tau_{k_n}}}n=\frac{X_{\tau_{k_n}}}{k_n}\frac{k_n}n \to_{n\to
\infty}
\frac{ \ES^{\Pi_K}[X_{\tau_1^{(K)}}]} { \ES^{\Pi_K}[\tau
_1^{(K)}]},
\]
where the right-hand side is always well defined (even if $\ES^{\Pi
}[\tau_1]=\infty$), since $\ES^{\Pi_K} [\llvert X_{\tau
_1^{(K)}}\rrvert
]\in(0,\infty)$. The last assertion follows from Theorem \ref
{tailtau2} and the fact that $X_{\tau_1}\cdot\vec{\ell}>2/\sqrt{d}$.

Moreover we have $\PR$-a.s.
\[
\frac{\llvert X_n-X_{\tau_{k_n}}\rrvert}n \leq\frac{\tau
_{k_{n}+1}-\tau_{k_n}}n = \frac{\tau_{k_{n}+1}}{k_{n}+1}
\frac{k_{n}+1}n-\frac{\tau_{k_n}}{k_n} \frac{k_n}n \to_{n\to\infty} 0.
\]

The three last equations imply that $\PR$-a.s.
%
\begin{equation}
\label{ozx} \lim_{n\to\infty}\frac{X_n}n = v =
\frac{\ES^{\Pi}[X_{\tau
_1}]}{\ES
^{\Pi}[\tau_1]},
\end{equation}
even in the case where $\ES^{\Pi}[\tau_1]=\infty$.

We introduce $k_n'$, a nondecreasing sequence going to infinity such
that $\tau_{k_n'}\leq\Delta_n<\tau_{k_{n}'+1}$. This means that
\[
X_{\tau_{k_n'}}\cdot\vec{\ell} \leq X_{\Delta_n}\cdot\vec{\ell}
<X_{\tau_{k_{n}'+1}}\cdot\vec{\ell}
\]
and in particular,
\[
X_{\tau_{k_n'}}\cdot\vec{\ell} \leq n+1 \quad\mbox{and}\quad X_{\tau
_{k_{n}'+1}}\cdot
\vec{\ell}\geq n.
\]
Dividing the left equation by $\tau_{k_n'}$ and the right one by $\tau
_{k_{n}'+1}$ we can take the limit as $n$ goes to infinity using (\ref
{ozx}), which yields $\PR$-a.s.
\[
\limsup_{n\to\infty} \frac n {\tau_{k_{n}'+1}} \leq
\frac{\ES
^{\Pi
}[X_{\tau_1}\cdot\vec{\ell}]}{\ES^{\Pi}[\tau_1]} \quad\mbox{and}\quad \liminf
_{n\to\infty} \frac n{
\tau_{k_{n}'}} \geq\frac{\ES^{\Pi
}[X_{\tau
_1}\cdot\vec{\ell}]}{\ES^{\Pi}[\tau_1]}.
\]

Now, we see that
\begin{eqnarray*}
\frac{\tau_{k_{n+1}'}-\tau_{k_n'}}n & = & \frac{\tau
_{k_{n}'+1}}{k_{n}'+1} \frac{k_{n}'+1}n-\frac{\tau_{k_n'}}{k_n'}
\frac
{k_n'}n
\\
&=& \bigl(\ES^{\Pi}[\tau_1]+o(1) \bigr) \biggl(
\frac{k_n'}n+o(1) \biggr)- \bigl(\ES^{\Pi}[\tau_1]+o(1)
\bigr)\frac{k_n'}n
\\
&=& \biggl(\frac{k_n'}n+o(1) \biggr)o(1),
\end{eqnarray*}
where the right-hand side goes to zero because $k_n'\leq n$ (since
$\Delta_n<\tau_{n}$), so that the two previous equations imply that
\[
\lim_{n\to\infty} \frac n {\tau_{k_{n}'}} =
\frac{\ES^{\Pi
}[X_{\tau
_1}\cdot\vec{\ell}]}{\ES^{\Pi}[\tau_1]}.
\]

Recalling the definition of the sequence $k_n'$, this implies $\PR$-a.s.
%
\begin{equation}
\label{ozxc} \lim_{n\to\infty} \frac{\Delta_n}n =
\frac{\ES^{\Pi}[\tau
_1]}{\ES^{\Pi
}[X_{\tau_1}\cdot\vec{\ell}]},
\end{equation}
even in the case where $\ES^{\Pi}[\tau_1]=\infty$.

If $E_*[c_*]<\infty$, then Theorem \ref{tauu1} (and the fact that $\{
0\in\mathrm{GOOD}_K\}$ has positive probability) imply that the almost
sure limit in (\ref{ozxc}) cannot be infinite. This means that $\ES
^{\Pi
}[\tau_1]<\infty$. Since $X_{\tau_1}\cdot\vec{\ell}>2/\sqrt{d}$, this
means that (\ref{ozx}) implies that \mbox{$v \cdot\vec{\ell} > 0$}.
\end{pf}

\begin{remark}
Interestingly, we do not know of any direct way of showing that $\ES
^{\Pi}[\tau_1]<\infty$.
\end{remark}

\section{Zero-speed regime}
\label{sectzerospeed}
\subsection{Characterization of the zero-speed regime}

We set $\mathcal{A}$ to be the set of vertices: $0$, $e_1$, $e_1\pm
e_i$, $2e_1 \pm e_i$, $2e_1\pm2e_i$, $3e_1\pm2e_i$, $3e_1\pm e_i$,
$4e_1\pm e_i$, for all $i\in[2,d]$, $5e_1$, $6e_1$ and the events
\[
A=\{\mbox{any $x\in\mathcal{A}$ is $6e_1$-open}\} \quad\mbox{and}\quad B=\{
6e_1 \mbox{ is good}\},
\]
where we recall that a vertex is called $x$-open if it is open in
$\omega_x$ coinciding with $\omega$ on all edges, but those that are
adjacent to $x$ which are normal in~$\omega_x$.

Note that $A$ and $B$ are independent and independent of $c_*([2e_1,3e_1])$.

\begin{lemma}
\label{LBa}
If $E_*[c_*]=\infty$, then $\min_{a\in[1/K,K]^{\mathcal{E}}}
E^a[\tau
_1\mid D=\infty] =\infty$.
\end{lemma}

The typical configuration that will slow the walk down is depicted in
Figure \ref{fig5}: the walk is likely to reach the edge $[2e_1,3e_1]$ and then
%
\begin{figure}

\includegraphics{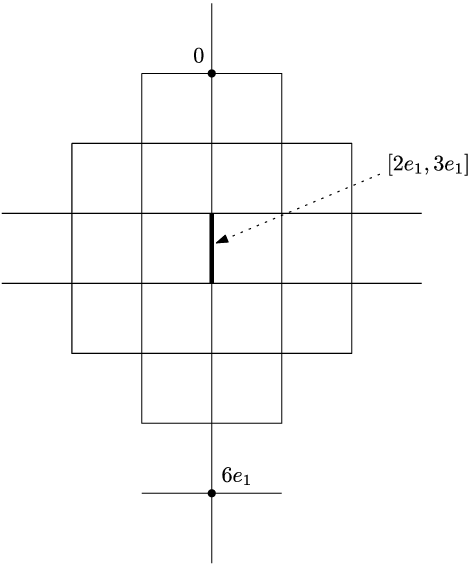}

\caption{The typical configurations slowing the walk down. All the
edges depicted are normal and the bold one has unusually high
conductance.} \label{fig5}
\end{figure}
stay there for a long time. More precisely, we will consider the
following chain of events on $A\cap B$:
\begin{longlist}[(3)]
\item[(1)] $X_1=e_1, X_2=2e_1, X_3=e_1, X_4=2e_1$, which forces $\tau
_1$ to be after time~4.
\item[(2)] From there we go back and forth on $[2e_1,3e_1]$, the
quenched expected time spent on this edge is of order
$c_*([2e_1,3e_1])$. With the previous point, this implies that $\tau_1$
is larger than the time spent on $[2e_1,3e_1]$ and so has, in some
sense, an infinite annealed expectation.
\item[(3)] After leaving the edge $[2e_1,3e_1]$ the walk goes to $6e_1$
before reaching any point of $\partial A$ and from there never backtracks.
\end{longlist}

In the series of events described above, we have that $D=\infty$, and
in some sense, we have $\tau_1$ has infinite expectation. Moreover,
under the quenched measure all the three events described are independent.

\begin{pf}
We have
%
\begin{equation}
\label{zf} \ES^a[\tau_1 \mid D=\infty]\geq
\ES^a\bigl[{\mathbf1} {\{D=\infty\} }\tau_1\bigr] \geq
\ES^a\bigl[{\mathbf1} {\{A,B\}} {\mathbf1} {\{D=\infty\}}
\tau_1\bigr].
\end{equation}

We may notice that on $A\cap B$, if:
\begin{longlist}[(3)]
\item[(1)] $X_1=e_1, X_2=2e_1,X_3=e_1, X_4=2e_1$ (hence $\tau_1\geq4$),
\item[(2)] $T_{6e_1}\circ\theta_{T_{\Z^d\setminus\{2e_1,3e_1\}
}\circ
\theta_4}\leq T_{\partial(\mathcal{A}\setminus\{0\})}\circ\theta
_{T_{\Z^d\setminus\{2e_1,3e_1\}}\circ\theta_4}$,
\item[(3)] $D\circ\theta_{T_{6e_1}} =\infty$,
\end{longlist}
then we have $D=\infty$ and $\tau_1\geq T_{\Z^d\setminus\{
2e_1,3e_1\}
}\circ\theta_4$. Now, we aim at estimating the three different events
and $T_{\Z^d\setminus\{2e_1,3e_1\}}\circ\theta_4$, which will allow
us to give a lower-bound in (\ref{zf}).

On $A\cap B$, we see, by Remark \ref{fakeUE}, that we have
%
\begin{equation}
\label{poiq} P^{\omega}[X_1=e_1,
X_2=2e_1, X_3=e_1,
X_4=2e_1]\geq\kappa_0^4
\end{equation}
and moreover on $A\cap B$
%
\begin{equation}
\label{poiw} P^{\omega}_{2e_1}[T_{6e_1}\circ
\theta_{T_{\Z^d\setminus\{
2e_1,3e_1\}
}\circ\theta_4}\leq T_{\partial(\mathcal{A}\setminus\{0\})}\circ\theta
_{T_{\Z^d\setminus\{2e_1,3e_1\}}\circ\theta_4}] \geq
\kappa_0^7,
\end{equation}
which follows from Remark \ref{fakeUE} and the fact that, on $A\cap
B$, for any neighbor of $2e_1$ or $3e_1$, there exists an open
nearest-neighbor path of length at most $7$ in $\mathcal{A}\setminus\{
0\}$ to $6e_1$.

Using Remark \ref{fakeUE} again, we may see that on $A\cap B$, we
have $P_{2e_1}^{\omega}[X_1\neq3e_1]\leq C/c_*([2e_1,3e_1])$ and
$P_{3e_1}[X_1\neq2e_1]\leq C/c_*([2e_1,3e_1]) $. This implies that
%
\begin{equation}
\label{za} P^{\omega}_{2e_1}[T_{\Z^d\setminus\{2e_1,3e_1\}} \geq n] \geq
\bigl(1- C/c_*\bigl([2e_1,3e_1]\bigr)
\bigr)^n,
\end{equation}
so
%
\begin{equation}
\label{exittrap} E^{\omega}_{2e_1}[T_{\Z^d\setminus\{2e_1,3e_1\}}
]\geq c c_*\bigl([2e_1,3e_1]\bigr).
\end{equation}

Using the series of events below (\ref{zf}) and Markov's property (at
times $4$, $T_{\Z^d\setminus\{2e_1,3e_1\}}\circ\theta_4$ and
$T_{6e_1}$) along with (\ref{exittrap}), (\ref{poiq}) and (\ref{poiw})
we may see
%
\begin{eqnarray}
\label{tyu} &&\ES^a\bigl[{\mathbf1} {\{A,B\}} {\mathbf1} {\{D=\infty
\}}\tau_1\bigr]
\nonumber\\
&&\qquad\geq c {\mathbf E}^a\bigl[{\mathbf1} {\{A,B
\}}E^{\omega}_{2e_1}[T_{\Z
^d\setminus\{
2e_1,3e_1\}}]P_{6e_1}^{\omega}[D=
\infty] \bigr]
\\
&&\qquad\geq  c {\mathbf E}^a\bigl[{\mathbf1} {\{A,B\}} c_*
\bigl([2e_1,3e_1]\bigr)P_{6e_1}^{\omega}[D=
\infty] \bigr],\nonumber
\end{eqnarray}
which is a consequence of the estimates of the events defined
after (\ref{zf}).

We may now notice that ${\mathbf1}{\{A\}}$, $c_*([2e_1,3e_1])$ and
${\mathbf1}{\{B\}}P_{6e_1}^{\omega}[D=\infty]$ are ${\mathbf
P}^a$-independent, so that
\[
\ES^a\bigl[{\mathbf1} {\{A,B\}} {\mathbf1} {\{D=\infty\}}
\tau_1\bigr]\geq{\mathbf P}^a[A] {\mathbf E}^a
\bigl[c_*\bigl([2e_1,3e_1]\bigr)\bigr]{\mathbf
E}^a\bigl[{\mathbf1} {\{B\}}P_{6e_1}^{\omega
}[D=
\infty]\bigr].
\]

We have\vspace*{1pt} $\min_{a\in[1/K,K]^{\mathcal{E}}} {\mathbf P}^a[A] \geq c>0$,
${\mathbf E}^a[c_*([2e_1,3e_1])]={\mathbf E}[c_*([2e_1,3e_1])]=\infty$, and by
translation invariance
\[
{\mathbf E}^a\bigl[{\mathbf1} {\{B\}}P_{6e_1}^{\omega}[D=
\infty]\bigr]={\mathbf E}\bigl[{\mathbf1} {\{0\mbox{ is good}\}}P^{\omega}[D=
\infty]\bigr]>0
\]
by Lemma \ref{posescape} and the fact that ${\mathbf P}[0\mbox{ is
good}]>0$. This means that
\[
{\mathbf E}^a\bigl[{\mathbf1} {\{A,B\}} c_*\bigl([2e_1,3e_1]
\bigr)P_{6e_1}^{\omega
}[D=\infty] \bigr]=\infty.
\]

Hence, by (\ref{zf}) and (\ref{tyu}), we have
\[
\min_{a\in[1/K,K]^{\mathcal{E}}} \ES^a[\tau_1 \mid D=
\infty]=\infty.
\]
\upqed\end{pf}

\begin{remark}
\label{tauinf}
Using a reasoning similar to the previous proof but using a normal
edge surrounded by edges with small conductances (see Figure \ref{fig1}), we may
show that if $P_*[1/c_*\geq x]\geq c \ln(x)^{-\epsilon}$ for any
$\epsilon>0$, then
\[
\ES_0\bigl[\tau_1 \ln(\tau_1)^{\epsilon}
\mid D=\infty\bigr]=\infty
\]
for any $\epsilon>0$. Essentially, without any assumption on the tail
of $c_*$ at 0, we cannot expect any stronger integrability of
regeneration times than the first moment being finite.
\end{remark}

\begin{remark}
\label{tcl}
Using a reasoning similar to the previous proof and a modified version
using a normal edge surrounded by edges with $4d-2$ small conductances
(see Figure \ref{fig1}), we may show that
\[
\mbox{if $P_*[1/c_*\geq x]\geq x^{-1/(4d-2)+\epsilon}$ or $P_*[c_*\geq
x] \geq
x^{-2+\epsilon}$}\qquad \mbox{for some $\epsilon>0$,}
\]
then
\[
\ES_0\bigl[\tau_1^2 \mid D=\infty\bigr]=
\infty,
\]
which is a strong indication that the annealed central limit theorem
does not hold.
\end{remark}

The previous lemma implies:
%
\begin{proposition}
\label{LB}
If $E_*[c_*]=\infty$, then $\lim X_n/n=\vec0$ $\PR$-a.s.
\end{proposition}
\begin{pf}
By Lemma \ref{LBa}, we see that
\[
\ES^{\Pi}[\tau_1]=\infty,
\]
%
and then (\ref{ozx}) implies that
%
\[
\lim_{n\to\infty} \frac{X_n}n=\vec0.
\]
\upqed\end{pf}

\subsection{\texorpdfstring{Lower-bound on the polynomial order of $\Delta_n$}
{Lower-bound on the polynomial order of Delta n}}

The trap described in Figure \ref{fig5} is the most efficient when conductances
have heavy tails. Actually, before~$\Delta_n$, the walk will typically
have encountered such a trap associated with a high conductances of
roughly $n^{1/\gamma}$ which, just on it's own, will be responsible for
a sharp lower bound on $\Delta_n$.\vadjust{\goodbreak}
%
\begin{lemma}
\label{LB1}
If $-\lim\frac{\ln P_*[c_*>n]}{\ln n}=\gamma<1$, we have
\[
\liminf\frac{\ln\Delta_n}{\ln n} \geq1/\gamma,\qquad \PR\mbox{-a.s.}
\]
\end{lemma}
\begin{pf}
Using Theorem \ref{tailtau2},
\[
\max_{a\in[1/K,K]^{\mathcal{E}}} \ES^a[X_{\tau_1^{(K)}}\cdot\vec{\ell
}]<\infty,
\]
which implies
%
\begin{equation}
\label{zz} \ES^{\Pi}[X_{\tau_1}\cdot\vec{\ell}]<\infty.
\end{equation}

Because of Theorem \ref{thnu2} and (\ref{zz}), we may Birkhoff's
ergodic theorem (page~341 in~\cite{durrett1}) to see that
\[
\frac{(X_{\tau_n}-X_{\tau_1})\cdot\vec{\ell}}{n}\to_{n\to\infty
} \ES^{\Pi}[X_{\tau_1}
\cdot\vec{\ell}]<\infty.
\]

This implies\vspace*{1pt} that with probability going to 1, we have $X_{\tau
_n}\cdot
\vec{\ell}< cn$ for $c>\ES^{\Pi}[X_{\tau_1}\cdot\vec{\ell}]$. Hence,
for $c<1/\ES^{\Pi}[X_{\tau_1}\cdot\vec{\ell}]$, we have $X_{\tau
_{cn}}\cdot\vec{\ell}< n$ with probability going to 1, which means
that for a small $c>0$,
\[
\PR[\tau_{cn} \leq\Delta_n] \to1.
\]

Using (\ref{za}) and a reasoning similar to the proof of Lemma \ref{LBa},
\begin{eqnarray*}
\min_{a\in[1/K,K]^{\mathcal{E}}} \PR_0^a[
\tau_1\geq n\mid D=\infty]&\geq& c {\mathbf E}\bigl[\bigl(1- C/c_*
\bigl([2e_1,3e_1]\bigr)\bigr)^n\bigr]
\\
&\geq& cP_*\bigl[c_* \geq n^{1+2\epsilon_1}\bigr]\geq cn^{-(\gamma
(1+\epsilon_1))}
\end{eqnarray*}
for any $\epsilon_1>0$.

Notice that if $\tau_{cn}=\sum_{i=1}^{cn-1}( \tau_{i+1}-\tau_i)
\leq
n^{1/{\gamma} -\epsilon}$, then $
\tau_{i+1}-\tau_i \leq n^{1/{\gamma} -\epsilon}$ for any
$i\leq
cn-1$. This and the previous two equations imply that for any $\epsilon
>0$ we have
\begin{eqnarray*}
\PR\bigl[\Delta_n\leq n^{1/{\gamma} -\epsilon}\bigr]&\leq&\PR
\Biggl[\sum
_{i=1}^{cn-1} (\tau_{i+1}-
\tau_i) \leq n^{1/{\gamma} -\epsilon
} \Biggr] +o(1)
\\
&\leq& o(1)+ \prod_{i=2}^{cn-1} \max
_{a\in[1/K,K]^{\mathcal{E}}} \PR_0^a\bigl[
\tau_1 <cn^{1/{\gamma}-\epsilon}\mid D=\infty\bigr]
\\
&\leq& o(1) +\bigl(1-cn^{-(1/{\gamma}-\epsilon)\gamma(1+\epsilon
_1)}\bigr)^{cn-2}
\end{eqnarray*}
by Theorem \ref{thindep} and Lemma \ref{posescape}, for any
$\epsilon_1>0$.

Then taking $\epsilon_1>0$ small enough such that $(\frac1 {\gamma
}-\epsilon)\gamma(1+\epsilon_1)<1$, we see that
\[
\PR\bigl[\Delta_n\leq n^{1/{\gamma} -\epsilon}\bigr]\to0.
\]

This being true for all $\epsilon>0$, we have the lemma.
\end{pf}

\subsection{\texorpdfstring{Upper-bound on the polynomial order of $\Delta_n$}
{Upper-bound on the polynomial order of Delta n}}

Given a realization of an environment $\omega$ and a walk, $\widetilde
{\mathrm{GOOD}}{}^i_K$ [resp., $\widetilde{\mathrm{BAD}}{}^i(x)$] denotes the
set of good vertices (resp., the bad cluster of $x$) in the
configuration $\omega^i$ where all edges of $\mathcal{L}^{X_{\tau_i}}$
[defined at (\ref{defleft})] are considered to be normal, and all
other have the same state as in $\omega$.

Let us introduce,\vspace*{-2pt} for $i\geq0$, $\tilde{\tau}_i^{(K)}:=\{j\geq
T_{\widetilde{\mathrm{GOOD}}{}^i_K} \circ\tau_i$ with $j\leq
\tau
_{i+1}$ or $X_j\in\widetilde{\mathrm{BAD}}{}^i(X_{\tau_{i+1}})\}$.
Let us explain rapidly why we introduce the definition $\tilde{\tau
}_i$. These random variables will approximate $\tau_{i+1}-\tau_i$, but
they also have the following properties:
\begin{longlist}[(2)]
\item[(1)] we can estimate the time of an excursion in a trap, but not
the time to exit a trap. With the previous definition of $\tilde{\tau
}_i$, we will only need to consider excursions in traps.
\item[(2)] A simple inductance can be used to see that we can
upper-bound $\tau_n$ by $(T_{\mathrm{GOOD}}\circ\tau_1+\tau_1)
+\tilde
{\tau}_1+\cdot+\tilde{\tau}_n$. This is ultimately what allows us to
upper-bound~$\Delta_n$.
\end{longlist}

In exchange for these advantages, we lose any independence properties
between all $\tilde{\tau}_i$. This will not be a major issue, as we
will see Proposition \ref{LB2}. Even though this may seem surprising at
first sight, we recall that we are working with $\gamma<1$, which means
that we are working in the heavy tailed regime. Hence, the limiting
behavior is determined by what happens in only one trap and thus, at
least heuristically, correlations should not be important.

On the bright side, the random variables $\tilde{\tau}_i$ are
measurable with respect to $\sigma\{X_{n+\tau_i}-X_{\tau_i}\dvtx n\geq
0\}$
and $\sigma\{c_*(e),e\in\mathcal{R}^{X_{\tau_i}}\}$, which is well
suited for an application of Theorem \ref{thindep}. This theorem, will
allow us, loosely speaking, to say that $\tilde{\tau}_i$ (for $i\geq
1$) is distributed like $\tilde{\tau}_0$ under the law $\PR
^{a_{X_{\tau
_i}}}$. This is why we seek to prove the following.

\begin{lemma}\label{decaytau}
Assume that $-\lim\frac{\ln P_*[c_*>n]}{\ln n}=\gamma<1$. For any
$\epsilon>0$, there exists $K_0$ such that, for any $K\geq K_0$,
\[
\max_{a\in[1/K,K]^{\mathcal{E}}} \PR^a\bigl[\tilde{
\tau}_0^{(K)}>n\mid D=\infty\bigr]\leq C(K)
n^{-(\gamma-\epsilon)}.
\]
\end{lemma}

We will follow the ideas of Lemma \ref{lkjh} to prove it.

\begin{pf}
To simplify the notations, we will do the proof for $\PR$. In a similar
fashion, we could do it for any $\PR^a$ for $a\in[1/K,K]^{\mathcal
{E}}$. We may also notice that on the event $\{D=\infty\}$, we have that:
\begin{longlist}[(2)]
\item[(1)] $T_{\widetilde{\mathrm{GOOD}}{}^0_K} =T_{\mathrm{GOOD}_K}$,
\item[(2)] $\widetilde{\mathrm{BAD}}{}^i(X_{\tau_{1}}) \subset\mathrm
{BAD}(X_{\tau_{1}})$,
\end{longlist}
which means that $\tilde{\tau}_0$ is smaller than the random variable
$\overline{\tau}_0:=\{j\geq T_{\mathrm{GOOD}_K}$ with $j\leq
\tau
_{1}$ or $X_j\in\mathrm{BAD}(X_{\tau_{1}})\}$. Hence, to prove
our theorem it will be enough to prove that
\[
\max_{a\in[1/K,K]^{\mathcal{E}}} \PR^a\bigl[\overline{\tau
}_0^{(K)}>n\bigr]\leq C(K) n^{-(\gamma-\epsilon)}.
\]

In this proof, we will point out the $K$ dependence of constants,
since the proof requires us to be careful with this dependence.

Fix $\epsilon>0$.
We have
\[
\overline{\tau}_0 \leq\sum_{x\in\mathrm{GOOD}(\omega)} {
\mathbf1} {\{T_x<\tau_1\}} + \sum
_{x\in\partial\mathrm{BAD}(\omega) } \sum_{i=1}^{\infty
}{
\mathbf1} {\{T_x<\tau_1\}} {\mathbf1} {
\{X_i=x\}} T^+_{\mathrm
{GOOD}(\omega)}\circ\theta_i.
\]

Recalling the definition of $\chi$ at the beginning of the proof of
Lemma \ref{hittrap}, we see that
\begin{eqnarray*}
&&E^{\omega}\bigl[{\mathbf1} {\bigl\{\chi\leq n^{\epsilon}\bigr\}}
\overline{\tau}_0\bigr]
\\
&&\qquad\leq  \sum_{x\in B(n^{\epsilon}, n^{C\epsilon})} \Biggl[{\mathbf1}
{\bigl\{x\in
\mathrm{GOOD}(\omega)\bigr\}}\\
&&\qquad\quad\hspace*{49.9pt}{}+ {\mathbf1} {\bigl\{x\in\partial\mathrm
{BAD}(\omega)
\bigr\}} E^{\omega
} \Biggl[ \sum_{i=1}^{\infty}{
\mathbf1} {\{X_i=x\}} T^+_{\mathrm{GOOD}(\omega
)}\circ\theta_i
\Biggr] \Biggr].
\end{eqnarray*}

Using Markov's property and Lemma \ref{backbone} [just as in (\ref
{cantstandthis})], we obtain
\begin{eqnarray*}
&& E^{\omega}\bigl[{\mathbf1} {\bigl\{\chi\leq n^{\epsilon}\bigr\}}
\overline{\tau}_0\bigr]
\\
&&\qquad\leq C(K) \sum_{x\in B(n^{\epsilon}, n^{C\epsilon})} \bigl[{\mathbf1}
{\bigl\{ x\in
\mathrm{GOOD}(\omega)\bigr\}}\\
&&\qquad\quad\hspace*{77.4pt}{}+ {\mathbf1} {\bigl\{x\in\partial\mathrm
{BAD}(\omega)
\bigr\}} E^{\omega
}_x\bigl[T^+_{\mathrm{GOOD}(\omega)}\bigr]\bigr],
\end{eqnarray*}
and now, since $\gamma<1$, we have
\begin{eqnarray*}
&& E^{\omega}\bigl[{\mathbf1} {\bigl\{\chi\leq n^{\epsilon}\bigr\}}
\overline{\tau}{}^{\gamma
-\epsilon}_0\bigr] \\
&&\qquad\leq E^{\omega}\bigl[{
\mathbf1} {\bigl\{\chi\leq n^{\epsilon}\bigr\}} \overline{\tau
}_0\bigr]^{\gamma-\epsilon}
\\
&&\qquad\leq C(K) \biggl[\sum_{x\in B(n^{\epsilon}, n^{C\epsilon})} \bigl
[{\mathbf1} {\bigl
\{x\in\mathrm{GOOD}(\omega)\bigr\}}\\
&&\qquad\quad\hspace*{80pt}{}+ {\mathbf1} {\bigl\{x\in\partial
\mathrm{BAD}(
\omega)\bigr\}} E^{\omega
}_x\bigl[T^+_{\mathrm{GOOD}(\omega)}\bigr]
\bigr] \biggr]^{\gamma-\epsilon}.
\end{eqnarray*}\eject

We may now apply Lemma \ref{timetrap},
\begin{eqnarray*}
&& E^{\omega}\bigl[{\mathbf1} {\bigl\{\chi\leq n^{\epsilon}\bigr\}}
\overline{\tau}{}^{\gamma
-\epsilon}_0\bigr]
\\
&&\qquad\leq  C(K) \biggl[\sum_{x\in B(n^{\epsilon}, n^{C\epsilon})} {\mathbf1}
{\bigl\{x
\in\mathrm{GOOD}(\omega)\bigr\}}\\
&&\qquad\quad\hspace*{78pt}{}+ {\mathbf1} {\bigl\{x\in\partial\mathrm{BAD}(
\omega)\bigr\}} E^{\omega
}_x\bigl[T^+_{\mathrm{GOOD}(\omega)}\bigr]
\biggr]^{\gamma-\epsilon}
\\
&&\qquad\leq  C(K) \biggl[\sum_{x\in B(n^{\epsilon}, n^{C\epsilon})} {\mathbf1}
{\bigl\{x
\in\mathrm{GOOD}(\omega)\bigr\}}
\\
&&\hspace*{78pt}\qquad\quad{} +\sum_{x\in B(n^{\epsilon},n^{C\epsilon})} {\mathbf1} {\bigl\{x\in
\partial\mathrm
{BAD}(\omega)\bigr\}} \\
&&\qquad\quad\hspace*{90.5pt}\hspace*{46.8pt}{}\times\exp\bigl(3\lambda\bigl\llvert\partial
\mathrm{BAD}^{\mathrm{s}}_x(\omega)\bigr\rrvert\bigr) \\
&&\qquad\quad\hspace*{90.5pt}\hspace*{46.8pt}{}\times \biggl(1+\sum
_{e\in
E(\mathrm{BAD}^{\mathrm
{s}}_x(\omega))} c_*(e) \biggr) \biggr]^{\gamma-\epsilon}
\\
&&\qquad\leq C(K)n^{C\epsilon} \max_{x\in B(n^{\epsilon},n^{C\epsilon})}
{\mathbf1} {\bigl\{x
\in\partial\mathrm{BAD}(\omega)\bigr\}}\bigl\llvert E\bigl(\mathrm
{BAD}^{\mathrm{s}}_x(\omega)\bigr)\bigr\rrvert^{\gamma}\\
&&\qquad\quad\hspace*{90.5pt}{}\times
\exp\bigl(3\gamma\lambda\bigl\llvert\partial\mathrm{BAD}^{\mathrm{s}}_x(
\omega)\bigr\rrvert\bigr)
\\
&&\hspace*{90.5pt}\qquad\quad{} \times\Bigl(1+\max_{e\in E(\mathrm{BAD}^{\mathrm{s}}_x(\omega))}
c_*(e)^{\gamma-\epsilon} \Bigr)
\\
&&\qquad\leq C(K) n^{C\epsilon} \sum_{x\in B(n^{\epsilon},n^{C\epsilon})}
{\mathbf1} {
\bigl\{x \in\partial\mathrm{BAD}(\omega)\bigr\}} \exp\bigl(4\lambda\bigl
\llvert
\partial\mathrm{BAD}^{\mathrm{s}}_x(\omega)\bigr\rrvert\bigr)
\\
&&\qquad\quad{} \times\biggl(1+ \sum_{e\in E(\mathrm{BAD}^{\mathrm{s}}_x(\omega))}
c_*(e)^{\gamma-\epsilon}
\biggr),
\end{eqnarray*}
where we used that $\llvert E(\mathrm{BAD}^{\mathrm{s}}_x(\omega
))\rrvert\leq C
\llvert\partial\mathrm{BAD}^{\mathrm{s}}_x(\omega)\rrvert^d$.

Now, averaging over the environment, we get
\begin{eqnarray*}
&& \ES\bigl[{\mathbf1} {\bigl\{\chi\leq n^{\epsilon}\bigr\}} \overline
{\tau
}{}^{\gamma
-\epsilon}_0 \bigr]
\\
&&\qquad\leq  C(K)n^{C\epsilon} {\mathbf E} \biggl[{\mathbf1} {\bigl\{0\in\mathrm{GOOD}(
\omega)\bigr\}} \\
&&\qquad\quad\hspace*{54.1pt}{}\times\sum_{x\in B(n^{\epsilon}, n^{C\epsilon})} {\mathbf1}
{\bigl\{x\in
\partial\mathrm{BAD}(\omega)\bigr\}} \exp\bigl(4\lambda\bigl\llvert
\partial\mathrm
{BAD}^{\mathrm{s}}_x(\omega)\bigr\rrvert\bigr)
\\
&&\hspace*{146.2pt}\qquad\quad{} \times\biggl(1+\sum_{e\in E(\mathrm{BAD}^{\mathrm{s}}_x(\omega))}
c_*(e)^{\gamma-\epsilon}
\biggr) \biggr],
\end{eqnarray*}
so using a reasoning similar to (\ref{step3}) (based on Lemma \ref
{condtrap}) yields
\begin{eqnarray*}
&& \ES\bigl[{\mathbf1} {\bigl\{\chi\leq n^{\epsilon}\bigr\}} \overline
{\tau
}{}^{\gamma
-\epsilon}_0\bigr] /\bigl(C(K) n^{C\epsilon}\bigr)
\\
&&\qquad\leq  \sum_{x\in B(n^{\epsilon}, n^{C\epsilon})}{\mathbf E} \biggl[\exp
\bigl(4\lambda
\bigl\llvert\partial\mathrm{BAD}^{\mathrm{s}}_x(\omega)\bigr\rrvert
\bigr) \biggl(1+ \sum_{e\in E(\mathrm{BAD}^{\mathrm{s}}_x(\omega))}
c_*(e)^{\gamma-\epsilon
}
\biggr) \biggr]
\\
&&\qquad\leq  C(K)\sum_{x\in B(n^{\epsilon}, n^{C\epsilon})}{\mathbf E}\bigl[\bigl
\llvert
\partial\mathrm{BAD}^{\mathrm{s}}_x(\omega)\bigr\rrvert^d
\exp\bigl(4\lambda\bigl\llvert\partial\mathrm{BAD}^{\mathrm{s}}_x(
\omega)\bigr\rrvert\bigr)\bigr],
\end{eqnarray*}
where we used that $\llvert E(\mathrm{BAD}^{\mathrm{s}}_x(\omega
))\rrvert\leq C
\llvert\partial\mathrm{BAD}^{\mathrm{s}}_x(\omega)\rrvert^d$.

We may see that Lemma \ref{tgb} implies that for $K$ large enough,
\[
{\mathbf E}\bigl[\bigl\llvert\partial\mathrm{BAD}^{\mathrm{s}}_x(
\omega)\bigr\rrvert^d\exp\bigl(4\lambda\bigl\llvert\partial
\mathrm{BAD}^{\mathrm{s}}_x(\omega)\bigr\rrvert\bigr)\bigr]<C(K)<
\infty,
\]
which means that for any $\epsilon>0$,
\[
\ES\bigl[{\mathbf1} {\bigl\{\chi\leq n^{\epsilon}\bigr\}} \overline{\tau
}{}^{\gamma-\epsilon
} _0\bigr] \leq C(K) n^{C\epsilon}.
\]

From this, using Chebyshev's inequality, we get
\begin{eqnarray*}
\PR\bigl[\chi\leq n^{\epsilon}, \overline{\tau}_0> n\bigr] &=&
\PR\bigl[{\mathbf1} {\bigl\{\chi\leq n^{\epsilon}\bigr\}} \overline{
\tau}_0> n\bigr]
\\
&\leq& n^{-(\gamma-\epsilon
)} \ES\bigl[{\mathbf1} {\bigl\{\chi\leq n^{\epsilon}
\bigr\}} \overline{\tau}{}^{\gamma
-\epsilon}_0\bigr]
\\
& \leq & C(K) n^{-(\gamma-\epsilon)}n^{C_1\epsilon}.
\end{eqnarray*}

For any $\epsilon_1>0$, we may apply the previous equality for a small
$\epsilon$ (which depends only on $\gamma$ and $C_1$ but not on $K$)
to obtain
\[
\PR\bigl[\chi\leq n^{\epsilon}, \overline{\tau}_0> n\bigr]
\leq C(K)n^{-(\gamma
-\epsilon_1)}.
\]

Now, using (\ref{expostep15}), we can choose $K$ large enough such that
\[
\PR\bigl[\chi> n^{\epsilon}\bigr]\leq Cn^{-1},
\]
and the previous two equations imply that for any $\epsilon_1>0$ there
exists $K$ large enough, we obtain
\[
\PR\bigl[\overline{\tau}{}^{(K)}_0>n\bigr]\leq\PR\bigl[\chi>
n^{\epsilon}\bigr] + \PR\bigl[\chi\leq n^{\epsilon}, \overline{
\tau}{}^{(K)}_0> n\bigr] \leq C(K)n^{-(\gamma
-\epsilon_1)},
\]
which proves the lemma.
\end{pf}

The random variables $\tilde{\tau}_i$ verify
%
\begin{equation}
\label{jonstewart} \Delta_n\leq\tau_n
\leq(T_{\mathrm{GOOD}}\circ\tau_1+\tau_1)+\sum
_{i=1}^{n} \tilde{\tau}_i,
\end{equation}
since, by (\ref{rightdir}), we necessarily have $X_{\tau
_n^{(K)}}\cdot
e_1\geq n$.

This means that the previous lemma will allow us to give a sharp upper
bound on $\Delta_n$.
%
\begin{proposition}\label{LB2}
If $\lim\frac{\ln P_*[c_*>n]}{\ln n}=-\gamma$ with $\gamma<1$, then
\[
\limsup\frac{\ln\Delta_n}{\ln n}\leq\frac1{\gamma},\qquad \PR\mbox{-a.s.}
\]
\end{proposition}
\begin{pf}Fix $M\geq1$. We know, by Lemma \ref{decaytau} and
Theorem \ref{thindep}, that there exists $K$ large enough such that
for $i\leq M+1$,
\begin{eqnarray*}
&&
\ES\bigl[\mathrm{card}\bigl\{1\leq j\leq n, \tilde{\tau}_j^{(K)}
\geq n^{i/(M\gamma
)}\bigr\}\bigr] \\
&&\qquad = \ES\biggl[\sum
_{{1\leq j\leq n}} 1\bigl\{ \tilde{\tau}_j^{(K)}
\geq n^{i/(M\gamma)}\bigr\} \biggr]
\\
&&\qquad= \sum_{{1\leq j\leq n}} \ES\bigl[ \PR^{a_{X_{\tau_{j}}}}\bigl[
\tilde{\tau}_1^{(K)}\geq n^{i/(M\gamma)}\mid D=\infty
\bigr] \bigr]
\\
&&\qquad\leq n\max_{a\in[1/K,K]^{\mathcal{E}}} \PR^a\bigl[\tilde{
\tau}_0\geq n^{i/(M\gamma)}\mid D=\infty\bigr]
\\
&&\qquad\leq Cn n^{-(i/M)(1-1/M)}=Cn^{(1-i/M)+2/M},
\end{eqnarray*}
where we used Lemma \ref{posescape}. Hence by Markov's inequality for
any $L$,
\begin{eqnarray*}
&&\PR\biggl[\mathrm{card}\bigl\{1\leq j \leq n, \tilde{\tau}_j^{(K)}
\geq n^{i/(M\gamma)}\bigr\}\geq\frac1 {2M} n^{1/\gamma
+L/M-(i+1)/(M\gamma
)} \biggr]
\\
&&\qquad\leq  C(M) n^{(1-1/\gamma)(1-i/M)+(1/\gamma+2 -L)/M}.
\end{eqnarray*}

Fix $L \geq1/\gamma+3$ (which does not depend on $M$). We denote the event
\begin{eqnarray*}
&&B(n,i, M)
\\
&&\qquad= \biggl\{\mathrm{card} \bigl\{1\leq j \leq n, \tilde{\tau}_j^{(K)}
\in\bigl(n^{i/(M\gamma)},n^{(i+1)/(M\gamma)}\bigr]\bigr\}\\
&&\qquad\hspace*{99.4pt}\geq\frac1 {2M}
n^{1/\gamma
+L/M-(i+1)/(M\gamma)}
\biggr\},
\end{eqnarray*}
and we get that for any fixed $M$ and $i\leq M$
%
\begin{equation}
\label{exponentb} \PR\bigl[B(n,i,M)\bigr]\leq C
n^{-1/M}=o(1),
\end{equation}
since $\gamma< 1$.

In the same way, by Lemma \ref{decaytau} and Theorem \ref{thindep},
we get that
\[
B(n,M+1,M)=\bigl\{\mathrm{card}\bigl\{1\leq j \leq n, \tilde{\tau}_j^{(K)}
\geq n^{(M+1)/(M\gamma)}\bigr\}\geq1\bigr\},
\]
verifies
\[
\PR\bigl[B(n,M+1,M)\bigr]\leq n^{-\epsilon}=o(1).
\]

This means that, denoting $B(n,M)=\bigcup_{j=0}^{M+1} B(n,i,M)$, we have
\[
\PR\bigl[B(n,M)\bigr]=o(1).
\]

Now, on $B(n,M)^c$, we can give an upper bound for $\Delta_n$ by
using (\ref{jonstewart})
\begin{eqnarray*}
\Delta_n &\leq& (T_{\mathrm{GOOD}}\circ\tau_1+
\tau_1)+\sum_{i=0}^{M}
n^{(i+1)/(M\gamma)} \biggl(\frac1 {2M} n^{1/\gamma
+L/M-(i+1)/(M\gamma
)} \biggr)
\\
&=& (T_{\mathrm{GOOD}}\circ\tau_1+\tau_1)+
\frac{M+1}{2M}n^{1/\gamma+L/M}.
\end{eqnarray*}

It follows that $\Delta_n \leq(T_{\mathrm{GOOD}}\circ\tau_1+\tau_1)+
n^{1/\gamma+L/M}$ on $B(n,M)^c$, and hence, since $(T_{\mathrm
{GOOD}}\circ\tau_1+\tau_1)<\infty$, we have
\[
\mbox{on $B(n,M)^c$}\qquad \frac{\ln\Delta_n}{\ln n} \leq1/\gamma+L/M.
\]

Hence we have proved that for any $M\geq1$, by (\ref{exponentb})
\[
\limsup\frac{\ln\Delta_n}{\ln n} \leq1/\gamma+L/M,\qquad \PR\mbox{-a.s.},
\]
and letting $M$ go to infinity, we get the result (we recall that $L$
does not depend on~$M$).
\end{pf}

\subsection{The polynomial order of the distance of the random walk
from the origin}

By Proposition \ref{LB2} and Lemma \ref{LB1} we see that if $-\lim
\frac
{\ln P_*[c_*>n]}{\ln n}=\gamma<1$, we have
\[
\lim\frac{\ln\Delta_n}{\ln n} = 1/\gamma, \qquad\PR\mbox{-a.s.},
\]
which implies, using a classical inversion argument similar to the one
used in the proof of Theorem 1.3 in Section 5 of \cite{FH}, we see that
\[
\lim\frac{\ln X_n \cdot\vec{\ell}}{\ln n} =\gamma, \qquad\PR\mbox{-a.s.}
\]



\section*{Acknowledgments}

I am very grateful to the referee for various remarks that greatly
improved the quality of the paper.



\printaddresses

\end{document}